\documentclass[12pt]{amsart}
\usepackage{SSdefn}
\usepackage[shortlabels]{enumitem}

\theoremstyle{plain} % just in case the style had changed
\newcommand{\thistheoremname}{}
\newtheorem*{genericthm}{\thistheoremname}
\newenvironment{namedthm}[1]
  {\renewcommand{\thistheoremname}{#1}%
   \begin{genericthm}}
  {\end{genericthm}}

\newcommand{\defn}[1]{\textit{#1}}
\DeclareMathOperator{\magn}{mag}
\DeclareMathOperator{\tdim}{tdim}
\DeclareMathOperator{\Sh}{Sh}
\DeclareMathOperator{\sh}{sh}

\newcommand{\orb}{\mathrm{orb}}
\newcommand{\Map}{\mathrm{Map}}
\newcommand{\uMap}{\ul{\smash{\mathrm{Map}}}}
\newcommand{\umu}{\ul{\smash{\mu}}}
\newcommand{\unu}{\ul{\smash{\nu}}}
\newcommand{\ulambda}{\ul{\smash{\lambda}}}
\newcommand{\urho}{\ul{\smash{\rho}}}
\newcommand{\type}{\mathrm{type}}
\renewcommand{\int}{\operatorname{int}}
\newcommand{\lpp}{(\!(}
\newcommand{\rpp}{)\!)}
\newcommand{\prin}{\mathrm{prin}}

\newcommand{\doi}[1]{\href{https://doi.org/#1}{{\tt DOI:#1}}}

\title{The geometry of polynomial representations}

\author{Arthur Bik}
\address{University of Bern, Switzerland, and MPI for Mathematics in the Sciences, Germany}
\email{\href{mailto:arthur.bik@mis.mpg.de}{arthur.bik@mis.mpg.de}}
\urladdr{\url{http://arthurbik.nl}}

\author{Jan Draisma}
\address{University of Bern, Switzerland, and Eindhoven University of Technology, The Netherlands}
\email{\href{mailto:jan.draisma@math.unibe.ch}{jan.draisma@math.unibe.ch}}
\urladdr{\url{https://mathsites.unibe.ch/jdraisma/}}

\author{Rob H. Eggermont}
\address{Eindhoven University of Technology, The Netherlands}
\email{\href{mailto:r.h.eggermont@tue.nl}{r.h.eggermont@tue.nl}}
\urladdr{\url{https://www.tue.nl/en/research/researchers/rob-eggermont/}}

\author{Andrew Snowden}
\address{Department of Mathematics, University of Michigan, Ann Arbor, MI}
\email{\href{mailto:asnowden@umich.edu}{asnowden@umich.edu}}
\urladdr{\url{http://www-personal.umich.edu/~asnowden/}}

\begin{document}

\begin{abstract}
We define a \defn{$\GL$-variety} to be a (typically infinite dimensional) algebraic variety equipped with an action of the infinite general linear group under which the coordinate ring forms a polynomial representation. Such varieties have been used to study asymptotic properties of invariants like strength and tensor rank, and played a key role in two recent proofs of Stillman's conjecture. We initiate a systematic study of $\GL$-varieties, and establish a number of foundational results about them. For example, we prove a version of Chevalley's theorem on constructible sets in this setting.
\end{abstract}

\maketitle
\tableofcontents

\let\thefootnote\relax
\footnotetext{\hspace*{-14pt}
\begin{minipage}{.05\textwidth}
\includegraphics[width=\textwidth]{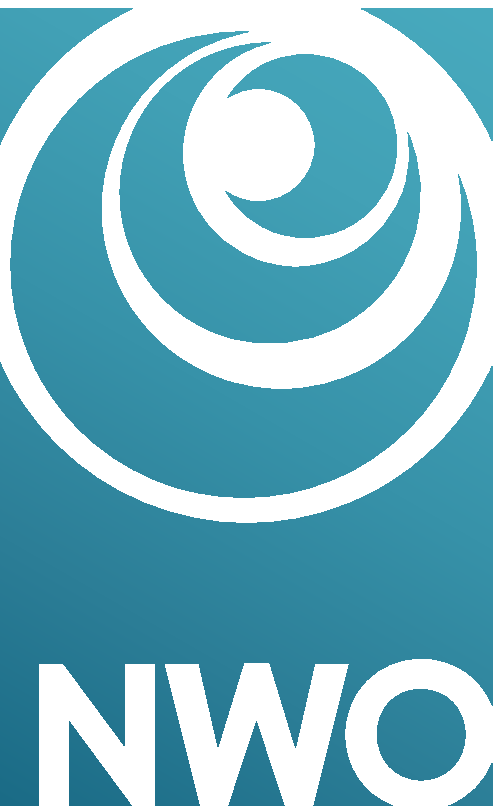}
\end{minipage}
\begin{minipage}{442pt}
AB was partially supported by JD's Vici Grant.
JD was partially supported by the NWO Vici grant
639.033.514 and SNSF project grant 200021\_191981.
RE was partially supported by the NWO Veni grant 016.Veni.192.113.
AS was supported by NSF grant DMS-1453893.
\end{minipage}}

\section{Introduction} \label{s:intro}

Across pure and applied mathematics, varieties with an action of the
general linear group feature prominently. These varieties are typically
embedded in a linear representation of the group, and the
group action on the variety gives rise to group representations on
defining ideals, higher-order syzygies, and other invariants. A thorough
understanding of both the commutative algebra and the
geometry of these varieties---their orbits, singularities,
etc.---is crucial for many purposes.

An important class of examples is secant varieties of the orbit of
highest-weight vectors in an irreducible representation.  This class
includes, for instance, varieties of tensors of bounded border rank;
a standard reference is \cite{landsberg}.  Representation-theoretic
techniques have long been used effectively in the quest for the defining
ideals of secant varieties. For instance, the GSS conjecture from
algebraic statistics \cite{gss}, which says that the ideal of the variety
of rank-two tensors is generated by $3 \times 3$-minors of flattenings,
is proved in \cite{raicu}, and generalized to Segre--Veronese varieties
and their tangential varieties in \cite{raicu2,or}. For tensors of higher
bounded rank, minors of flattenings are far from enough to cut out the
variety \cite{bb}; and in \cite{lo}, minors of flattenings are vastly
generalized to new classes of equations, via so-called Young flattenings.
At a more geometric level, \cite{lw} studies the singularities of secant
varieties, and \cite{bl} studies the orbits in the third secant variety.

These results concern varieties in representations of the general
linear group that can be defined without specifying the dimension
of the underlying vector space. In such a setting, one can study all
finite dimensional instances of the variety at once by passing to a
suitable infinite dimensional limit.  The limit lives in an ambient
space whose coordinate ring is a polynomial representation, now for
the infinite general linear group $\GL$. In this paper, we initiate
a systematic study of such $\GL$-varieties and prove a number of
fundamental results about them. Our results show that, although a
$\GL$-variety is typically infinite dimensional, the action of $\GL$
imposes so much structure on the variety that many aspects of it can be
captured in finite terms.  This finite description then implies uniform
descriptions of all finite dimensional instances of the variety that we
started out with.  To illustrate, we establish a unirationality theorem,
which says that any $\GL$-variety is the closure of a polynomial map
from a finite dimensional variety times an affine space acted upon by
$\GL$.  One might argue that this result is the {\em raison d'\^etre}
of the vast research area in applied mathematics that aims to decompose
tensors into simpler, structured tensors: any basis-independent variety
of tensors admits such a decomposition, given by said polynomial map!
Another result, a version of Chevalley's theorem, predicts, for instance,
that for any fixed $k$ and $d$, the set of sums of $k$ $d$-th powers
of linear forms has a uniform description as a constructible set,
independent of the number of variables; this is akin to the uniform
description of the orbits in the third secant variety
studied in \cite{lo}.

Infinite dimensional varieties acted upon by $\GL$ have featured
prominently in some recent work: \cite{draisma,es} prove noetherian
results in this setting, \cite{bde} establishes a structural result for
certain varieties of this form, and \cite{dll,genstillman,ess} give proofs
of Stillman's conjecture (and some related results in commutative algebra)
using such varieties.  Furthermore, $\GL$-varieties are closely connected
to research in which Gowers-norm techniques are used to establish
properties of high-rank (tuples of) polynomials \cite{kaz1,kaz2}.

\subsection{\texorpdfstring{$\GL$}{GL}-varieties} \label{ss:intro-glvar}

Fix a field $K$ of characteristic~0. Let $\GL=\bigcup_{n \ge 1}
\GL_n(K)$ be the infinite general linear group, and let
$\bV=\bigcup_{n \ge 1} K^n$ be its standard representation. For a
partition $\lambda$, we let $\bV_{\lambda}=\bS_{\lambda}(\bV)$ be the
irreducible polynomial representation of $\GL$ with highest weight
$\lambda$; here $\bS_{\lambda}$ denotes a Schur functor. For example,
$\bV_{(d)}$ is the $d$th symmetric power of $\bV$. We let
$\bA^{\lambda}$ be the spectrum of the polynomial ring
$\Sym(\bV_{\lambda})$; this is an affine scheme equipped with an
action of the group $\GL$. For a tuple of partitions
$\ulambda=[\lambda_1, \ldots, \lambda_r]$, we let
$\bA^{\ulambda}=\bA^{\lambda_1} \times \cdots \times \bA^{\lambda_r}$.
The role of the spaces $\bA^{\ulambda}$ in our theory is analogous to
the role of the usual affine spaces $\bA^n$ in classical algebraic
geometry.

\begin{namedthm}{Main definition}
An \defn{affine $\GL$-variety} is a closed $\GL$-stable subset of $\bA^{\ulambda}$, for some $\ulambda$.
\end{namedthm}

This paper is a study of affine $\GL$-varieties: our goal is
to understand their structure and establish analogs of
classical theorems of algebraic geometry for them. In the
process, it will sometimes be convenient to work with the larger
class of quasi-affine varieties, which are $\GL$-stable
open subsets of affine $\GL$-varieties.

Before proceeding, we give some examples of $\GL$-varieties that are typical of the kinds appearing in applications.

\begin{example} \label{ex:rank-var}
Let $V$ be a complex vector space. Recall that an element of $V^{\otimes d}$ has \defn{(tensor) rank} $\le r$ if it can be written as a sum of $r$ pure tensors. Tensor rank has been widely studied; for a general introduction, see \cite{landsberg}. The rank $\le r$ locus in $V^{\otimes d}$ is typically not Zariski closed (see, e.g., \cite[\S 2.4.5]{landsberg}). A point in the closure is said to have \defn{border rank} $\le r$.

Now, let $X=(\bV^{\otimes d})^*$; this space has the form $\bA^{\ulambda}$ for an appropriate choice of $\ulambda$. One can make sense of the notions of rank and border rank in $X$, and the border rank $\le r$ locus is a closed $\GL$-subvariety of $X$. By studying this $\GL$-variety, one can hope to understand aspects of border rank that are independent of the dimension of the vector space.
\end{example}

\begin{example} \label{ex:str-var}
Let $R=\bigoplus_{n \ge 0} R_n$ be a graded ring. The \defn{strength} of a homogeneous element $f \in R$ is the minimal $n$ for which there exists an expression $f=\sum_{i=1}^n g_i h_i$ where $g_i$ and $h_i$ are homogeneous elements of positive degree. This concept was introduced by Ananyan and Hochster \cite{ah} in their proof of Stillman's conjecture, though similar ideas had been previously considered. Applying this in the case $R=\Sym(V)$, with $V$ a vector space, we obtain a notion of strength for elements of $\Sym^d(V)$. This is not the most direct analog of tensor rank, but it is a closely related idea. The strength $\le k$ locus in $\Sym^d(V)$ is not closed in general \cite{bbov}, but its closure is an interesting closed subvariety of $\Sym^d(V)$.

Now let $X=\Sym^d(\bV)^* = \bA^{(d)}$. Then one can make sense of the strength $\le k$ locus in $X$, and its closure is a closed $\GL$-subvariety of $X$. As in the previous example, one can hope to understand aspects of strength that are independent of the number of variables by studying this variety. This strategy has been carried out successfully in \cite{ess,dll} to give new proofs of Stillman's conjecture.
\end{example}

\subsection{Structural results}

We now describe our main structural results about $\GL$-varieties. Let $G(n)$ be the subgroup of $\GL$ consisting of block matrices of the form
\begin{displaymath}
\begin{pmatrix} \id_n & 0 \\ 0 & \ast \end{pmatrix}.
\end{displaymath}
The group $G(n)$ is in fact isomorphic to $\GL$. Thus if $X$ is a $\GL$-variety, we can restrict the action of $\GL$ to $G(n)$ and then identify $G(n)$ with $\GL$ to obtain a new action of $\GL$. We call this the $n$th \defn{shift} of $X$, and denote it by $\Sh_n(X)$. This is a very important operation on $\GL$-varieties that appears throughout the paper.

We can now state our first main theorem:

\begin{namedthm}{Embedding theorem}[Theorem~\ref{thm:embed2}]
Let $Y$ be a $\GL$-variety, let $\lambda$ be a non-empty partition, and let $X$ be a closed $\GL$-subvariety of $Y \times \bA^{\lambda}$. Then one of the following holds:
\begin{enumerate}
\item We have $X=Y_0 \times \bA^{\lambda}$ for some closed $\GL$-subvariety $Y_0 \subset Y$.
\item There is a non-empty open $\GL$-subvariety of
$\Sh_n(X)$, for some $n$, that embeds into $\Sh_n(Y) \times
\bA^{\umu}$ for some tuple $\umu$, where each partition in $\umu$ is smaller than $\lambda$ (in the sense that the Young diagram has fewer boxes).
\end{enumerate}
\end{namedthm}

This theorem is important since it yields an inductive
approach to studying $\GL$-varieties. Indeed, suppose $X$ is
an affine $\GL$-variety. Then, by definition, $X$ is a
closed $\GL$-subvariety of $\bA^{\ul{\kappa}}$ for some
tuple $\ul{\kappa}=[\kappa_1, \ldots, \kappa_r]$. Arrange
the $\kappa$'s so that $\kappa_r$ has maximal size, and
apply the theorem with $Y=\bA^{\kappa_1} \times \cdots
\times \bA^{\kappa_{r-1}}$ and $\lambda=\kappa_r$. In case
(a), $X$ has a relatively simple form. In case (b), we find
that an open subset of a shift of $X$ embeds into a space of
the form $\Sh_n(Y) \times \bA^{\umu}$, where each $\mu_i$
has smaller size than $\kappa_r$. This space has the form
$\bA^{\urho}$ for some tuple $\urho$ with
$\magn(\urho)<\magn(\ul{\kappa})$, where $\magn$ is defined in \S \ref{ss:part}.

In fact, the embedding theorem was essentially proved by the
second author in the course of proving his noetherianity
theorem \cite{draisma}. However, the theorem was not
formally stated in that paper: it appeared only implicitly
within the main proof. We have isolated it here as a
standalone theorem since we have found it to be an important
tool; indeed, it is the key input to our next theorem.

An important theme in representation stability is that objects can be often made ``nice'' by shifting. A prototypical result of this kind is Nagpal's theorem \cite{nagpal} that a sufficient shift of a finitely generated $\mathbf{FI}$-module over a field of characteristic~0 is projective. We establish a shift theorem for $\GL$-varieties:

\begin{namedthm}{Shift theorem}[Theorem~\ref{thm:shift}]
Let $X$ be a non-empty affine $\GL$-variety. Then there is a non-empty open affine $\GL$-subvariety of $\Sh_n(X)$, for some $n$, that is isomorphic to $B \times \bA^{\ul{\kappa}}$ for some variety $B$ and tuple $\ul{\kappa}$.
\end{namedthm}

Here, and elsewhere in this paper, a variety without further
qualifications is a finite dimensional affine variety over $K$, i.e.,
the spectrum of a finitely generated reduced $K$-algebra.

The shift theorem has a number of important consequences.
For one, it shows that constructing rational points on
$\GL$-varieties is no more difficult than on finite
dimensional varieties. For example, if $K$ is algebraically
closed then the $K$-points of a $\GL$-variety are dense;
this follows easily from the definitions if $K$ is
uncountable (see \cite{lang} for details), but is not obvious when $K$ is countable. More generally, given a $\GL$-variety $X$ there is some $d$ such that the points of $X$ defined over degree $d$ extensions of $K$ are dense in $X$.

Here is a more interesting result that follows from the shift theorem:

\begin{namedthm}{Unirationality theorem}[Theorem~\ref{thm:uni}]
Let $X$ be an irreducible affine $\GL$-variety. Then there is a dominant morphism $\phi \colon B \times \bA^{\ulambda} \to X$ for some variety $B$ and tuple $\ulambda$.
\end{namedthm}

Recall that a variety $C$ is \defn{unirational} if there is a dominant rational morphism $\bA^n \dashrightarrow C$ for some $n$. Thus the above theorem can be interpreted as saying that $X$ is ``unirational in the $\GL$ direction'' or ``unirational up to a finite dimensional error.'' We emphasize that the $\phi$ in the theorem is actually a morphism, and not just a rational map. As an application, we show that the field $K(X)^{\GL}$ of invariant rational functions on $X$ is a finitely generated extension of $K$. We prove some finer results surrounding the unirationality theorem as well: for instance, we show that one can find $\phi$ that is surjective (though with $B$ reducible), and show how one can control $\ulambda$.

The shift theorem shows that if $X$ is an arbitrary $\GL$-variety then, after shifting, $X$ has a fairly simple form outside of some closed subset. We now describe how this theorem can be used to give a more cohesive picture of $\GL$-varieties. An \defn{elementary} $\GL$-variety is one of the form $B \times \bA^{\ulambda}$, where $B$ is a finite dimensional quasi-affine variety; we apply the same terminology to $G(n)$-varieties. We say that a $\GL$-variety is \defn{locally elementary} if every point admits an open elementary $G(n)$-stable neighborhood. With this language, we can now state the theorem:

\begin{namedthm}{Decomposition theorem}[Theorem~\ref{thm:decomp}]
Let $X$ be a quasi-affine $\GL$-variety. Then there is a finite decomposition $X=\bigsqcup_{i=1}^r Y_i$ where each $Y_i$ is a locally elementary $\GL$-variety that is locally closed in $X$.
\end{namedthm}

We also establish a version of the decomposition theorem for morphisms of $\GL$-varieties; in fact, this is a much more useful statement. The decomposition theorem (for morphisms) allows one to easily prove a variety of results about $\GL$-varieties, such as:

\begin{namedthm}{Chevalley's theorem}[Theorem~\ref{thm:chev}]
Let $\phi \colon X \to Y$ be a morphism of quasi-affine $\GL$-varieties and let $C$ be a $\GL$-constructible subset of $X$. Then $\phi(C)$ is a $\GL$-constructible subset of $Y$.
\end{namedthm}

Here, we say that a subset of a quasi-affine $\GL$-variety is \defn{$\GL$-constructible} if it is a finite union of locally closed $\GL$-stable subsets. One consequence of this theorem is that if $\phi$ is surjective on $\ol{K}$-points then it is scheme-theoretically surjective.

\subsection{Orbits and types} \label{ss:type}

Given a group action on a variety, it is often of interest to study
the orbits. In the case of $\GL$-varieties, it turns out that orbits
in the usual sense are pathological. Indeed, $\bA^{(2)}$ can be
identified with the space of all symmetric bilinear forms on $\bV$;
thus its points can be represented by infinite symmetric matrices.
Since elements of $\GL$ differ from the identity matrix in only
finitely many entries, they cannot change an infinite symmetric matrix
very much. To remedy this, we introduce the notion of
\defn{generalized orbit}: two points of a $\GL$-variety belong to the
same generalized orbit if and only if they have the same orbit
closure. Generalized orbits are well-behaved. For example, generalized
orbits on $\bA^{(2)}$ are parameterized by rank.
%In particular, there is a (unique) dense generalized orbit, consisting of infinite-rank symmetric matrices.
We let $X^{\orb}$ be the space of generalized orbits on a $\GL$-variety $X$. We are thus interested in the problem of describing $X^{\orb}$.

The first interesting phenomenon one observes is that many $\GL$-varieties contain a $K$-point with dense orbit; we call such points \defn{$\GL$-generic}. For example, the space $\bA^{(3)}=\Sym^3(\bV)^*$ of cubic forms in infinitely many variables contains a $\GL$-generic $K$-point. This is rather counterintuitive since the analogous statement in finitely many variables is clearly false: the variety $\Sym^3(K^n)^*$ of cubic forms in $n \gg 0$ variables clearly does not contain a $K$-point with dense $\GL_n$ orbit, as $\GL_n$ has dimension $n^2$ but the space has dimension $\approx \tfrac{1}{6} n^3$. In fact, we show (Proposition~\ref{prop:generic-exists}) that if $\ulambda$ is any \defn{pure} tuple (meaning it does not contain the empty partition) then $\bA^{\ulambda}$ contains a $\GL$-generic $K$-point. (We note that if $\lambda=\emptyset$ is the empty partition then $\bA^{\lambda}$ is the ordinary affine line with trivial $\GL$ action, and does not contain a $\GL$-generic $K$-point.)

Combining the above observation and the unirationality theorem leads to a strategy for studying generalized orbits. Let $x$ be a point of $X$ and let $\ol{O}_x$ be its orbit closure. We assume for simplicity that $x$ is a $K$-point and $K$ is algebraically closed.
Using the unirationality theorem, we show that there is a dominant map
$\phi \colon \bA^{\ulambda} \to \ol{O}_x$ for some irreducible variety $B$ and pure tuple
$\ulambda$. In fact, if we assume that
$\ulambda$ is appropriately minimal then $\phi$ is unique up to the
action $\Aut(\bA^{\ulambda})$; in this case, we say that $x$ has
\defn{type} $\ulambda$ and that $\phi$ is a \defn{typical} morphism
for $x$. It then follows that $x=\phi(y)$ for some $\GL$-generic $K$-point $y
\in \bA^{\ulambda}$. Now suppose that $x'$ is a second $K$-point that happens to have the same typical morphism $\phi$; this is not so improbable as the mapping space $\cM_{\ulambda}=\Map(\bA^{\ulambda}, X)$ is finite dimensional. We can then write $x'=\phi(y')$ for some $\GL$-generic $K$-point $y'$. Since $y$ and $y'$ are both $\GL$-generic, each belongs to the orbit closure of the other, and so the same is true for $x$ and $x'$; that is, $x$ and $x'$ belong to the same generalized orbit. This suggests that we might be able to use typical morphisms to detect generalized orbits.

To pursue this line of thought in more detail, we introduce
some notation. Let $X^{\orb}_{\ulambda}$ be the subspace of
$X^{\orb}$ on points having type $\le \ulambda$, where $\le$
is the containment partial order. Let $X^{\type}_{\ulambda}$
be the quotient of the finite dimensional variety
$\cM_{\lambda}$ by the action of the algebraic group
$\Aut(\bA^{\ulambda})$, equipped with the quotient topology. We have a well-defined map
\begin{displaymath}
\rho_{\ulambda} \colon X^{\type}_{\ulambda} \to X^{\orb}_{\ulambda}
\end{displaymath}
given by evaluating a morphism $\bA^{\ulambda} \to X$ on any $\GL$-generic point and taking the associated generalized orbit. (The well-definedness here is essentially the content of the previous paragraph.) The following is our main theorem about this construction:

\begin{theorem}[Theorem~\ref{thm:rholambda}]
The map $\rho_{\ulambda}$ is a continuous bijection.
\end{theorem}
Let us explain why this is a useful result. The space
$X^{\type}_{\ulambda}$, while not necessarily a
finite dimensional variety, is {\em algebraic} in nature: it is a quotient of a finite dimensional variety by the action of a finite dimensional algebraic group. The space $X^{\orb}_{\ulambda}$, on the other hand, is \emph{analytic} in nature. Indeed, to determine if two points $x$ and $y$ belong to the same generalized orbit, one must construct a sequence $g_{\bullet}$ in $\GL$ such that $g_i x$ converges to $y$ in an appropriate sense. The above theorem connects these two very different looking objects.

We introduced typical morphisms for $K$-points in the above discussion. In fact, we define the notion for arbitrary irreducible affine $\GL$-varieties: given such a variety $X$, a typical morphism is a dominant map $B \times \bA^{\ulambda} \to X$ that is appropriately minimal. This is unique up to the action of $\Aut(\bA^{\ulambda})$ and passing to finite covers and open subsets of $B$. In particular, the dimension of $B$ is an invariant of $X$, which we call the \defn{typical dimension} of $X$. We give the following elegant characterization of this quantity:

\begin{theorem}[Theorem~\ref{thm:typdim}]
Let $\phi:B\times \bA^{\ulambda} \to X$ be a typical morphism. Then
$K(B)$ is a finite extension of the field $K(X)^\GL$ of invariant
rational functions. In particular, the typical dimension of $X$ is the
transcendence degree of $K(X)^\GL$ over $K$.
\end{theorem}

\subsection{Further results} \label{ss:further}

Let $X$ be an irreducible $\GL$-variety. In a follow-up paper (in preparation), we prove the following results.
\begin{itemize}
\item Any two points of $X$ can be joined by an irreducible curve.
\item There exists a surjective morphism $B \times \bA^{\ulambda} \to X$ for some tuple $\ulambda$ and \emph{irreducible} variety $B$.
\item The continuous bijection $\rho_{\ulambda} \colon X^{\type}_{\ulambda} \to X^{\orb}_{\ulambda}$ is a homeomorphism.
\item The mapping space $\Map(\bA^{\ulambda}, X)$ is irreducible for all $\ulambda$.
\item Let $\phi \colon Y \to X$ be a map of $\GL$-varieties. Then any point in the image closure of $\phi$ can be realized as the limit of a ``nice'' 1-parameter family in $Y$.
\end{itemize}
These statements are closely related to one another; for example, it is easy to see that the first is a consequence of the second. The proofs of these results are rather long, which is why we have deferred them to a separate paper.

\subsection{Applications}

We discuss a few applications of our theory.

\begin{itemize}
\item Let $\cR$ be the inverse limit of the graded rings $K[x_1, \ldots, x_n]$. In \cite{ess} (see also \cite{am}), it is shown that $\cR$ is isomorphic to a polynomial ring (in uncountably many variables), which was an important component in the proofs of Stillman's conjecture given there. We give a new proof of this result using the theory developed in this paper. When $K=\bC \lpp t \rpp$ is the field of Laurent series, \cite{ess} introduced a subring $\cR^{\flat}$ of $\cR$ consisting of elements with bounded denominators, and showed that it too is a polynomial ring. In recent work \cite{relbig}, the third author showed that $\cR$ is a polynomial algebra over $\cR^{\flat}$; this result is significantly more difficult than the polynomiality results of \cite{ess}. We reprove this result as well. See \S \ref{ss:bigpoly}.
\item In \cite{bbov}, Ballico-Oneto-Ventura and one of the the current authors give a nonconstructive proof that the strength $\leq3$ locus in $\Sym^4(V)$ is not closed in general. In \cite{bdde}, Danelon and three of the current authors generalize
a theorem of Kazhdan-Ziegler \cite{kaz2} on universality of high-strength polynomials
to arbitrary polynomial functors. These results build on a weaker
variant of the unirationality theorem that appeared in the
first author's thesis \cite{bik}.
\item Our version of Chevalley's theorem shows that the tensor rank $r$ locus in Example~\ref{ex:rank-var} or the strength $r$ locus in Example~\ref{ex:str-var} are $\GL$-constructible sets.
\item Our Chevalley theorem also implies that for any
$\GL$-equivariant morphism $\bA^{\ulambda} \to \bA^{\umu}$ there exists
a polynomial-time membership test for the image. To
illustrate, for fixed $k_1,k_2,k_3$, there exists such a polynomial-time
algorithm that, on input $n$ and a tensor $T \in \ol{\bQ}^n \otimes
\ol{\bQ}^n
\otimes \ol{\bQ}^n$, tests whether $T$ is the sum of $k_1$ terms of the form $v_1 \otimes
v_2 \otimes v_3$, $k_2$ terms of the form $v_1 \otimes A_{23}$, and $k_3$
terms of the form $A_{12} \otimes v_3$, where the $v_i$ are
vectors in $\ol{\bQ}^n$ and the $A_{ij}$ are matrices in
$\ol{\bQ}^n \otimes \ol{\bQ}^n$. This algorithm just
verifies, on this input tensor, the uniform quantifier-free description of this set
promised by Chevalley's theorem; it
does not compute the $v_i$ and the $A_{ij}$. However, using the
decomposition theorem for morphisms, there is a polynomial-time
algorithm that computes these, as well.
\end{itemize}

\subsection{On the characteristic}

We restrict ourselves to characteristic zero for two related
reasons.  First, in positive characteristic, the Schur
functors $\bS_\lambda$
are typically not
irreducible, and one should replace direct sums of these
with objects in the larger category of polynomial functors. This category is not
semisimple, and this complicates some of the reasoning. Second, the
proof of one of our key tools, the second version of the Embedding
Theorem, does not work as stated in positive characteristic, even when direct sums of Schur functors are replaced by arbitrary
polynomial functors; see Remark~\ref{re:poschar}. However, the proof of Noetherianity in \cite{draisma}
shows that in positive characteristic, the map that is required to be
a closed embedding in Theorem~\ref{thm:embed2} is a universal
homeomorphism onto its image. This suggests that most of our results
have analogs in positive characteristic. We will pursue this
line of thought in a follow-up paper.

\subsection{Notation}

We typically denote ordinary, finite dimensional varieties by $B$ or $C$, and $\GL$-varieties by $X$ or $Y$. Other important notation:
\begin{description}[align=right,labelwidth=2cm,leftmargin=!]
\item [$K$] the base field, always characteristic~0
\item [$\Omega$] a (variable) extension of $K$, often algebraically closed
\item [$\bV$] a fixed infinite dimensional $K$-vector space with basis $\{e_i\}_{i \ge 1}$
\item [$\GL$] the group of automorphisms of $\bV$ fixing almost all basis vectors
\item [$G(n)$] the subgroup of $\GL$ fixing $e_1, \ldots, e_n$ and mapping $\operatorname{span} \{ e_i \}_{i>n}$ into itself.
\item [$\Sh_n$] the shift functor
\item [$\sh_n$] the shift operation on partitions or tuples
\item [{$X[1/h]$}] the open subset of the variety $X$ defined by $h \ne 0$.
\item [$X\{V\}$] the result of taking the affine $\GL$-scheme $X$, regarding it as functor on vector spaces, and evaluating on the $K$-vector space $V$
\item [$\bV_{\lambda}$] the irreducible $\GL$-representation $\bS_{\lambda}(\bV)$, where $\bS_{\lambda}$ is a Schur functor
\item [$\ulambda$] a tuple $[\lambda_1, \ldots, \lambda_r]$ of partitions
\item [$\bV_{\ulambda}$] the representation $\bV_{\lambda_1} \oplus \cdots \oplus \bV_{\lambda_r}$
\item [$\bR_{\ulambda}$] the symmetric algebra on $\bV_{\ulambda}$
\item [$\bA^{\ulambda}$] the affine $\GL$-variety with coordinate ring $\bR_{\ulambda}$
\item [$\cM_{\ulambda}(X)$] the space of maps $\bA^{\ulambda} \to X$.
\end{description}

\section{\texorpdfstring{$\GL$}{GL}-varieties} \label{s:GLvar}

\subsection{Partitions and tuples} \label{ss:part}

Recall that a \defn{partition} is a sequence $\lambda=(\lambda_1, \lambda_2, \ldots)$ of non-negative integers with $\lambda_i \ge \lambda_{i+1}$ for all $i$ and $\lambda_i=0$ for $i \gg 0$. We write $\vert \lambda \vert$ for the sum $\lambda_1+\lambda_2+\cdots$, and call it the \defn{size} of $\lambda$. The empty partition $\emptyset$ is the partition with all parts equal to~0.

A \defn{tuple of partitions} (often simply called a \defn{tuple}) is a finite tuple $\ulambda=[\lambda_1, \lambda_2, \ldots, \lambda_r]$ where each $\lambda_i$ is a partition. We use square brackets for the notation since each $\lambda_i$ is denoted with parentheses. We say that a tuple is \defn{pure} if it does not contain the empty partition. We say that a tuple $\ulambda=[\lambda_1, \ldots, \lambda_r]$ \defn{contains} a tuple $\umu=[\mu_1, \ldots, \mu_s]$ if $r \ge s$ and, perhaps after reordering, we have $\lambda_i=\mu_i$ for $1 \le i \le s$; we denote this by $\umu \subset \ulambda$. For tuples $\ulambda=[\lambda_1, \ldots, \lambda_r]$ and $\umu=[\mu_1, \ldots, \mu_s]$, we let $\ulambda \cup \umu$ be the tuple $[\lambda_1, \ldots, \lambda_r, \mu_1, \ldots, \mu_s]$.

We define the \defn{magnitude} of a tuple $\ulambda$,
denoted $\magn(\ulambda)$, to be the sequence $(n_0, n_1, \ldots)$ where $n_i$ is the number of partitions in $\ulambda$ of size $i$. We compare magnitudes lexicographically, that is, $\magn(\ulambda)<\magn(\umu)$ if $\magn(\ulambda)_i<\magn(\umu)_i$ where $i$ is the largest index where the magnitudes disagree. This is a well-order. We can therefore argue by induction on magnitude. We define the \defn{degree} of a tuple $\ulambda$, denoted $\deg(\ulambda)$ to be the maximal size of a partition in $\ulambda$.

\subsection{Polynomial representations} \label{ss:polyrep}

We fix a field $K$ of characteristic~0 throughout this paper. Let $\GL=\bigcup_{n \ge 1} \GL_n(K)$ be the infinite general linear group, thought of as a discrete group, and let $\bV=\bigcup_{n \ge 1} K^n$ be its standard representation. A \defn{polynomial representation} of $\GL$ is one that occurs as a subquotient of a (possibly infinite) direct sum of tensor powers of $\bV$. The category of polynomial representations is a semi-simple abelian category.

For a partition $\lambda$, we write $\bV_{\lambda}$ for the polynomial representation $\bS_{\lambda}(\bV)$ where $\bS_{\lambda}$ denotes the Schur functor associated to $\lambda$. The $\bV_{\lambda}$'s are exactly the irreducible polynomial representations. For a tuple $\ulambda=[\lambda_1, \ldots, \lambda_r]$ of partitions, we write $\bV_{\ulambda}$ for $\bigoplus_{i=1}^r \bV_{\lambda_i}$. Every finite length polynomial representation is isomorphic to some $\bV_{\ulambda}$.

Every polynomial representation $V$ carries a canonical grading. For $t \in K^{\times}$, let $A_n(t) \in \GL$ be the diagonal matrix where the first $n$ entries are $t$ and the remaining entries are~1; one should think of $A_n(t)$ as an approximation of the scalar matrix $t \cdot \id$. If $v \in V$ then $A_n(t) v$ is independent of $n$ for $n \gg 0$; we denote this common value by $A(t) v$. Then $v$ is homogeneous of degree $d$ if and only if $A(t) v=t^d v$ for all $t \in K^{\times}$. Under this grading, the irreducible representation $\bV_{\lambda}$ is homogeneous of degree $\vert \lambda \vert$. We note that an element of $v$ is $\GL$-invariant if and only if it has degree~0. We define an action of the multiplicative group $\bG_m$ on a polynomial representation $V$ by letting $t \in \bG_m$ act by $t^d$ on the degree $d$ piece. We refer to this as the ``central $\bG_m$.''

The category of polynomial representations is equivalent to the category of polynomial functors. Given a polynomial representation $V$, we let $V\{K^n\}$ be the result of evaluating the corresponding polynomial functor on $K^n$. Explicitly, this can be identified with $V^G$ where $G$ is the subgroup of $\GL$ fixing each of $e_1, \ldots, e_n$. The operation $V \mapsto V\{K^n\}$ is compatible with tensor products and arbitrary direct sums.

The action of $\GL$ on a polynomial representation $V$ extends in a few ways. First the action of the discrete group $\GL_n(K) \subset \GL$ extends uniquely to an action of the algebraic group $\GL_n$; that is, $V$ is naturally a comodule over the Hopf algebra $K[\GL_n]$. And second, the action of $\GL$ naturally extends to an action of the monoid $\End(\bV)$. These new actions are completely natural, and compatible with direct sums and tensor products.

\subsection{\texorpdfstring{$\GL$}{GL}-algebras}

A \defn{$\GL$-algebra} is a $K$-algebra $R$ equipped with an action of the group $\GL$ by algebra automorphisms such that it forms a polynomial representation. Such an algebra is naturally graded, since every polynomial representation is graded, and any $\GL$-equivariant homomorphism of $\GL$-algebras is homogeneous. We let $\bR_{\ulambda}=\Sym(\bV_{\ulambda})$; this is a typical, and important, example of a $\GL$-algebra. We say that a $\GL$-algebra $R$ is \defn{finitely $\GL$-generated} if it is generated, as a $K$-algebra, by the $\GL$-oribts of finitely many elements; equivalently, $R$ is a quotient of $\bR_{\ulambda}$ for some $\ulambda$.

Let $R$ be a $\GL$-algebra. A \defn{$\GL$-ideal} of $R$ is an ideal that is $\GL$-stable. We say that a $\GL$-ideal is \defn{finitely $\GL$-generated} if it is generated, as an ideal, by the $\GL$-orbits of finitely many elements. Equivalently, a $\GL$-ideal $I$ is finitely $\GL$-generated if there exists a $\GL$-equivariant surjection of $R$ modules $R \otimes \bV_{\ulambda} \to I$ for some $\ulambda$.

Let $R$ be a $\GL$-algebra over an ordinary ring $A$, that is, $\GL$ acts trivially on $A$ and we have a ring homomorphism $A \to R_0$. We say that $R$ is \defn{pure} over $A$ if the map $A \to R_0$ is an isomorphism. In this case, there is a natural map $R \to A$ that sends all positive degree elements to~0. We note that $\bR_{\ulambda}$ is pure over $K$ if and only if $\ulambda$ is pure.

\subsection{\texorpdfstring{$\GL$}{GL}-varieties} \label{ss:GLvar}

An \defn{affine $\GL$-scheme} is an affine scheme $X$ equipped with an action of $\GL$ such that $\Gamma(X, \cO_X)$ is a $\GL$-algebra. An \defn{affine $\GL$-variety} is a reduced affine $\GL$-scheme such that $\Gamma(X, \cO_X)$ is finitely $\GL$-generated. A \defn{quasi-affine $\GL$-variety} is a $\GL$-stable open subscheme of an affine $\GL$-variety. A \defn{morphism} of $\GL$-schemes is simply a $\GL$-equivariant morphism of schemes.

Let $\bA^{\ulambda}=\Spec(\bR_{\ulambda})$. This is an affine $\GL$-variety, and serves as the principal example: any affine $\GL$-variety can be realized as a closed $\GL$-subvariety of $\bA^{\ulambda}$ for some $\ulambda$. Thus the $\bA^{\ulambda}$ take the place of the usual affine spaces in classical algebraic geometry. For a singleton $\ulambda=[\lambda]$, we simply write $\bA^{\lambda}$ in place of $\bA^{\ulambda}$. If $\lambda$ is the empty partition then $\bA^{\lambda}=\bA^1$ is the ordinary affine line.

%We will, in a few places, require a slightly more general notion. Let $S$ be an ordinary scheme over $K$. An \defn{affine $\GL$-scheme over $S$} is a pair $(X,f)$ where $X$ is a scheme equipped with an action of $\GL$ and $f \colon X \to S$ is a $\GL$-invariant morphism of schemes such that $S$ admits an open affine cover $\{U_{\alpha}\}_{\alpha \in I}$ such that $f^{-1}(U_{\alpha})$ is an affine $\GL$-scheme for all $\alpha$. A typical example is $\cE \otimes \bA^{\ulambda}$, where $\cE$ is a vector bundle on $S$; if $\cE$ has rank $r$ then $\cE \otimes \bA^{\ulambda}$ is locally isomorphic to $(\bA^{\ulambda})^r$, and the glueing maps are deduced from those of $\cE$. The notion of \defn{quasi-affine $\GL$-scheme over $S$} is defined similarly, and the notion of morphism for such objects is evident.

Let $f \colon X \to S$ be an affine $\GL$-scheme over the ordinary affine scheme $S$. We say that $X$ is \defn{pure} over $S$ if (locally) the map on coordinate rings is pure In this case, the map $f$ admits a canonical $\GL$-invariant section $S \to X$ that we call the \defn{zero section}. We note that $\bA^{\ulambda}$ is pure over $K$ if and only if the tuple $\ulambda$ is pure.

Let $X=\Spec(R)$ be an affine $\GL$-scheme. Recall that $R\{K^n\}$ is the result of evaluating $R$ on $K^n$, where we regard $R$ as a polynomial functor. This is a $K$-algebra equipped with a $\GL_n$-action. We define $X\{K^n\}=\Spec(R\{K^n\})$. We note that $R\{K^n\}$ can be realized as a subalgebra of $R$, and so if $R$ is reduced or integral then so is $R\{K^n\}$.

\subsection{Noetherianity} \label{ss:noeth}

The following noetherianity theorem was established by the second author, and suggested that the theory pursued in this paper might be within reach:

\begin{theorem}[\cite{draisma}] \label{thm:Noetherianity}
For any tuple $\ulambda$, the closed $\GL$-subsets of $\bA^{\ulambda}$ satisfy the descending chain condition. Equivalently, any closed $\GL$-subvariety of $\bA^{\ulambda}$ is the zero locus of a finitely $\GL$-generated $\GL$-ideal of $\bR_{\ulambda}$.
\end{theorem}

In fact, \cite{draisma} establishes the natural analog of the above
theorem in positive characteristic, and more recently it has been
established over arbitrary noetherian base rings \cite{bdd}. We stress,
though, that this finitely $\GL$-generated $\GL$-ideal might not be
radical; e.g., the stronger, ideal-theoretic version of Noetherianity
is not known to hold in general.

\subsection{Mapping spaces}
\label{ss:mapping}

Let $X$ and $Y$ be $\GL$-schemes over an affine base $K$-scheme $S$.
We let $\Map_S(X,Y)$ be the set of maps $X \to Y$ of
$\GL$-schemes over $S$. The group $\bG_m(S)$ acts on this
set, via the action of the central $\bG_m$ on $Y$ (or on
$X$; the two actions differ by an inverse). If $Y$ is pure
over $S$ then we let $0 \in \Map_S(X,Y)$ be the composition
of $X \to S$ with the zero section $S \to Y$. We let
$\uMap_S(X,Y)$ be the the functor from affine $S$-schemes to sets
given by $T \mapsto \Map_T(X_T, Y_T)$, where $(-)_T$ denotes
base change to $T$. This carries an action of $\bG_m$, and
if $Y$ is pure over $S$, then there is a zero section $0
\colon S \to \uMap_S(X,Y)$ (where we identify $S$ with
the functor it represents).

\begin{example} \label{ex:mapping}
Suppose $S=\Spec(K)$, $X=\Spec(A)$ for some finitely $\GL$-generated $\GL$-algebra $A$ with $A_0=K$, and $Y=\bA^{\lambda}$ for a partition $\lambda$. Suppose that $\bV_{\lambda}$ occurs with multiplicity $n$ in $A$. Then
\begin{displaymath}
\Map_S(X,Y)=\Hom_{\text{$\GL$-alg}}(\Sym(\bV_{\lambda}), A)=\Hom_{\text{$\GL$-rep}}(\bV_{\lambda}, A) \cong K^n.
\end{displaymath}
More generally, if $R$ is a $K$-algebra then $\uMap_S(X,Y)(\Spec(R)) \cong R^n$. It follows that $\uMap_S(X,Y)$ is represented by $\bA^n$.
\end{example}

\begin{proposition} \label{prop:Map}
Let $S$ be a noetherian affine scheme over $K$, let $X \to S$ be a $\GL$-scheme that is affine, pure, $\GL$-finite type, and flat over $S$, and let $Y \to S$ be a $\GL$-scheme that is affine and $\GL$-finite type over $S$.
\begin{enumerate}
\item $\uMap_S(X,Y)$ is represented by a finite type affine scheme over $S$.
\item Suppose $Y$ is pure over $S$. Then the quotient of $\uMap_S(X,Y) \setminus \{0\}$ by $\bG_m$ is a closed subvariety of a weighted projective space over $S$.
\item Suppose $Y=\bA^{\ulambda}_S$ for a (possibly non-pure) tuple $\ulambda$. Then $\uMap_S(X,Y)$ is a vector bundle over $S$.
\end{enumerate}
\end{proposition}

\begin{proof}
(c) This is essentially the same computation as Example~\ref{ex:mapping}. Let $\cO_S$, $\cO_X$, and $\cO_Y$ be the rings for $S$, $X$, and $Y$, and let $\cO_X = \bigoplus_{\lambda} \cO_{X,\lambda} \otimes_K \bV_{\lambda}$ be the isotypic decomposition of $\cO_X$. Note that since $\cO_X$ is finitely $\GL$-generated over $\cO_S$, each multiplicity space $\cO_{X,\lambda}$ is a finitely generated $\cO_S$-module. Since $X$ is flat over $S$ it follows that $\cO_{X,\lambda}$ is flat as an $\cO_S$-module; since $\cO_S$ is noetherian, it thus follows that $\cO_{X,\lambda}$ is projective as an $\cO_S$-module.

Regarding $\uMap_S(X,Y)$ as a functor on $\cO_S$-algebras, we have
\begin{align*}
\uMap_S(X,Y)(R)
=& \Hom_{\text{$(\cO_S,\GL)$-alg}}(\cO_Y, R \otimes_{\cO_S} \cO_X) \\
=& \Hom_{\text{$\GL$-rep}}(\bV_{\ulambda}, R \otimes_{\cO_S} \cO_X)
= \bigoplus_{i=1}^r R \otimes_{\cO_S} \cO_{X,\lambda_i}
\end{align*}
where $\ulambda=[\lambda_1, \ldots, \lambda_r]$. Now, if $P$ is a finitely generated projective $\cO_S$-module then, putting $P^{\vee}=\Hom_{\text{$\cO_S$-mod}}(P, \cO_S)$, we have
\begin{displaymath}
R \otimes_{\cO_S} P = \Hom_{\text{$\cO_S$-mod}}(P^{\vee}, R) = \Hom_{\text{$\cO_S$-alg}}(\Sym(P^{\vee}), R).
\end{displaymath}
Thus $R \mapsto R \otimes_{\cO_S} P$ is represented by the vector bundle $\Spec(\Sym(P^{\vee}))$ over $S$. We thus see that if $\cE_i$ is the vector bundle associated to $\cO_{X,\lambda_i}$ then $\uMap_S(X,Y)$ is represented by the vector bundle $\cE=\bigoplus_{i=1}^r \cE_i$.

Before continuing, we make one additional observation. An element $x \in \bG_m$ acts on $\cE_i$ via multiplication by $x^{d_i}$ where $d_i=\vert \lambda_i \vert$. If $\ulambda$ is pure then $d_i>0$ for all $i$, and so the quotient of $\cE \setminus \{0\}$ by $\bG_m$ is a family of weighted projective spaces over $S$.

(a) Since $Y$ is of finite type over $S$, it embeds into $\bA^{\ulambda}_S$ for some (possibly non-pure) tuple $\lambda$. First suppose the ideal of $Y$ in $\bA^{\ulambda}_S$ is equivariantly finitely generated. We can then realize $Y$ as $f^{-1}(0)$ for some map of $\GL$-schemes $f \colon \bA^{\ulambda}_S \to \bA^{\umu}_S$ with $\umu$ some other tuple (the $\mu_i$'s correspond to the generators of the ideal). In other words, we have a cartesian diagram
\begin{displaymath}
\xymatrix{
Y \ar[r] \ar[d] & \bA^{\ulambda}_S \ar[d] \\
S \ar[r]^-0 & \bA^{\umu}_S. }
\end{displaymath}
(Note that even if $\umu$ is not pure, the space $\bA^{\umu}_S$ has a natural zero section.) Since $\uMap$ is compatible with fiber products in the second variable, we obtain a cartesian diagram
\begin{displaymath}
\xymatrix{
\uMap_S(X,Y) \ar[r] \ar[d] & \uMap_S(X, \bA^{\ulambda}_S) \ar[d] \\
\uMap_S(X,S) \ar[r]^-0 & \uMap_S(X,\bA^{\umu}_S). }
\end{displaymath}
These functors, other than the top left, are representable by finite type affine schemes by (c), and so it follows that the top left one is as well.

We now treat the case where the ideal for $Y$ is not equivariantly finitely generated. Express the ideal for $Y$ as a directed union of equivariantly finitely generated ideals; thus $Y=\bigcap_{i \ge 1} Y_i$ for a descending chain $\{Y_i\}_{i \ge 1}$ of closed $\GL$-subschemes of $\bA^{\ulambda}_S$. Since $\uMap$ is compatible with such intersections in its second argument, we see that $\uMap_S(X,Y)=\bigcap_{i \ge 1} \uMap_S(X,Y_i)$ as subfunctors of $\uMap_S(X,\bA^{\ulambda}_S)$. Since each $\uMap_S(X,Y_i)$ is a closed subscheme of $\uMap_S(X,\bA^{\ulambda}_S)$, it follows that $\uMap_S(X,Y)$ is too. In fact, since $\uMap_S(X,\bA^{\ulambda}_S)$ is noetherian, this descending chain stabilizes, and so $\uMap_S(X,Y)=\uMap_S(X,Y_i)$ for $i \gg 0$.

(b) Choose a closed embedding $Y \to \bA^{\ulambda}_S$ for a pure tuple $\ulambda$. Then $\uMap_S(X,Y)$ is a $\bG_m$-stable closed subscheme of $\uMap_S(X,\bA^{\ulambda}_S)$. We have already seen that the quotient of $\uMap_S(X,\bA^{\ulambda}_S) \setminus \{0\}$ by $\bG_m$ is a weighted projective space, and so the result follows.
\end{proof}

\begin{remark}
The hypotheses in the above proposition can be relaxed in various ways. For example, instead of $X$ being pure over $S$, it is enough for $X_0 \to S$ to be finite.
\end{remark}

In this paper, we only apply the proposition when $X$ is an affine space. We introduce a special notation for that case:

\begin{definition}
Let $Y$ be an affine $\GL$-variety over $K$ and let $\ulambda$ be a pure tuple. We put $\cM_{\ulambda}(Y)=\uMap_K(\bA^{\ulambda}, Y)$, which is a finite type affine scheme over $K$.
\end{definition}

\begin{remark} \label{rmk:mapping-ones}
Let $\ulambda=[(1)^n]$ be the tuple consisting of $n$ copies of the
partition $(1)$. Then $\cM_{\ulambda}(Y)$ is naturally identified
with $Y\{K^n\}$. This is easiest explained on $K$-points: given a
$K$-point $y$ of $Y\{K^n\}$, we have a morphism of $K$-schemes
$\bA^{\ulambda} \to Y$ defined on $K$-points as follows. A $K$-point
of the former space is a tuple $(v_1,\ldots,v_n)$ of vectors $v_i
\in \bV^*$. This tuple defines a linear map $\bV \to K^n, e_i \mapsto
(v_{1i},\ldots,v_{ni})$ and hence, by functoriality, a morphism $Y\{K^n\}
\to Y\{\bV\}=Y$ of $K$-schemes. Applying this morphism to $y$ yields a
$K$-point of $Y$. Conversely, if $\phi$ is a $K$-morphism $\bA^{\ulambda}
\to Y$, then we apply $\phi$ to the $K$-point $(e^1,e^2,\ldots,e^n)$
of the former space to get a $K$-point in $Y\{K^n\}$; here $e^i$ is
the linear function on $\bV$ defined by $e^i(e_j)=\delta_{ij}$. These
two constructions are inverse to each other, and extends to points with
values in arbitrary $K$-algebras.
\end{remark}

\section{Generalized orbits} \label{s:genorbit}

\subsection{Orbit closures}

Let $X$ be an affine $\GL$-scheme. We define the \defn{orbit closure} of a subset $W$ of $X$, denoted $\ol{O}_W$, to be the intersection of all closed $\GL$-subsets of $X$ that contain $W$. We are most interested in the case where $W=\{x\}$ is a point, in which case we simply write $\ol{O}_x$. The following proposition summarizes the basic properties of orbit closures:

\begin{proposition} \label{prop:FirstProperties}
We have the following:
\begin{enumerate}
\item For any subset $W$ of $X$, we have $\ol{O}_W= \ol{\GL \cdot W}$.
\item For any subset $W$ of $X$ we have $\ol{O}_W=\ol{O}_{\ol{W}}$, where $\ol{W}$ denotes the closure of $W$.
\item If $W$ is an irreducible subset of $X$ then $\ol{O}_W$ is also irreducible.
\item Every irreducible component (maximal irreducible subset) of $X$ is $\GL$-stable.
\item For every $x \in X$ the orbit closure $\ol{O}_x$ is irreducible and its generic point is $\GL$-fixed.
\item Every point $x\in X$ that is fixed by $\GL$ is the generic point of its orbit closure $\ol{O}_x=\overline{\{x\}}$.
\end{enumerate}
\end{proposition}

\begin{proof}
(a) For $\subseteq$ observe that the right-hand side is a closed
$\GL$-subset of $X$ that contains $W$. For $\supseteq$ observe that each
closed $\GL$-subset of $X$ containing $W$ also contains the right-hand
side.

(b) The inclusion $\subseteq$ follows from $W \subseteq
\ol{W}$. For the inclusion $\supseteq$ observe that every closed
$\GL$-subset of $X$ that contains $W$ also contains $\ol{W}$.

(c) By (b), and since the closure of an irreducible set is irreducible, we may assume that $W$ is a closed subset, hence (the underlying space of) an irreducible affine scheme. For every $n \in \bZ_{\geq 0}$, $\overline{\GL_n \cdot W}$ is irreducible since it is the closure of the image of the irreducible affine scheme $\GL_n \times W$ under a morphism. (Here we appeal to the comment at the end of \S \ref{ss:polyrep} that the action of the discrete group $\GL_n(K) \subset \GL$ on the coordinate ring of $X$ extends uniquely to an action of the algebraic group $\GL_n$.) Hence, if $\GL \cdot W \subseteq X_1 \cup X_2$ with $X_1$ and $X_2$ closed, then for all $n$ we have $\GL_n \cdot W \subseteq X_{i_n}$ for some $i_n \in \{1,2\}$. Taking the union of the $\GL_n$ over a subsequence where the $i_n$ are constant, we find that $\GL \cdot W \subseteq X_1$ or $\GL \cdot W \subseteq X_2$, and hence also $\ol{\GL \cdot W} \subseteq X_1$ or $\ol{\GL \cdot W} \subseteq X_2$. Now use (a).

(d) Let $W$ be an irreducible component. Then $W$ is contained in
the closed $\GL$-subset $\ol{O}_W$, which by (c) is irreducible. By
maximality of $W$, we have $W=\ol{O}_W$.

(e) The first statement follows from (c). The statement about
the generic point says that the prime ideal defining $\ol{O}_x$ is
$\GL$-stable.

(f) If $x \in X$ is $\GL$-fixed, then $\ol{O}_x=\ol{\{x\}}$ is
irreducible and has $x$ as its generic point.
\end{proof}

\subsection{Generalized orbits}

Let $X$ be an affine $\GL$-scheme. We define the \defn{generalized orbit} of $x \in X$, denoted $O_x$, to be the set of all points $y \in X$ such that $\ol{O}_x=\ol{O}_y$. The generalized orbits partition $X$ into disjoint subsets, just as ordinary orbits do. Note that $O_x$ contains the usual orbit $\GL \cdot x$; in particular, by Proposition~\ref{prop:FirstProperties}(a), the closure of $O_x$ is $\ol{O}_x$, so the notation is consistent.

\begin{example}
Suppose $X=\bA^{(1)}$, which is (the scheme associated to) the dual vector space $\bV^*$. The orbits of $\GL$ on $\bV^*$ behave pathologically: indeed, a point of $\bV^*$ can be represented by infinitely many coordinates but an element of $\GL$ only affects finitely many coordinates, and so two points of $\bV^*$ can only belong to the same orbit if all but finitely many of their coordinates are equal. The generalized orbits, on the other hand, behave as expected: there are two, namely, the origin and everything else.
\end{example}

\begin{example} \label{ex:A2-orbits}
Let $X=\bA^{(2)}$, which one can regard as the space of symmetric bilinear forms on $\bV$. One can show that two points belong to the same generalized orbit if and only if they have the same rank. Thus the generalized orbits are naturally index by the set $\bN \cup \{\infty\}$.
\end{example}

\begin{proposition} \label{prop:O-as-int-opens}
The set $O_x$ is the intersection of all non-empty open $\GL$-subsets of $\ol{O}_x$.
\end{proposition}

\begin{proof}
Let $W$ be the intersection of all non-empty open $\GL$-subsets of $\ol{O}_x$. We show $W=O_x$.

Let $U$ be a non-empty open $\GL$-subset of $\ol{O}_x$, and let $Z$ be its complement. Suppose $y \in O_x$. If $y$ belonged to $Z$ then $\ol{O}_y=\ol{O}_x$ would be contained in $Z$, which is not the case. Thus $y \in U$. Since $y$ is arbitrary, we have $O_x \subset U$, and since $U$ is arbitrary, we find $O_x \subset W$.

Now suppose that $y \in \ol{O}_x \setminus O_x$. Then, by definition, $Z=\ol{O}_y$ is not all of $\ol{O}_x$. Thus $U=\ol{O}_x \setminus Z$ is a non-empty open $\GL$-subset of $\ol{O}_x$ and $y \not\in U$; thus $y \not\in W$. Since $y$ is arbitrary, we have $\ol{O}_x \setminus O_x \subset \ol{O}_x \setminus W$, and so $W \subset O_x$. This completes the proof.
\end{proof}

\begin{proposition} \label{prop:gen-orbit-fixed-pt}
The generic point of $\ol{O}_x$ is the unique $\GL$-fixed point in $O_x$.
\end{proposition}

\begin{proof}
For $x \in X$, the set $\ol{O}_x$ is irreducible by Proposition~\ref{prop:FirstProperties}. Let $y$ be its generic point. Then $y$ is $\GL$-fixed since $\ol{O}_x$ is $\GL$-stable, and we have $\ol{O}_y = \ol{\{y\}} = \ol{O}_x$, so $O_x =O_y$. Suppose that $z \in O_x$ were some other $\GL$-fixed point. By Proposition~\ref{prop:FirstProperties}(f), we see that $z$ is the generic point of $\ol{O}_z=\ol{O}_x$, and thus equal to $y$ since $\ol{O}_x$ has a unique generic point (affine schemes are sober).
\end{proof}

\subsection{The space of orbits}

Let $X$ be an affine $\GL$-scheme. We define the \defn{orbit space} $X^{\orb}$ to be the set of generalized orbits in $X$, and we let $\pi \colon X \to X^{\orb}$ be the quotient map, defined by $\pi(x)=O_x$. We give $X^{\orb}$ the quotient topology, so that a subset $U$ of $X^{\orb}$ is open if and only if $\pi^{-1}(U)$ is an open subset of $X$. We let $X^{\GL}$ denote the subset of $X$ consisting of points fixed by the group $\GL$, and we endow it with the subspace topology.

\begin{proposition} \label{prop:XGL-Xorb}
The map $X^{\GL} \to X^{\orb}$ induced by $\pi$ is a homeomorphism.
\end{proposition}

\begin{proof}
By Proposition~\ref{prop:gen-orbit-fixed-pt}, the map is a bijection, and it is the restriction of a
continuous map, hence continuous.  Next assume that $U$ is an open subset
of $X^\GL$, say $U=V \cap X^\GL$ with $V$ open in $X$. Then
$V':=\GL \cdot V$
is also open in $X$ and $\GL$-stable, and $V' \cap X^\GL=V \cap X^\GL=U$,
so after replacing $V$ by $V'$ we may assume that $V$ is a $\GL$-stable
open subset of $X$. Let $Y \subseteq X$ be its complement. Then the orbit
closure of a point $y \in Y$ is contained in $Y$ and hence $O_v=O_y$ does not hold for any $v \in V$. It follows that $V$ is a union of
generalized orbits. Consequently, $\pi^{-1}(\pi(V))=V$, hence $\pi(V)$
is open by definition of the quotient topology. Finally,
since each generalized orbit  contains a $\GL$-stable point,
we have $\pi(V)=\pi(U)$. So the inverse of $\pi|_{X^{\GL}}$
is also continuous, as desired.
\end{proof}

\begin{proposition} \label{prop:opens-of-Xorb}
The open (resp.\ closed) $\GL$-subsets of $X$ are in bijection with the open (resp.\ closed) subsets of $X^{\orb}$, via $\pi$ and $\pi^{-1}$.
\end{proposition}

\begin{proof}
A closed $\GL$-suset $Z \subseteq X$ contains the generalized orbit of each of its points. By Proposition~\ref{prop:gen-orbit-fixed-pt} we have $\pi(Z)=\pi(Z^\GL)$, which by Proposition~\ref{prop:XGL-Xorb} is closed, and its preimage in $X^\GL$ is $Z^\GL$; it follows that $\pi^{-1}(\pi(Z))=Z$. Moreover, also by Proposition~\ref{prop:XGL-Xorb}, all closed subsets of $X^{\orb}$ are of the form $\pi(Z)$ with $Z \subseteq X$ $\GL$-stable and closed, and we find $\pi(\pi^{-1}(\pi(Z)))=\pi(Z)$. Similarly for open $\GL$-sets.
\end{proof}

Let $U$ be an open $\GL$-subset of $X$. We say that $U$ is \defn{$\GL$-quasi-compact} if any open cover of $U$ by open $\GL$-subsets admits a finite subcover. This is equivalent to $\pi(U)$ being quasi-compact by Proposition~\ref{prop:opens-of-Xorb}.

\begin{proposition} \label{prop:GL-quasi-cpt}
Let $Z$ be a closed $\GL$-subset of $X$. Then $U=X \setminus Z$ is $\GL$-quasi-compact if and only if $Z$ can be defined by a finitely $\GL$-generated ideal.
\end{proposition}

\begin{proof}
Suppose the orbits of $f_1, \ldots, f_r$ define an ideal cutting out $Z$. Let $D(f)$ be the usual distinguished open set defined by $f \ne 0$. Suppose that $\{V_{\alpha}\}_{\alpha \in I}$ is an open cover of $U$ by open $\GL$-subsets. Since $D(f_i) \subset U$, the $V_{\alpha}$'s cover $D(f_i)$. But $D(f_i)$ is an affine scheme, and thus quasi-compact, and so there exists a finite subset $J_i$ of $I$ such that $D(f_i) \subset \bigcup_{\alpha \in J_i} V_{\alpha}$. Since the $V_{\alpha}$'s are $\GL$-stable, we thus have $\bigcup_{g \in \GL} gD(f_i) \subset \bigcup_{\alpha \in J_i} V_{\alpha}$. Since $U=\bigcup_{i=1}^r \bigcup_{g \in \GL} gD(f_i)$, we see that $U=\bigcup_{\alpha \in J} V_{\alpha}$, where $J=\bigcup_{i=1}^r J_i$ is finite.

Now suppose that $U$ is $\GL$-quasi-compact. Let $\{f_{\alpha}\}_{\alpha \in I}$ define an ideal cutting out $Z$. Then $U=\bigcup_{\alpha \in I} \bigcup_{g \in \GL} gD(f_{\alpha})$ is a cover by open $\GL$-sets, and so there is a finite subset $J$ of $I$ such that $U=\bigcup_{\alpha \in J} \bigcup_{g \in \GL} gD(f_{\alpha})$. We thus see that $Z=\bigcap_{\alpha \in J} \bigcap_{g \in \GL} gV(f_{\alpha})$. Thus $\{f_{\alpha}\}_{\alpha \in J}$ is a finite set of elements that generated a $\GL$-ideal cutting out $Z$.
\end{proof}

\begin{proposition} \label{prop:Xorb-spectral}
The space $X^{\orb}$ is a spectral topological space.
\end{proposition}

\begin{proof}
Let $W$ be an irreducible closed subset of $X^{\orb}$.  By Proposition~ \ref{prop:XGL-Xorb}, its preimage $U$ in $X^\GL$ is closed and irreducible. Then $\ol{U}=\ol{O}_U$
is a closed irreducible $\GL$-set in the sober space $X$, hence has a unique
generic point $x$. Then $x \in X^\GL$, and $x$ lies in the closure of
$U$ in $X^\GL$, which is $U$ itself. So $W$ has a generic point, namely,
$\pi(x)$. The pre-image in $U$ of any other generic point of $W$ would
also be a generic point of $\ol{U}$, hence equal to $x$. This shows that $X$ is sober.

Let $U$ be an open $\GL$-stable subset of $X$, and let $Z=X \setminus U$. Let $Z=V(\fa)$ for some $\GL$-ideal $\fa$ of $\Gamma(X, \cO_X)$. Write $\fa=\sum_{\alpha \in I} \fa_{\alpha}$ where each $\fa_{\alpha}$ is finitely $\GL$-generated. Let $Z_{\alpha}=V(\fa_{\alpha})$, so that $Z=\bigcap_{\alpha \in I} Z_{\alpha}$. Let $U_{\alpha}=X \setminus Z_{\alpha}$, so that $U=\bigcup_{\alpha \in I} U_{\alpha}$. By Propositions~\ref{prop:opens-of-Xorb} and~\ref{prop:GL-quasi-cpt}, we see that $\pi(U)=\bigcup_{\alpha \in I} \pi(U_{\alpha})$ is a cover of $\pi(U)$ by quasi-compact open subsets. Thus the quasi-compact open subsets of $X^{\orb}$ form a basis for the topology. It follows immediately from the same propositions that the intersection of two quasi-compact open subsets is again quasi-compact. Thus $X$ is spectral.
\end{proof}

\begin{proposition}
If $X$ is an affine $\GL$-variety then $X^{\orb}$ is a noetherian topological space.
\end{proposition}

\begin{proof}
This follows from Theorem~\ref{thm:Noetherianity} and Proposition~\ref{prop:opens-of-Xorb}.
\end{proof}

\subsection{\texorpdfstring{$\GL$}{GL}-generic points} \label{ss:glgen}

We say that a point $x \in X$ is \defn{$\GL$-generic} if $\ol{O}_x=X$. By Proposition~\ref{prop:opens-of-Xorb}, this is equivalent to $\pi(x)$ being a generic point of $X^{\orb}$. An interesting feature of $\GL$-varieties is that there can be closed points that are $\GL$-generic.

\begin{example}
Let $\ulambda=[(1)^d]$, so that $\bV_{\ulambda}=\bV^{\oplus d}$. One can identify a point of $\bA^{\ulambda}$ with a $d$-tuple $(v_1, \ldots, v_d)$ where $v_i \in \bV^*$. Such a point is $\GL$-generic if and only if the $v_i$'s are linearly independent.
\end{example}

\begin{example}
A point of $\bA^{(2)}$ is $\GL$-generic if and only if it has infinite rank, when thought of as a symmetric bilinear form on $\bV$. This follows immediately from the description of the generalized orbits given in Example~\ref{ex:A2-orbits}.
\end{example}

In fact, it is not difficult to generalize the above examples:

\begin{proposition} \label{prop:generic-exists}
The space $\bA^{\ulambda}$ admits a $\GL$-generic $K$-point if and only if $\ulambda$ is pure.
\end{proposition}

\begin{proof}
First suppose that $\ulambda$ is not pure. We then have a $\GL$-invariant projection map $\bA^{\ulambda} \to \bA^1$, under which any $\GL$-generic point of $\bA^{\ulambda}$ maps to a generic point of $\bA^1$. Since $\bA^1$ does not have a generic $K$-point, it follows that $\bA^{\ulambda}$ does not have a $\GL$-generic $K$-point.

We now prove the converse direction. We start by considering
a special case. Let $T_d=\bV^{\otimes d}$ be the $d$th
tensor power representation and let $\bT^d=\Spec(\Sym(T_d))$
be the corresponding $\GL$-variety. Assume that $d>0$, so
that $\bT^d$ is pure. Let $\{e_i\}_{i \ge 1}$ be a basis of
$\bV$, so that we have a natural basis of $T_d$ given by
elements of the form $e_{i_1,\ldots,i_d}=e_{i_1} \otimes
\cdots \otimes e_{i_d}$. A $K$-point $v$ of $\bT^d$ is an element of the dual space $T_d^*$, and can be expressed as a (possibly infinite) sum $\sum_{i_1,\ldots,i_d} c_{i_1,\ldots,i_d} e_{i_1,\ldots,i_d}^*$ where the $c$'s belong to $K$.

Let $f \colon [d] \times \bN \to \bN$ be an injective function, and put $v=\sum_{n \ge 0} e^*_{f(1,n),\ldots,f(d,n)}$. For example, if $d=2$ one could take $v=e^*_{1,2}+e^*_{3,4}+\cdots$. This will be our $\GL$-generic point of $\bT^d$. Before proving this, we establish a property of its orbit closure. Let $T_{d,\le n}$ be the subspace of $T_d$ spanned by the vectors $e_{i_1,\ldots,i_d}$ with $i_1,\ldots,i_d \le n$. We claim that $T_{d,\le n}^*$ is contained in $\ol{O}_v$. Indeed, let $w \in T_{d,\le n}^*$ be given. Write $w=w_1+\cdots+w_r$ where each $w_k$ is a pure tensor. We can then find an endomorphism $m$ of $\bV$ that satisfies $m(e^*_{f(1,k),\ldots,f(d,k)})=w_k$ for $1 \le k \le r$ and $m(e^*_{f(1,k),\ldots,f(d,k)})=0$ for $k>r$. We thus see that $m(v)=w$. As remarked in \S \ref{ss:polyrep}, the monoid $\End(\bV)$ naturally acts on any polynomial representation. Thus $m$ acts on $\Sym(T_d)$, and leaves any subrepresentation stable; in particular, the ideal of $\ol{O}_v$ is stable. It follows that the action $m$ induces on $\bT^d$ carries $\ol{O}_v$ into itself. Since $m(v)=w$, we see that $w \in \ol{O}_v$.

It now follows that $\ol{O}_v=\bT^d$. Indeed, suppose that $h$ is some non-zero element of $\Sym(T_d)$. Let $n$ be such that $h$ belongs to $\Sym(T_{d,\le n})$. Since $h$ is non-zero, it does not vanish on all of $T^*_{d,\le n}$, and so we can find some $w \in T^*_{d,\le n}$ such that $h(w) \ne 0$. Since $w \in \ol{O}_v$, it follows that $h$ does not vanish on $\ol{O}_v$. Thus no element of $\Sym(T_d)$ vanishes on $\ol{O}_v$, and so $\ol{O}_v$ is Zariski dense in $\bT^d$. Since it is also closed, it must be the entire space.

Now consider the space $\bT^{d_1} \times \cdots \times \bT^{d_r}$, where each $d_i$ is positive. Let $v_i \in \bT^{d_i}(K)$ be as in the previous paragraph, but choose the $v_i$'s so that they use distinct basis vectors. In other words, if $v_i$ is associated to the function $f_i$, then the images of the $f_i$'s should be disjoint from one another. This ensures that we can move the $v_i$'s independently under the action $\End(\bV)$. Put $v=(v_1,\ldots,v_r)$. Then from our previous argument, we see that $\ol{O}_v$ contains $T_{d_1,\le n}^* \times \cdots \times T_{d_r,\le n}^*$ for all $n$. Thus $v$ is $\GL$-generic just as before.

Finally, suppose that $\ulambda$ is an arbitrary pure tuple.
We have $\bT^{d_1} \times \cdots \times \bT^{d_r} =
\bA^{\ulambda} \times \bA^{\umu}$ for appropriate $d_1,
\ldots, d_r>0$ and $\umu$. Taking a $\GL$-generic $K$-point
of $\bT^{d_1} \times \cdots \times \bT^{d_r}$, its projection to $\bA^{\ulambda}$ is a $\GL$-generic $K$-point.
\end{proof}

We have the following mild generalization of the above proposition:

\begin{proposition} \label{prop:generic-exists2}
Let $B$ be an irreducible variety and let $\ulambda$ be a pure tuple. Suppose that $B$ admits a generic point over the extension $\Omega/K$. Then $B \times \bA^{\ulambda}$ admits a $\GL$-generic $\Omega$-point.
\end{proposition}

\begin{proof}
Let $x$ be a generic $\Omega$-point of $B$ and let $v$ be a $\GL$-generic $K$-point of $\bA^{\ulambda}$. Then one easily sees that $(x,v)$ is a $\GL$-generic $\Omega$-point of $B \times \bA^{\ulambda}$. (For instance, base change along $x \colon \Spec(\Omega) \to B$ and then apply the previous proposition over the base field $\Omega$.)
\end{proof}

We observe that, in general, one does need to pass to an extension of $K$ to obtain $\GL$-generic points:

\begin{example} \label{ex:no-K-generic}
Let $X$ be the closed $\GL$-subvariety of $\bA^{[(1),(1)]}$
consisting of pairs of linearly dependent vectors. We have a
surjective map of $\GL$-varieties $\phi \colon \bA^2 \times
\bA^{(1)} \to X$ defined by $(\alpha,\beta,v) \mapsto
(\alpha v, \beta v)$. Moreover, $\phi$ is surjective on
$K$-points, and so the orbit closure of any non-zero
$K$-point is a copy of $\bA^{(1)}$ inside of $X$. Thus $X$
does not admit a $\GL$-generic $K$-point. The restriction of
$\phi$ to $\bA^1 \times \{1\} \times \bA^{(1)} \subset \bA^2
\times \bA^{(1)}$ (the line $y=1$ in the first factor) is dominant, and so, by Proposition~\ref{prop:generic-exists2}, we see that $X$ admits a $\GL$-generic $K(t)$-point.
\end{example}

\section{The embedding theorem} \label{s:embed}

The proof of the noetherianity theorem in \cite{draisma} relies on a
fundamental embedding theorem. Roughly speaking, this theorem states
that if $X$ is a proper closed $\GL$-subvariety of $\bA^{\ulambda}$
then an open subset of some shift of $X$ embeds into $\bA^{\umu}$ where
$\magn(\umu)<\magn(\ulambda)$. This provides an effective tool for
performing inductive arguments on $\GL$-varieties. We now formulate
two precise versions of this theorem.

\subsection{The shift operation} \label{ssec:Shift}

For any partition $\lambda$ and any $n \in \bZ_{\geq 0}$, we define $\Sh_n
(\bV_\lambda)$ to be the polynomial $\GL$-representation $\bS_\lambda(K^n
\oplus \bV)$, where $\GL$ acts trivially on $K^n$. The functoriality of
$\bS_\lambda$ implies that for $m,n \in \bZ_{\geq 0}$ we have $\Sh_{m+n}
(\bV_\lambda) \cong \Sh_m (\Sh_n(\bV_\lambda))$. We extend $\Sh_n$
(additively) to arbitrary polynomial $\GL$-representations, and as such
it is a functor from the category of polynomial representations to itself
that commutes with tensor products, symmetric powers, and indeed general
Schur functors.

As a consequence, if $R$ is a $\GL$-algebra, then the polynomial
$\GL$-representation $\Sh_n(R)$ is naturally a $\GL$-algebra, with the
product coming from the image of the multiplication homomorphism $R
\otimes R \to R$ under the shift functor. Note that $\Sh_n$ commutes
with tensor products of $\GL$-algebras.

If $X$ is an affine $\GL$-scheme with coordinate ring $R$, then we define $\Sh_n(X)$
to be the spectrum of the $\GL$-algebra $\Sh_n(R)$.  In particular,
$\Sh_n(\bA^{\ulambda})$ is the spectrum of $\Sh_n(\bR_{\ulambda})
\cong \Sym(\Sh_n(\bS_{\ulambda}))$.  Note that $\Sh_n$ commutes with
products of affine $\GL$-schemes.

Since we take $\bV$ and coordinate rings as the primary objects, the shift
operation is dual to that in \cite{draisma}. In particular, the natural
inclusion $\bV \to K^n \oplus \bV$ yields injective homomorphisms
$R \to \Sh_n(R)$ of $\GL$-algebras and surjective morphisms $\Sh_n(X)
\to X$ of affine $\GL$-schemes. Indeed, the natural projection $K^n
\oplus \bV \to \bV$ yields a left inverse $\Sh_n(R) \to R$ and a right
inverse $X \to \Sh_n(X)$, respectively. Similarly, note that the natural inclusion
$K^n\to K^n\oplus V$ yields the inclusion $\Sh_n(R)^{\GL}\to\Sh_n(R)$
of the subalgebra of $\GL$-invariant functions on $\Sh_n(X)$ into $\Sh_n(R)$.

Alternatively, the shift operation $\Sh_n$ may and will be understood as
follows. The linear shift isomorphism $\sigma_n: K^n \oplus \bV \to \bV$
that sends the standard basis of $K^n$ to $e_1,\ldots,e_n$ and the basis
vector $e_i$ of $\bV$ to $e_{i+n}$ is $\GL$-equivariant if we let $\GL$ act
on the right-hand side via the isomorphism $\GL \to G(n)$ that shifts an
infinite matrix to the south-east and inserts an $n \times n$-identity
matrix in the top left corner. Consequently, $\Sh_n(\bV_\lambda)$ is
isomorphic to the restriction of the $\GL$-representation $\bS_\lambda(\bV)$
to the subgroup $G(n)$, regarded as a polynomial $\GL$-representation
via the isomorphism $\GL \to G(n)$. We write $\sigma_n:\Sh_n(\bV_\lambda)
\to \bV_\lambda$ for this isomorphism, but stress that the
$\GL$-action on $\bV_\lambda$ is non-standard. Note that the linear inclusion
$K^n\to \bV$ sending the standard basis vectors to $e_1,\ldots,e_n$ yields
the inclusion $R^{G(n)}\to R$ of the subalgebra of $G(n)$-invariant functions on $X$ into $R$.

Similarly, for a $\GL$-algebra $R$, $\sigma_n$ induces an
isomorphism $\sigma_n: \Sh_n(R) \to R$ if we let $\GL$
act on $R$ via its isomorphism into $G(n)$, and for an affine $\GL$-scheme
$X$, $\sigma_n$ induces an isomorphism $\sigma_n: X \to \Sh_n(X)$
if we let $\GL$ act on $X$ via that isomorphism.

For a tuple $\ulambda$, we write $\sh_n(\ulambda)$
for the tuple such that $\Sh_n(\bA^{\ulambda}) =
\bA^{\sh_n(\ulambda)}$. The tuple $\sh_n(\ulambda)$
always contains $\ulambda$; we let
$\sh_{n,0}(\ulambda)$ be the complement. Thus
$\Sh_n(\bA^{\ulambda}) \cong \bA^{\ulambda} \times
\bA^{\sh_{n,0}(\ulambda)}$. For a partition $\lambda$,
we write $\sh_n(\lambda)$ and $\sh_{n,0}(\lambda)$ in place
of $\sh_n([\lambda])$ and $\sh_{n,0}([\lambda])$. We note
that every partition appearing in $\sh_{n,0}(\lambda)$ has
size strictly smaller than that of $\lambda$, and that
$\sh_{n,0}(\lambda)$ is the empty tuple if $\lambda$ is the
empty partition.

\subsection{First version}

Let $Y$ be an affine $\GL$-variety, and let $\lambda$ be
a non-empty partition. Let $f$ be a regular function on $Y \times
\bA^{\lambda}$ that does not factor through $Y$ and let $X \subset Y
\times \bA^{\lambda}$ be the closed $\GL$-subvariety defined by the
orbit of $f$. Let $h$ be a non-zero partial derivative of $f$ with
respect to some coordinate on $\bA^{\lambda}$. Let $n$ be such that
$f$ and the coordinate are $G(n)$-invariant; then so is $h$. Recall
that $\Sh_n(\bA^{\lambda})=\bA^{\lambda} \times \bA^{\umu}$, where
$\umu=\sh_{n,0}(\lambda)$ is a tuple whose constituents have
sizes smaller
than $\lambda$, so that $\magn(\umu)<\magn([\lambda])$.
Furthermore, we recall that $\Sh_n(Y \times \bA^{\lambda})$ can
be regarded as $Y \times \bA^{\lambda}$ but with $\GL$ acting via the
isomorphism $\GL \to G(n)$.  Via this latter interpretation, $h$ defines
a $\GL$-invariant function on $\Sh_n(Y \times \bA^{\lambda}) = \Sh_n(Y)
\times \bA^{\lambda} \times \bA^{\umu}$. In particular,
this invariant function does not
involve the coordinates on $\bA^\lambda$ (which have positive degree
since $\lambda$ is not the empty partition), and therefore
factors through the projection
\begin{displaymath}
\pi \colon \Sh_n(Y \times \bA^{\lambda}) \to \Sh_n(Y)
\times \bA^{\umu}.
\end{displaymath}
We can thus regard $h$ as a function on either the source or target. The embedding theorem is the following statement:

\begin{theorem} \label{thm:embed1}
The projection map $\pi$ restricts to a closed immersion
\begin{displaymath}
\pi \colon \Sh_n(X)[1/h] \to (\Sh_n(Y) \times \bA^{\umu})[1/h].
\end{displaymath}
\end{theorem}

\begin{proof}
This is proved in \cite[\S 2.9]{draisma}; we only clarify some
minor differences in the set-up.  The proof there, for the purpose
of an induction, assumes that $X$ is a closed $\GL$-subvariety of
$\bA^{\unu} \times \bA^{\lambda}$ with $|\lambda| \geq |\nu_i|$
for all $i$, but this inequality is not needed in the proof.  There,
$Y$ is just the closure of the image of $X$ in $\bA^{\unu}$, and
the requirement is that $X$ is not equal to $Y \times \bA^{\lambda}$;
here we have made sure of this fact by requiring that $f$ does not
factor through $Y$. Also, there, $X$ might be cut out from $Y \times
\bA^{\lambda}$ by further equations in addition to the orbit of $f$,
but of course the proof in particular applies to the case where $X$ is
defined by this single orbit. Finally, in \cite{draisma} it is required
that $h$ does not vanish identically on $X$. But if it does
(e.g., if $f$ is a square), then the conclusion of the theorem is trivial since $\Sh_n(X)[1/h]$ is empty.
\end{proof}

\subsection{Second version}

The above formulation of the embedding theorem is very precise, but rather cumbersome to apply (ensuring $h$ is non-zero can be subtle). We now formulate a simpler version which is usually sufficient:

\begin{theorem} \label{thm:embed2}
Let $Y$ be an affine $\GL$-variety, let $\lambda$ be a non-empty partition, and let $X$ be a closed $\GL$-subvariety of $Y \times \bA^{\lambda}$. Then one of the following two cases hold:
\begin{enumerate}
\item There exists a closed $\GL$-subvariety $Y_0 \subset Y$ such that $X=Y_0 \times \bA^{\lambda}$.
\item There exists an integer $n \ge 0$ and a $\GL$-invariant function $h$ on $\Sh_n(Y) \times \bA^{\umu}$ that does not vanish identically on $\Sh_n(X)$ such that the projection map
\begin{displaymath}
\pi \colon \Sh_n(X)[1/h] \to (\Sh_n(Y) \times \bA^{\umu})[1/h]
\end{displaymath}
is a closed immersion. (Here $\pi$ and $\umu=\sh_{n,0}(\lambda)$ are defined as before.)
\end{enumerate}
\end{theorem}

\begin{proof}
Suppose case (a) does not hold. Let $\phi \colon Y \times \bA^{\lambda} \to Y$ be the projection map and let $Y_0=\ol{\phi(X)}$. Since (a) does not hold, it follows that $X$ is a proper closed subset of $\phi^{-1}(Y_0)=Y_0 \times \bA^{\lambda}$. There is therefore some function on $Y \times \bA^{\lambda}$ that vanishes identically on $X$ but not on $Y_0 \times \bA^{\lambda}$. Of all such functions, choose one $f$ of minimal degree. Let $\{x_i\}$ be coordinates on $\bA^{\lambda}$, and write $f=\sum_{\alpha \in S} g_{\alpha} x^{\alpha}$ where $S$ is a finite set of exponent vectors, $x^{\alpha}$ denotes the monomial $x_1^{\alpha_1} x_2^{\alpha_2} \cdots$, and $g_{\alpha}$ is a function on $Y$. We can assume that no $g_{\alpha}$ vanishes identically on $Y_0$; indeed, we can simply subtract off such terms from $f$. Let $h$ be the partial derivative of $f$ with respect to some $x_i$ that appears in $f$. Then $h$ does not vanish identically on $Y_0 \times \bA^{\lambda}$, since its coefficients are among the $g_{\alpha}$'s. Since $h$ is non-vanishing on $Y_0 \times \bA^{\lambda}$ and has smaller degree than $f$, it does not vanish on $X$ by the minimality of the degree of $f$. Now, let $\wt{X}$ be the $\GL$-subvariety of $Y \times \bA^{\lambda}$ defined by the orbit of $f$; note that $X$ is a closed subvariety of $\wt{X}$. Again take $n$ such that $f$ and $x_i$ are $G(n)$-invariant. Then, by the previous version of the theorem, the map
\begin{displaymath}
\pi \colon \Sh_n(\wt{X})[1/h] \to (\Sh_n(Y) \times \bA^{\umu})[1/h]
\end{displaymath}
is a closed immersion. Since $\Sh_n(X)[1/h]$ is a closed subvariety of $\Sh_n(\wt{X})[1/h]$, we obtain the desired result.
\end{proof}

\begin{remark} \label{re:poschar}
In the proof of Theorem~\ref{thm:embed2}, characteristic zero is used in an essential manner: indeed, in positive characteristic, it is possible for all derivatives of all possible $f$'s to vanish identically. We do not know if Theorem~\ref{thm:embed2} remains valid in positive characteristic; we hope to return to this topic in the future.
\end{remark}

\subsection{A companion result}

The following result is often used in conjunction with the embedding theorem:

\begin{proposition} \label{prop:shiftmap}
Let $X$ be an affine $\GL$-variety over a field $K$ and let $h$ be a non-zero invariant function on $\Sh_n(X)$.
\begin{enumerate}
\item The natural inclusion $\iota: \bV \to K^n \oplus \bV$ induces a $\GL$-morphism $\phi \colon \Sh_n(X)[1/h] \to X$.
\item The image of $\phi$ contains a non-empty open $\GL$-subset $U$ of $X$.
\item Every point $x \in U$ admits an open neighborhood $V$ such that the map $\phi^{-1}(V) \to V$ has a section. (The open set and section are not necessarily $\GL$-stable.)
\item Let $\Omega/K$ be an extension field. Given any $\Omega$-point $x$ of $U$, there is an $\Omega$-point $y$ of $\Sh_n(X)[1/h]$ such that $\phi(y)=x$.
\end{enumerate}
\end{proposition}

\begin{proof}
(a) By the setup in \S \ref{ssec:Shift}, that inclusion yields an
inclusion at the level of coordinate rings and a surjection $\Sh_n X
\to X$; $\phi$ is the restriction to $\Sh_n(X)[1/h]$.

(b) Consider the linear map $\tau: K^n \oplus \bV \to \bV$ that is the identity
on $\bV$ and sends the standard basis of $K^n$ to $e_1,\ldots,e_n$. We
have $\tau \circ \iota=1_V$, and therefore $\tau$ induces a
(non-$\GL$-equivariant) section $\tau^*$
of the $\GL$-equivariant surjective morphism $\Sh_n(X) \to
X$ coming from $\iota$. Let $K^n \to K^n\oplus\bV$ be the natural inclusion and let $K^n\to\bV$ be the linear inclusion sending the standard basis to $e_1,\ldots,e_n$. Then, since $h$
is $\GL$-invariant and nonzero and the diagram
\begin{displaymath}
\xymatrix{
K^n\ar[r]\ar[rd] & K^n\oplus \bV \ar[d]^{\tau} \\
& \bV}
\end{displaymath}
commutes, the pull-back $\tau h$ of $h$ along $\tau^{*}$
is a nonzero $G(n)$-invariant function on $X$. It
follows that $X[1/\tau h]$ is in the image of
$\phi:\Sh_n(X)[1/h] \to X$. Since
$\im \phi$ is $\GL$-stable, the open $\GL$-set $U:=\GL \cdot
X[1/\tau h]$ is also contained in $\im \phi$.

(c) On $X[1/\tau h]$ we have constructed the section
$\tau^*$. On $g \cdot
X[1/h]$ for $g \in \GL$ we have the section $g \tau^* g^{-1}$.

(d) Use the section from (c).
\end{proof}

\section{The shift theorem and consequences} \label{s:shift}

\subsection{The shift theorem}

We now establish the important shift theorem. We first
introduce a piece of terminology. Suppose that $X$ is a
non-empty closed $\GL$-subvariety of $\bA^{\ulambda}$ for
some tuple $\ulambda$. Let $d$ be the maximal size of a
partition appearing in $\ulambda$, and decompose $\ulambda$
as $\umu \cup \unu$, where $\unu$ consists of the partitions
in $\ulambda$ of size $d$ and $\umu$ consists of the
remaining partitions. We say that $X$ is \defn{imprimitive}
if $d>0$ and $X$ has the form $Y \times \bA^{\unu}$ for some
closed $\GL$-subvariety $Y$ of $\bA^{\umu}$, and
\defn{primitive} otherwise. It is not difficult to see that
this notion is independent of coordinates, i.e., invariant
under the action of $\Aut(\bA^{\ulambda})$. Given any $X$,
one can find a decomposition $\ulambda=\umu \cup \unu$ (not
necessarily the same as the decomposition above) with $\unu$ pure such that $X=Y \times \bA^{\unu}$ for some primitive closed $\GL$-subvariety $Y$ of $\bA^{\umu}$.

\begin{theorem} \label{thm:shift}
Let $X$ be an affine $\GL$-variety.
\begin{enumerate}
\item There exists an integer $n \ge 0$ and a non-zero invariant function $h$ on $\Sh_n(X)$ such that $\Sh_n(X)[1/h]$ is isomorphic, as a $\GL$-variety, to $B \times \bA^{\ul{\kappa}}$ for some variety $B/K$ and some pure tuple $\ul{\kappa}$.
\item Suppose that $X$ is a primitive closed $\GL$-subvariety of $\bA^{\ulambda}$, and $\ulambda$ contains a non-empty partition. Then one can realize the situation in (a) with $\magn(\ul{\kappa})<\magn(\ulambda)$. In fact, if $d$ is the maximal size of a partition in $\ulambda$, then one can take $\ul{\kappa}$ so that $\magn(\ul{\kappa})_i=0$ for $i>d$ and $\magn(\ul{\kappa})_d<\magn(\ulambda)_d$.
\end{enumerate}
\end{theorem}

\begin{proof}
(a) Let $\ulambda$ be a pure tuple and let $d$ be the maximum size of a partition appearing in $\ulambda$. For a pure tuple $\ul{\kappa}$, we write $\ul{\kappa} \lessapprox \ulambda$ if $\magn(\ul{\kappa})_i=0$ for $i>d$ and $\magn(\ul{\kappa})_d \le \magn(\ulambda)_d$. Consider the following statement for a tuple $\ulambda$:
\begin{itemize}
\item[$S(\ulambda)$:] Given a closed $\GL$-subvariety $X$ of $\bA^{\ulambda}$ there exists $n \ge 0$ and a non-zero invariant function $h$ on $\Sh_n(X)$ such that $\Sh_n(X)[1/h]$ is isomorphic to $B \times \bA^{\ul{\kappa}}$ for some variety $B$ and pure tuple $\ul{\kappa} \lessapprox \ulambda$.
\end{itemize}
We prove the statement $S(\ulambda)$ for all $\ulambda$ by induction on the magnitude of $\ulambda$. This will establish the theorem.

Let $\ulambda$ be given, and suppose that $S(\umu)$ holds
for all $\umu$ with magnitude strictly less than that of
$\ulambda$. We prove $S(\ulambda)$. If $\ulambda$ consists
of empty partitions the statement is obvious, so assume this
is not the case. Let $X \subseteq \bA^{\ulambda}$ be given.
Let $d>0$ be the maximal size of a partition in $\ulambda$ and let $\nu$ be a partition in $\ulambda$ of size $d$. Let $\umu$ be the tuple obtained by removing $\nu$ from $\ulambda$, so that $\ulambda=\umu \cup \nu$. Let $Y=\bA^{\umu}$, so that $X$ is a subvariety of $Y \times \bA^{\nu}$. We now apply Theorem~\ref{thm:embed2}; we consider the two possible cases in turn.

In the first case, there is a closed $\GL$-subvariety $Y_0$ of $Y$ such
that $X=Y_0 \times \bA^{\nu}$. Since $\magn(\umu)<\magn(\ulambda)$,
the statement $S(\umu)$ holds, and so $\Sh_n(Y_0)[1/h] \cong B \times
\bA^{\ul{\tau}}$ for some variety $B$, tuple $\ul{\tau} \lessapprox \umu$, integer $n$,
and non-zero invariant function $h$ on $\Sh_n(Y_0)$. We thus find
\begin{displaymath}
\Sh_n(X)[1/h] = \Sh_n(Y_0)[1/h] \times \Sh_n(\bA^{\nu}) \cong B \times \bA^{\ul{\tau} \cup \sh_n(\nu)}.
\end{displaymath}
We note that $\ul{\tau} \cup \sh_n(\nu) \lessapprox \ulambda$, since all partitions in $\sh_n(\nu)$ other than $\nu$ have size $<d$. Hence the conclusion of the theorem holds for $X$.

In the second case, there exists an integer $n \ge 0$ and a non-zero
invariant function $h$ on $\Sh_n(Y) \times \bA^{\sh_{n,0}(\nu)}$ that
does not vanish identically on $\Sh_n(X)$ such that the map
\begin{displaymath}
\pi \colon \Sh_n(X)[1/h] \to (\Sh_n(Y) \times \bA^{\sh_{n,0}(\nu)})[1/h]
\end{displaymath}
is a closed immersion. Now, $\Sh_n(Y) \times \bA^{\sh_{n,0}(\nu)}=\bA^{\ul{\sigma}}$ where $\ul{\sigma}=\sh_n(\umu) \cup \sh_{n,0}(\nu)$. Since $h$ is a $\GL$-invariant function on $\bA^{\ul{\sigma}}$, we can realize $\bA^{\ul{\sigma}}[1/h]$ as a closed $\GL$-subvariety of $\bA^1 \times \bA^{\ul{\sigma}}=\bA^{\ul{\tau}}$, where $\ul{\tau}=[\emptyset] \cup \ul{\sigma}$. We have $\magn(\ul{\tau})<\magn(\ul{\lambda})$; in fact, $\magn(\ul{\tau})_d<\magn(\ulambda)_d$. Thus the statement $S(\ul{\tau})$ holds by the induction hypothesis. Hence, there exists an integer $m$ and a non-zero invariant function $h'$ on $\Sh_m(\Sh_n(X)[1/h])=\Sh_{n+m}(X)[1/h]$ such that $\Sh_{n+m}(X)[1/hh']$ is isomorphic to $B \times \bA^{\ul{\kappa}}$ with $\ul{\kappa} \lessapprox \ul{\tau}$. Note that $\ul{\kappa} \lessapprox \ulambda$; in fact, we have $\magn(\ul{\kappa})_d<\magn(\ulambda)_d$. Thus $S(\ulambda)$ holds.

(b) We maintain the notation from part (a). After possibly
making a linear change of coordinates, we can choose the
decomposition $\ulambda=\umu \cup \nu$ such that $X$ does
not have the form $Y_0 \times \bA^{\nu}$; this follows from
the assumption that $X$ is primitive. Following the proof in part~(a), we find ourselves in the second case, and so, as noted, we can realize the conclusion with $\magn(\ul{\kappa})_d<\magn(\ulambda)_d$.
\end{proof}

\begin{remark}
The variety $B$ in Theorem~\ref{thm:shift} is in fact $(X\{K^n\})[1/h]$, as one sees by evaluating on $K^0$. We can therefore deduce certain properties about $B$ from those of $X$. For example, if $X$ is irreducible then so is $B$.
\end{remark}

The shift theorem has the following consequence for $\GL$-algebras:

\begin{corollary}
Let $R$ be a reduced $\GL$-algebra that is finitely $\GL$-generated over a field~$K$. Then there is a non-zero element $h \in R$ such that, ignoring the $\GL$-actions, $R[1/h]$ is isomorphic to a polynomial algebra over a finitely generated $K$-algebra.
\end{corollary}

\subsection{Unirationality}

We now obtain the important unirationality theorem:

\begin{theorem} \label{thm:uni}
Let $X$ be an irreducible affine $\GL$-variety. Then there exists a dominant $\GL$-morphism $\phi \colon B \times \bA^{\umu} \to X$ for some irreducible variety $B/K$ and pure tuple $\ulambda$. Moreover:
\begin{enumerate}
\item The image of $\phi$ contains a non-empty open $\GL$-subset $U$ of $X$.
\item Every point $x \in U$ admits an open neighborhood $V$ such that the map $\phi^{-1}(V) \to V$ admits a section.
\item Let $\Omega/K$ be an extension. Then any $\Omega$-point of $U$ lifts to an $\Omega$-point of $B \times \bA^{\ulambda}$.
\item Suppose $X$ is a closed $\GL$-subvariety of $\bA^r \times \bA^{\ulambda}$, for some $r \ge 0$ and pure tuple $\ulambda$. Then either $X=C \times \bA^{\ulambda}$ for some closed subvariety $C$ of $\bA^r$, or else one can find $\phi$ as above with $\magn(\umu)<\magn(\ulambda)$.
\end{enumerate}
\end{theorem}

\begin{proof}
By the shift theorem (Theorem~\ref{thm:shift}), we have an isomorphism $\Sh_n(X)[1/h] \cong B \times \bA^{\ulambda}$ for some $n$, $h$, $B$, and $\ulambda$. Parts (a), (b) and (c) now follow from Proposition~\ref{prop:shiftmap} applied to $\Sh_n(X)[1/h] \to X$. This morphism is dominant by (a) since $X$ is irreducible.

We now handle part~(d). Write $X=Y \times \bA^{\ul{\tau}}$ where $\ulambda=\ul{\sigma} \cup \ul{\tau}$ and $Y$ is a primitive subvariety of $\bA^r \times \bA^{\ul{\sigma}}$. If $\ul{\sigma}$ is the empty tuple then $C=Y$ is a subvariety of $\bA^r$, and $X=C \times \bA^{\ulambda}$. Thus assume this is not the case. By the shift theorem, we have an isomorphism $\Sh_n(Y)[1/h] \cong B \times \bA^{\ul{\kappa}}$ for some $n$, $h$, $B$, and $\ul{\kappa}$ with $\magn(\ul{\kappa})<\magn(\ul{\sigma})$. As noted in the first paragraph, $\psi \colon \Sh_n(Y)[1/h] \to Y$ satisfies the analogues of (a)--(c) and is dominant. Now consider the map $\bA^{\ul{\tau}} \times \Sh_n(Y)[1/h] \to \bA^{\ul{\tau}} \times Y$ given by $\id \times \psi$; call this $\phi$. The domain of $\phi$ is identified with $B \times \bA^{\ul{\kappa} \cup \ul{\tau}}$ and the target with $X$. Clearly, $\phi$ satisfies (a)--(c) and is dominant. Moreover, $\umu=\ul{\kappa} \cup \ul{\tau}$ satisfies $\magn(\umu)<\magn(\ulambda)$. Thus (d) holds.
\end{proof}

\begin{corollary} \label{cor:fg-GL-gen}
Let $X$ be an irreducible affine $\GL$-variety. Then $X$ admits a $\GL$-generic point defined over a finitely generated extension of $K$.
\end{corollary}

\begin{proof}
Let $\phi \colon B \times \bA^{\ulambda} \to X$ be a dominant map as in Theorem~\ref{thm:uni}. Then $B \times \bA^{\ulambda}$ admits a $\GL$-generic point $x$ defined over a finitely generated extension of $K$ by Proposition~ \ref{prop:generic-exists2}, and $\phi(x)$ is a $\GL$-generic point of $x$.
\end{proof}

Combining the above result with a noetherian induction argument, we can obtain a finite collection of maps from varieties of the form $B \times \bA^{\umu}$ that are jointly surjective. Precisely:

\begin{proposition} \label{prop:Asurj}
Let $X$ be a closed $\GL$-subvariety of $\bA^{\ulambda}$ for some tuple $\ulambda$. Then we can find a finite family $\{ \phi_i \colon B_i \times \bA^{\ul{\kappa}_i} \to X \}_{1 \le i \le n}$ of morphisms $\GL$-varieties such that:
\begin{enumerate}
\item $B_i$ is an irreducible variety and $\magn(\ul{\kappa}_i) \le \magn(\ulambda)$.
\item For any extension $\Omega/K$, the $\phi_i$'s are jointly surjective on $\Omega$-points.
\end{enumerate}
\end{proposition}

\begin{proof}
By noetherian induction, we can assume that the result holds for every proper closed $\GL$-subvariety of $X$. Applying Theorem~\ref{thm:uni} to an irreducible component of $X$, we have a map $\phi_1 \colon B_1 \times \bA^{\ul{\kappa}_1} \to X$ for an irreducible variety $B_1$ such that:
\begin{itemize}
\item the image of $\phi_1$ contains a non-empty open $\GL$-subset $U$ of $X$,
\item every $\Omega$-point of $U$ lifts to an $\Omega$-point under $\phi_1$, and
\item $\magn(\ul{\kappa}_1) \le \magn(\ulambda)$.
\end{itemize}
To see the third point, we apply Theorem~\ref{thm:uni}(d); in the first case we take $\ul{\kappa}_1=\ulambda$, and in the second we actually get $\magn(\ul{\kappa}_1)<\magn(\ulambda)$. Let $Z$ be the complement of $U$. By the inductive hypothesis, we can find maps $\{\phi_i \colon B_i \times \bA^{\ul{\kappa}_i} \to Z\}_{2 \le i \le n}$ such that $\magn(\ul{\kappa}_i) \le \magn(\ulambda)$ and the $\phi_i$'s are jointly surjective on $\Omega$-points for all $\Omega$. We thus see that $\{\phi_i\}_{1 \le i \le n}$ satisfies the requisite conditions.
\end{proof}

Note that by taking $B=B_1 \amalg \cdots \amalg B_n$ and
$\ul{\kappa}=\ul{\kappa}_1 \cup \cdots \cup \ul{\kappa}_n$, one can
create a single morphism $\phi \colon B \times \bA^{\ul{\kappa}} \to X$
that is surjective. However, even if $X$ is irreducible, the $B$ produced
by this method may well be reducible. In our follow-up (see \S \ref{ss:further}), with a considerable amount of extra work,
we show that one can actually find an irreducible $B$ in this case.

\subsection{Invariant function fields}

Let $X$ be an irreducible affine $\GL$-variety. We define the \defn{invariant function field} of $X$ to be $K(X)^{\GL}$, i.e., the fixed field of $\GL$ acting on the usual function field $K(X)$.

\begin{proposition} \label{prop:Ainv}
Let $B/K$ be an irreducible variety with function field $E$ and let $\ulambda$ be a pure tuple. Then the invariant function field of $B \times \bA^{\ulambda}$ is $E$.
\end{proposition}

\begin{proof}
The function field $E(\bA^{\ulambda})$ is the field of rational functions in the coordinate variables on $\bA^{\ulambda}$. No non-constant rational function is invariant under $\GL$; in fact, none is invariant under the infinite symmetric group, which is a subgroup of $\GL$. We thus see that $E(\bA^{\ulambda})^{\GL}=E$, which establishes the result.
\end{proof}

\begin{proposition} \label{prop:inv}
Let $X$ be an irreducible $\GL$-variety over a field $K$. Then the invariant function field $K(X)^{\GL}$ is a finitely generated extension of $K$.
\end{proposition}

\begin{proof}
By the unirationality theorem (Theorem~\ref{thm:uni}), we have a dominant $\GL$-morphism $B \times \bA^{\ulambda} \to X$ with $B$ an irreducible variety over $K$ and $\ulambda$ a pure tuple. We thus have $K(X)^{\GL} \subset K(B \times \bA^{\ulambda})^{\GL}=E$, where $E$ is the function field of $B$. Thus $K(X)^{\GL}$ is contained in a finitely generated extension of $K$, and is therefore itself finitely generated.
\end{proof}

\begin{remark}
We can also consider the invariants of $K(X)$ under the subgroups $G(n)$. This leads to a tower of fields
\begin{displaymath}
K(X)^{\GL}=K(X)^{G(0)} \subseteq K(X)^{G(1)} \subseteq K(X)^{G(2)} \subseteq \cdots
\end{displaymath}
each of which is a finitely generated extension of $K$. The shift theorem implies that for $m \gg 0$, $K(X)^{G(n)}$ is a rational function field over $K(X)^{G(m)}$ for each $n \geq m$.
\end{remark}

\subsection{Dimension functions}

Let $X$ be an affine $\GL$-variety over $K$. We define the
\defn{dimension function} of $X$ to be the function
$\delta_X \colon \bN \to \bN$ given by $\delta_X(d)=\dim X\{K^d\}$.

\begin{proposition} \label{prop:dim}
The dimension function $\delta_X$ is eventually polynomial. That is, there is a polynomial $p$ with rational coefficients such that $\delta_X(d)=p(d)$ for all $d \gg 0$.
\end{proposition}

\begin{proof}
First suppose $X$ is irreducible. By the shift theorem, we
have $\Sh_n(X)[1/h] \cong B \times \bA^{\ulambda}$ for some
integer $n$. We have $\Sh_n(X)\{K^d\}=X\{K^{n+d}\}$, and
so $\delta_{\Sh_n(X)}(d)=\delta_X(d+n)$. Of course,
inverting $h$ does not affect the dimension since $X$ is
irreducible. We thus see that
$\delta_X(d+n)=\dim{B}+\delta_{\bA^{\ulambda}}(n)$. Now,
$\delta_{\bA^{\ulambda}}(d)$ is just the dimension of the
representation $\bigoplus_{i=1}^r \bS_{\lambda_i}(K^d)$,
where $\lambda=[\lambda_1, \ldots, \lambda_r]$, which is
known to be polynomial in $d$ of degree the maximum among
the sizes of the $\lambda_i$. Thus $\delta_{\bA^{\ulambda}}$ is polynomial. The result follows; in fact, we see that $d \mapsto \delta_X(d)$ is polynomial for $d \ge n$.

We now treat the general case. Let $Y_1, \ldots, Y_r$ be the irreducible components of $X$. Then $\delta_X(d)=\max(\delta_{Y_1}(d), \ldots, \delta_{Y_r}(d))$. Since each $\delta_{Y_i}$ is eventually polynomial by the previous paragraph, it follows that $\delta_X$ is too.
\end{proof}

\section{Some properties of affine spaces} \label{s:affine}

\subsection{Dominant maps to \texorpdfstring{$\bA^{\lambda}$}{Alambda}}

It is very easy to write down many examples of dominant maps $\bA^n \to \bA^m$ in standard algebraic geometry. By contrast, we now show that essentially the only dominant maps $\bA^{\umu} \to \bA^{\ulambda}$ are projection maps when $\ulambda$ and $\umu$ are pure.

\begin{proposition} \label{prop:Adom}
Let $B$ and $C$ be irreducible varieties, let $\umu$ and $\ulambda$ be pure tuples, and let $\phi \colon B \times \bA^{\umu} \to C \times \bA^{\ulambda}$ be a dominant morphism of $\GL$-varieties. Then $\ulambda \subset \umu$. Moreover, letting $\phi_0 \colon B \to C$ be the map induced by $\phi$, there exist non-empty open affine subvarieties $U \subset B$ and $V \subset C$ such that $\phi_0(U) \subset V$ and the restriction of $\phi$ to $U \times \bA^{\umu}$ factors as
\begin{displaymath}
\xymatrix@=1.5cm{
U \times \bA^{\umu} \ar[r]^{\sigma} & U \times \bA^{\umu} \ar[r]^-{\id \times \pi} & U \times \bA^{\ulambda} \ar[r]^-{\phi_0 \times \id} & V \times \bA^{\ulambda} }
\end{displaymath}
where $\sigma$ is an $U$-automorphism of $U \times \bA^{\umu}$ and $\pi$ is the projection map. In fact, it is possible to choose $U$ and $V$ such that $\phi$ maps $U \times \bA^{\umu}$ surjectively onto $V \times \bA^{\ulambda}$.
\end{proposition}

We require a few lemmas before giving the proof. Recall that $\bR_{\ulambda}=\Sym(\bV_{\ulambda})$ is the coordinate ring of $\bA^{\ulambda}$. We write $\bR_{\ulambda,+}$ for the ideal of positive degree elements of $\bR_{\ulambda}$.

\begin{lemma} \label{lem:Adom-1}
Let $E \subset F$ be field extensions of $K$, let $\ulambda$ and $\umu$ be pure tuples, and let $f \colon E \otimes \bR_{\ulambda} \to F \otimes \bR_{\umu}$ be an injection of $\GL$-algebras over $E$. Then $\dim \bS_{\ulambda}(K^n) \le c+\dim \bS_{\umu}(K^n)$ for some $c \in \bN$ and all $n \ge 0$.
\end{lemma}

\begin{proof}
The restriction of $f$ to $\bR_{\ulambda}$ is still an injection of $\GL$-algebras, so we may as well assume $E=K$. Then $f$ induces a map of representations $\bV_{\ulambda} \to F \otimes \bR_{\umu}$. The image of this map is necessarily contained in $U \otimes \bR_{\umu}$ for some finite dimensional $K$-subspace $U$ of $F$. It follows that the image of $f$ is contained in $B \otimes \bR_{\umu}$ where $B$ is the $K$-subalgebra of $F$ generated by $U$.

Evaluating the polynomial functors on $K^n$, we see that the map $\bR_{\ulambda}\{K^n\} \to B \otimes \bR_{\umu}\{K^n\}$ is an injection of finitely generated $K$-algebras. It follows that the Krull dimension of the source, which is $\dim{\bS_{\ulambda}(K^n)}$, is at most that of the target, which is $c+\dim{\bS_{\umu}(K^n)}$, where $c=\dim(B)$. The result follows.
\end{proof}

\begin{lemma} \label{lem:Adom-2}
Let $E \subset F$ be field extensions of $K$, let $\ulambda$ and $\umu$ be pure tuples, and let $f \colon E \otimes \bR_{\ulambda} \to F \otimes \bR_{\umu}$ be an injection of $\GL$-algebras over $E$. Then the induced map
\begin{displaymath}
\ol{f} \colon F \otimes \bR_{\ulambda,+}/\bR_{\ulambda,+}^2 \to F \otimes \bR_{\umu,+}/\bR_{\umu,+}^2
\end{displaymath}
is an injection. (Note: it is $F$ on both the left and right sides above.)
\end{lemma}

\begin{proof}
Suppose that $\ol{f}$ fails to be injective on the $\bV_{\nu}$ isotypic piece. Let $n$ be the multiplicity of $\nu$ in $\ulambda$, and let $m<n$ be the multiplicity of $\bV_{\nu}$ in the image of $\ol{f}$. We thus see that the $n$ ``$\nu$ variables'' in $E \otimes \bR_{\ulambda}$ are mapped to elements that make use of only $m$ ``$\nu$ variables'' in $F \otimes \bR_{\umu}$ and lower degree elements. In other words, $f$ induces an injection $E \otimes \bR_{\nu^n} \to F \otimes \bR_{\umu' \cup \nu^m}$ where $\umu'$ consists of the partitions in $\umu$ of size $<\vert \nu \vert$, and $\nu^n$ denotes the $n$-tuple $[\nu, \ldots, \nu]$. This contradicts the Lemma~\ref{lem:Adom-1}
\end{proof}

\begin{proof}[Proof of Proposition~\ref{prop:Adom}]
Let $B=\Spec(R)$ and $C=\Spec(S)$ where $R$ and $S$ are
finitely generated integral $K$-algebras. Since $\phi$ is
dominant, it follows that $\phi_0 \colon B \to C$ is
dominant, and that the corresponding $K$-algebra homomorphism $S \to R$ is injective. Let $E=\Frac(S)$ and $F=\Frac(R)$; we regard $E$ as a subfield of $F$. Then $\phi$ induces an injection $f \colon E \otimes \bR_{\ulambda} \to F \otimes \bR_{\umu}$ of $\GL$-algebras over $E$. Lemma~\ref{lem:Adom-2} implies that the multiplicity of $\bV_{\nu}$ in $\bV_{\ulambda}$ is at most the multiplicity of $\bV_{\nu}$ in $\bV_{\umu}$, for any partition $\nu$; note that $\bR_{\ulambda,+}/\bR_{\ulambda,+}^2 \cong \bV_{\ulambda}$. Thus $\ulambda \subset \umu$. Write $\umu=\ulambda \cup \unu$.

Let $\ol{f}$ be as in Lemma~\ref{lem:Adom-2}. The image of $\ol{f}$ is isomorphic to $F \otimes \bV_{\ulambda}$, and thus has a complementary representation isomorphic to $F \otimes \bV_{\unu}$. Choose an $F$-linear injection of representations $\ol{g} \colon F \otimes \bV_{\unu} \to F \otimes \bR_{\umu}$ such that the images of $\ol{f}$ and $\ol{g}$ in $F \otimes \bR_{\umu,+}/\bR_{\umu,+}^2$ form complementary subrepresentations. Let $g \colon F \otimes \bR_{\unu} \to F \otimes \bR_{\umu}$ be the algebra homomorphism induced by $g$. Consider the following sequence of maps
\begin{displaymath}
\xymatrix@C=3em{
E \otimes \bR_{\ulambda} \ar[r] &
F \otimes \bR_{\ulambda} \ar[r] &
F \otimes \bR_{\ulambda} \otimes \bR_{\unu} \ar[r]^-{f \otimes g} \ar[r] &
F \otimes \bR_{\umu} }
\end{displaymath}
where the first two maps are the inclusions. Since the right map is an isomorphism modulo the maximal ideal, Nakayama's lemma implies that it is an isomorphism. We identify the third algebra with $F \otimes \bR_{\umu}$. The second map is then the standard inclusion coming from $\ulambda \subset \umu$, and the right map is an automorphism.

Now, consider the automorphism $\ol{f} \otimes \ol{g}$ of $F \otimes \bR_{\umu,+}/\bR_{\umu,+}^2$. For each partition $\nu$ in $\umu$, let $A_{\nu}$ be the matrix of $\ol{f} \otimes \ol{g}$ on the $\nu$-multiplicity space in some basis. Each $A_{\nu}$ is a finite matrix, and there are only finitely many of them. It follows that we can find a non-zero $h \in R$ such that each of our bases is defined over $R[1/h]$, the determinants of all these matrices are units in $R[1/h]$, and $f$ and $g$ are defined over $R[1/h]$. We thus see that $f \otimes g$ induces an automorphism of $R[1/h] \otimes \bR_{\umu}$. We thus get maps
\begin{displaymath}
\xymatrix@C=3em{
S \otimes \bR_{\ulambda} \ar[r] &
R[1/h] \otimes \bR_{\ulambda} \ar[r] &
R[1/h] \otimes \bR_{\umu} \ar[r]^-{f \otimes g} \ar[r] &
R[1/h] \otimes \bR_{\umu} }
\end{displaymath}
where, again, the first map comes from the inclusion $S \subset R[1/h]$ and the second from the inclusion $\ulambda \subset \umu$. Let $U=\Spec(R[1/h])$, an open affine of $B$. The above maps translate to maps
\begin{displaymath}
\xymatrix@C=4em{
U \times \bA^{\umu} \ar[r]^-{\sigma} &
U \times \bA^{\umu} \ar[r]^-{\id \times \pi} &
U \times \bA^{\ulambda} \ar[r]^-{\phi_0 \times \id} &
C \times \bA^{\ulambda} }
\end{displaymath}
where $\sigma$ is a $U$-automorphism and $\pi$ is the projection map. This gives the main statement of the proposition, with $V=C$.

To prove the final part, about surjectivity, simply take $V$ to be a non-empty open subset of $C$ contained in $\phi_0(U)$ (which exists by Chevalley's theorem) and replace $U$ with $U \cap \phi_0^{-1}(V)$.
\end{proof}

\subsection{A lifting result} \label{ss:Alifting}

Using Proposition~\ref{prop:Adom}, we obtain a very useful lifting result:

\begin{proposition} \label{prop:lift}
Let $\phi \colon B \times \bA^{\ulambda} \to X$ and $\psi \colon Y \to X$ be morphisms of $\GL$-varieties, with $B$ an irreducible variety and $\ulambda$ a pure tuple. Let $x$ be the generic point of $B \times \bA^{\ulambda}$, and suppose that $\phi(x) \in \im(\psi)$. We can then find a commutative diagram
\begin{displaymath}
\xymatrix@C=4em{
C \times \bA^{\ulambda} \ar[r]^-{\alpha \times \id} \ar[d] &
B \times \bA^{\ulambda} \ar[d]^{\phi} \\
Y \ar[r]^{\psi} & X }
\end{displaymath}
where $C$ is an irreducible variety and $\alpha \colon C \to B$ is a quasi-finite dominant morphism.
\end{proposition}

\begin{proof}
Let $Z$ be the reduced subscheme of the fiber product of $(B \times \bA^{\ulambda}) \times_X Y$. This is a $\GL$-variety, and the condition $\phi(x) \in \im(\psi)$ ensures that the projection map $Z \to B \times \bA^{\ulambda}$ is dominant. Let $Z'$ be an irreducible component of $Z$ mapping dominantly to $B \times \bA^{\ulambda}$, and, applying Theorem~\ref{thm:uni}, choose a dominant morphism $C \times \bA^{\umu} \to Z'$ with $C$ an irreducible variety. We thus have a commutative diagram
\begin{displaymath}
\xymatrix@C=4em{
C \times \bA^{\umu} \ar[r] \ar[d] &
B \times \bA^{\ulambda} \ar[d]^{\phi} \\
Y \ar[r]^{\psi} & X }
\end{displaymath}
where the top morphism is dominant. We now apply Proposition~\ref{prop:Adom}. We find that $\ulambda \subset \umu$. Furthemore, replacing $C$ with a dense open and applying an automorphism of $C \times \bA^{\umu}$, we can assume that the top map has the form $\alpha \times \pi$ where $\alpha \colon C \to B$ is a morphism, necessarily dominant, and $\pi \colon \bA^{\umu} \to \bA^{\ulambda}$ is the projection map. We now simply restrict to $C \times \bA^{\ulambda} \subset C \times \bA^{\umu}$. This yields the stated result, except for quasi-finiteness. For this, we replace $C$ with an appropriate closed subvariety.
\end{proof}

We note an elegant corollary:

\begin{corollary}
Assume $K$ is algebraically closed. Suppose that $\phi \colon Y \to \bA^{\ulambda}$ is a dominant morphism of $\GL$-varieties, with $Y$ irreducible and $\ulambda$ pure. Then $\phi$ admits a $\GL$-equivariant section. In particular, $\phi$ is surjective.
\end{corollary}

\begin{proof}
Apply the proposition with $B=\Spec(K)$ and $X=\bA^{\ulambda}$. Then $C$ is necessarily $\Spec(K)$ (since $K$ is algebraically closed), and the map $C \times \bA^{\ulambda} \to Y$ is the sought after section.
\end{proof}

\subsection{Extending maps}

Let $Z$ be a closed subscheme of an affine scheme $X$. Then any morphism $Z \to \bA^n$ can be extended to a morphism $X \to \bA^n$. We now show that this property also holds for the $\GL$-analogs of affine spaces.

\begin{proposition}
Let $X$ be an affine $\GL$-scheme, let $Z$ be a closed
$\GL$-subscheme of $X$, and let $\phi \colon Z \to \bA^{\ulambda}$ be a morphism of $\GL$-schemes, with $\ulambda$ an arbitrary tuple. Then there exists a morphism of $\GL$-schemes $\psi \colon X \to \bA^{\ulambda}$ extending $\phi$.
\end{proposition}

\begin{proof}
Write $X=\Spec(A)$ and $Z=\Spec(A/I)$, where $A$ is a $\GL$-algebra and $I$ is a $\GL$-ideal of $A$. Write $\bA^{\ulambda}=\Spec(\bR_{\lambda})$ as usual, with $\bR_{\ulambda}=\Sym(\bV_{\ulambda})$. Consider the following diagram
\begin{displaymath}
\xymatrix{
\bR_{\ulambda} \ar@{..>}[r]^{\psi^*} \ar[rd]_{\phi^*} & A \ar[d] \\
& A/I }
\end{displaymath}
We must find a morphism $\psi^*$ of $\GL$-algebras that makes the diagram commute. The restriction of $\phi^*$ to $\bV_{\ulambda} \subset \bR_{\ulambda}$ is a map $\bV_{\ulambda} \to A/I$ of representations. It lifts to a map of representations $\bV_{\ulambda} \to A$, since polynomial representations of $\GL$ are semi-simple. This map induces a map of $\GL$-algebras $\psi^* \colon \bR_{\ulambda} \to A$, by the universal property of $\Sym$, which satisfies the necessary properties.
\end{proof}

\begin{remark}
This proposition does not remain valid in positive characteristic: for
instance, over a field $K$ of characteristic $3$, the space of infinite
cubics contains as a proper closed subset the linear space spanned by
third powers of linear forms, which is stable under the infinite general
linear group. The identity map from that space to itself does not extend
to an equivariant morphism defined on the entire space of infinite
cubics.
\end{remark}

\section{The decomposition theorem} \label{s:decomp}

\emph{Throughout this section, ``$\GL$-variety'' means ``quasi-affine $\GL$-variety'' by default.}

\subsection{Elementary varieties and morphisms}

We say that a $\GL$-variety $X$ is \defn{elementary} if it is isomorphic to a $\GL$-variety of the form $B \times \bA^{\ulambda}$ for some irreducible affine variety $B$ and tuple $\ulambda$; this implies that $X$ is irreducible and affine. Let $\phi \colon X \to Y$ be a morphism of $\GL$-varieties. We say that $\phi$ is \defn{elementary} if there exists a commutative diagram
\begin{equation} \label{eq:Diagram}
\begin{gathered}
\xymatrix@C=5em{
X \ar[r]^{\phi} \ar[d]_-i & Y \ar[d]^-j \\
B \times \bA^{\ulambda} \ar[r]^{\psi \times \pi} & C \times \bA^{\umu} }
\end{gathered}
\end{equation}
where $i$ and $j$ are isomorphisms of $\GL$-varieties, $B$ and $C$ are irreducible affine varieties, $\psi \colon B \to C$ is a surjective morphism of varieties, $\ulambda$ and $\umu$ are pure tuples with $\umu \subset \ulambda$, and $\pi \colon \bA^{\ulambda} \to \bA^{\umu}$ is the projection map. This implies that $X$ and $Y$ are elementary and that $\phi$ is surjective. Note that a $\GL$-variety is elementary if and only if its identity map is. We say that a $G(n)$-variety, or morphism of $G(n)$-varieties, is \defn{$G(n)$-elementary} if it is so after identifying $G(n)$ with $\GL$.

We say that a quasi-affine $\GL$-variety $X$ is \defn{locally elementary at $x \in X$} if there exists $m$ and a $G(m)$-stable open neighborhood of $x$ that is $G(m)$-elementary; we say that $X$ is \defn{locally elementary} if it is so at all points. Let $\phi \colon X \to Y$ be a morphism of $\GL$-varieties. We say that $\phi$ is \defn{locally elementary at $x \in X$} if there exists $m$ and $G(m)$-stable open neighborhoods $U$ of $x$ and $V$ of $\phi(x)$ such that $\phi$ induces a $G(m)$-elementary morphism $U \to V$. We say that $\phi$ is \defn{locally elementary} if it is surjective and locally elementary at all $x \in X$. This implies that both $X$ and $Y$ are locally elementary. We note that a $\GL$-variety is locally elementary if and only if its identity morphism is. We again define a $G(n)$-variety, or a morphism of $G(n)$-varieties, to be \defn{locally $G(n)$-elementary} if it is so after identifying $G(n)$ with $\GL$.

\begin{example}
Let $\ulambda$ be a non-empty pure tuple. Then $X=\bA^{\ulambda} \setminus \{0\}$ is an example of a locally elementary variety that is not elementary. The fact that $X$ is locally elementary follows from Corollary~\ref{cor:loc-elem-base-change} below. It is clear that $X$ is not elementary, as it is not affine. The identity morphism of $X$ is a (somewhat boring) example of a locally elementary morphism that is not elementary.
\end{example}

\begin{proposition} \label{prop:elem1}
Let $\phi \colon X \to Y$ be a morphism of $G(n)$-varieties,
and let $m \geq n$.
\begin{enumerate}
\item If $\phi$ is $G(n)$-elementary then it is $G(m)$-elementary.
\item $\phi$ is locally $G(n)$-elementary (at $x \in X$) if and only if it is locally $G(m)$-elementary (at $x$).
\end{enumerate}
\end{proposition}

\begin{proof}
(a) There is a diagram \eqref{eq:Diagram} after identifying
$G(n)$ with $\GL$. Now apply $\Sh_{m-n}$ to that diagram.
Recall that $\Sh_{m-n} X \cong X$ where the action of $\GL$ on the
right-hand side is via its isomorphism to $G(m-n)$, which under the
identification of $\GL$ with $G(n)$ is mapped to $G(m)$. So
the top row remains unchanged (except that we remember only
that $\phi$ is $G(m)$-equivariant). In the bottom row, $B$
and $C$ are replaced by products with ordinary affine spaces $\bA^b$
and $\bA^c$ where $b \geq c$ are the multiplicities of the empty
partition in $\sh_{m-n}\ulambda$ and $\sh_{m-n}\umu$,
respectively, $\psi$ is replaced by its product with a
linear surjection $\bA^b \to \bA^c$, and $\pi$ is replaced
by a $\GL$-equivariant linear surjection $\bA^{\ulambda'} \to \bA^{\umu'}$
where $\ulambda'$ and $\umu'$ are the pure parts of
$\sh_{m-n}\ulambda$ and $\sh_{m-n}\umu'$, respectively. We thus see that $\phi$ is $G(m)$-elementary.

(b) Assume that $\phi$ is locally $G(n)$-elementary at $x \in X$. Then there exist a $k \geq n$ and $G(k)$-stable open neighborhoods $U$ of $x$ and $V$ of $\phi(x)$ such that $\phi$ induces a $G(k)$-elementary morphism $U \to V$. By part (a), $\phi \colon U \to V$ is also $G(\max\{k,m\})$-elementary, so $\phi$ is locally $G(m)$-elementary at $x$. The converse is similar (but does not require the use of part (a), since an integer $\geq m$ is automatically $\geq n$).
\end{proof}

Due to part (b) above, we simply say ``locally elementary'' in place of ``locally $G(n)$-elementary,'' since the condition is independent of $n$.

\begin{proposition} \label{prop:loc-elem-unif}
Let $\phi \colon X \to Y$ be a morphism of $\GL$-varieties. The following are equivalent:
\begin{enumerate}
\item $\phi$ is locally elementary.
\item There exist $m$ and non-empty $G(m)$-stable open subvarieties $U_1, \ldots, U_n$ of $X$ and $V_1, \ldots, V_n$ of $Y$ such that $X$ is the union of the $\GL$-translates of the $U_i$'s, $Y$ is the union of the $\GL$-translates of the $V_i$'s, and $\phi$ induces a $G(m)$-elementary map $U_i \to V_i$ for all $i$.
\end{enumerate}
\end{proposition}

\begin{proof}
It is clear that (b) implies (a). Conversely, suppose (a)
holds. For each $x \in X$, choose open $G(m_x)$-stable
neighborhoods $U_x$ of $x$ and $V_x$ of $\phi(x)$ such that
$\phi$ induces a $G(m_x)$-elementary map $U_x \to V_x$. Let $U'_x=\bigcup_{g \in \GL} gU_x$.
The sets $U'_x$ form an open cover of $X$ by $\GL$-stable
open subsets. Since $X^{\orb}$ is quasi-compact, it follows
that there is a finite subcover; thus we can find points
$x_1, \ldots, x_n$ such that $X=\bigcup_{i=1}^n U'_{x_i}$.
Put $U_i=U_{x_i}$ and $V_i=V_{x_i}$; these are open and
$G(m)$-stable, where $m=\max(m_{x_1}, \ldots, m_{x_n})$.
Note that $\ol{\phi} \colon U_i \to V_i$ is $G(m)$-elementary by Proposition~\ref{prop:elem1}(a). It is thus clear that (b) holds.
\end{proof}

\begin{proposition} \label{prop:loc-elem-base-change}
Locally elementary maps are stable under base change along open immersions. More precisely, consider a pullback diagram
\begin{displaymath}
\xymatrix{
X' \ar[r] \ar[d]_{\phi'} & X \ar[d]^{\phi} \\
Y' \ar[r] & Y }
\end{displaymath}
where $\phi$ is a locally elementary map of $\GL$-varieties and $Y' \to Y$ is an open immersion of $\GL$-varieties. Then $\phi'$ is locally elementary.
\end{proposition}

\begin{proof}
We proceed in three steps. In what follows, we simply identify $Y'$ and $X'$ with open $\GL$-subvarieties of $Y$ and $X$.

\textit{\textbf{Step 1:} Suppose $\phi$ is elementary and $Y'=Y[1/h]$ for non-zero $\GL$-invariant function $h$ on $Y$.} Identify $X$ with $B \times \bA^{\ulambda}$ and $Y$ with $C \times \bA^{\umu}$ and $\phi$ with $\psi \times \pi$, where $\psi \colon B \to C$ is a surjective morphism of varieties and $\pi \colon \bA^{\ulambda} \to \bA^{\umu}$ is the projection map. We find that $Y'=C[1/h] \times \bA^{\umu}$ and $X'=B[1/h] \times \bA^{\ulambda}$. The induced map $\phi' \colon X' \to Y'$ is simply $\psi\vert_{B[1/h]} \times \pi$, and is thus elementary.

\textit{\textbf{Step 2:} Suppose $\phi$ is elementary.} Let $x \in X'$ and $y=f(x)$. We can then find a function $h$ on $Y$ such that $y \in Y[1/h] \subset Y'$. Let $m$ be such that $h$ is $G(m)$-invariant. Regard $\phi$ as a morphism of $G(m)$-varieties---as such, it is still elementary by Proposition~\ref{prop:elem1}(a)---and apply Step~1 with the open set $Y[1/h]$. We thus see that $\phi^{-1}(Y[1/h]) \to Y[1/h]$ is $G(m)$-elementary. Since $x \in \phi^{-1}(Y[1/h]) \subset X'$, it follows that $\phi'$ is locally elementary at $x$. Since $x$ is arbitrary and $\phi'$ is surjective, it follows that $\phi'$ is locally elementary.

\textit{\textbf{Step 3:} the general case.} Let $x \in X'$ and let $y=\phi(x)$. Since $\phi$ is locally elementary, we can find $G(m)$-stable open neighborhoods $U$ of $x$ and $V$ of $y$ such that $\phi$ induces a $G(m)$-elementary morphism $\ol{\phi} \colon U \to V$. Let $V'=V \cap Y'$ and let $U'=U \cap X'=\ol{\phi}^{-1}(V')$. By Step~2, we see that the induced morphism $U' \to V'$ is locally elementary. Since $x \in U' \subset X'$, it follows that $\phi'$ is locally elementary at $x$. Since $x$ is arbitrary and $\phi'$ surjective, it follows that $\phi'$ is locally elementary.
\end{proof}

Applying the proposition to the identity map, we find:

\begin{corollary} \label{cor:loc-elem-base-change}
Let $X$ be a locally elementary $\GL$-variety and let $U$ be an open $\GL$-subvariety. Then $U$ is locally elementary.
\end{corollary}

\subsection{Locally elementary decompositions}

Let $X$ be a $\GL$-variety. A \defn{locally elementary decomposition} of $X$ is a finite decomposition $X=\bigsqcup_{i \in I} X_i$ of $X$ such that each $X_i$ is $\GL$-stable, locally closed (i.e., open in its closure), and locally elementary. Let $\phi \colon X \to Y$ be a morphism of $\GL$-varieties. A \defn{locally elementary decomposition} of $\phi$ consists of locally elementary decompositions $X=\bigsqcup_{i \in I} X_i$ and $Y=\bigsqcup_{j \in J} Y_j$ such that for each $i \in I$ there is some $j \in J$ such that $\phi(X_i) \subset Y_j$ and the map $\phi \colon X_i \to Y_j$ is locally elementary. In what follows, we sometimes abbreviate ``locally elementary decomposition'' to LED.

\begin{proposition} \label{prop:decomp-pullback}
Consider a pullback diagram
\begin{displaymath}
\xymatrix{
X' \ar[r] \ar[d]_{\phi'} & X \ar[d]^{\phi} \\
Y' \ar[r] & Y }
\end{displaymath}
where $\phi$ is a map of $\GL$-varieties and $Y' \to Y$ is an open immersion of $\GL$-varieties. Suppose $\phi$ admits an LED given by $X=\bigsqcup_{i \in I} X_i$ and $Y=\bigsqcup_{j \in J} Y_j$. Put $X'_i = X_i \cap X'$ and $Y'_j = Y_j \cap Y'$. Then $\phi'$ admits the LED $X'=\bigsqcup_{i \in I} X'_i$ and $Y'=\bigsqcup_{j \in J} Y'_j$.
\end{proposition}

\begin{proof}
This follows immediately from Proposition~\ref{prop:loc-elem-base-change}.
\end{proof}

\subsection{The decomposition theorem}

The following is our main structure theorem for morphisms of $\GL$-varieties.

\begin{theorem} \label{thm:decomp}
Every morphism of $\GL$-varieties admits a locally elementary decomposition.
\end{theorem}

Applying the theorem to the identity map, we find:

\begin{corollary}
Every $\GL$-variety admits a locally elementary decomposition.
\end{corollary}

We require a number of lemmas before proving the theorem.

\begin{lemma} \label{lem:decomp-1}
Let $X$ be a non-empty $\GL$-variety. Then $X$ contains an irreducible (and, in particular, non-empty) locally elementary open $\GL$-subvariety.
\end{lemma}

\begin{proof}
First suppose that $X$ is affine. By the shift theorem (Theorem~\ref{thm:shift}), there is an integer $n \ge 0$ and a non-zero $\GL$-invariant function $h$ on $\Sh_n(X)$ such that $\Sh_n(X)[1/h]$ is elementary. Under the identification $(\GL, \Sh_n(X)) = (G(n), X)$, the open $\GL$-subvariety $\Sh_n(X)[1/h]$ of $\Sh_n(X)$ corresponds to an open $G(n)$-subvariety $U_0$ of $X$. Now, let $U=\bigcup_{g \in \GL} g U_0$, an irreducible open $\GL$-subvariety of $X$. Since $U_0$ is $G(n)$-elementary, it follows that $U$ is locally elementary.

We now treat the general case. Let $U$ be a non-empty open affine $G(n)$-subvariety of $X$. By the previous paragraph, $U$ contains a non-empty open locally elementary $G(n)$-subvariety $V_0$. Let $V=\bigcup_{g \in \GL} g V_0$. Then $V$ is an irreducible open locally elementary $\GL$-subvariety of $X$.
\end{proof}

\begin{lemma} \label{lem:decomp-2}
Let $\phi \colon X \to Y$ be a dominant morphism of irreducible locally elementary $\GL$-varieties. Then there exist non-empty open $\GL$-subvarieties $U \subset X$ and $V \subset Y$ such that $\phi(U) \subset V$ and the induced map $\phi \colon U \to V$ is locally elementary.
\end{lemma}

\begin{proof}
First suppose that $X$ and $Y$ are elementary, say $X=B \times \bA^{\ulambda}$ and $Y=C \times \bA^{\umu}$. By Proposition~\ref{prop:Adom}, we have $\umu \subset \ulambda$ and there are non-empty open subschemes $U_0$ of $B$ and $V_0$ of $C$ such that $\phi_0(U_0)=V_0$ and the restriction of $\phi$ to $U_0 \times \bA^{\mu}$ factors as
\begin{displaymath}
\xymatrix@=1.5cm{
U_0 \times \bA^{\ulambda} \ar[r]^{\sigma} & U_0 \times
\bA^{\ulambda} \ar[r]^-{\id \times \pi} & U_0 \times
\bA^{\umu} \ar[r]^-{\phi_0 \times \id} & V_0 \times \bA^{\umu} }
\end{displaymath}
where $\sigma$ is some $U_0$-automorphism of $U_0 \times \bA^{\ulambda}$ and $\pi$ is the projection map. Thus, putting $U=U_0 \times \bA^{\ulambda}$ and $V=V_0 \times \bA^{\umu}$, we see that the diagram
\begin{displaymath}
\xymatrix@C=5em{
U \ar[r]^{\phi} \ar[d]_-i & V \ar@{=}[d] \\
U_0 \times \bA^{\ulambda} \ar[r]^{\phi_0 \times \pi} & V_0 \times \bA^{\umu} }
\end{displaymath}
commutes, where $i=\sigma$. Thus $\phi \colon U \to V$ is elementary.

We now treat the general case. Since $Y$ is non-empty and
locally elementary, we can find a non-empty open elementary
$G(n)$-subvariety $Y'$ of $Y$. Since $\phi^{-1}(Y')$ is a locally elementary
$G(n)$-variety (Proposition~\ref{prop:loc-elem-base-change}), by the same reasoning we can find a non-empty open elementary affine $G(m)$-subvariety $X'$ of $\phi^{-1}(Y')$, for some $m \ge n$. Note that $Y'$ is an elementary $G(m)$-variety (Proposition~\ref{prop:elem1}(a)). Thus $\phi \colon X' \to Y'$ is a dominant morphism of irreducible elementary affine $G(m)$-varieties. By the previous paragraph, we can find non-empty open $G(m)$-subvarieties $U' \subset X'$ and $V' \subset Y'$ such that $\phi$ induces an elementary map $U' \to V'$. Put $U=\bigcup_{g \in \GL} gU'$ and $V=\bigcup_{g \in \GL} gV'$. Then $\phi$ induces a locally elementary map $U \to V$, which completes the proof.
\end{proof}

\begin{lemma} \label{lem:decomp-3}
Let $\phi \colon X \to Y$ be a dominant morphism of non-empty $\GL$-varieties. Then there exist non-empty open $\GL$-subvarieties $U \subset X$ and $V \subset Y$ such that $\phi(U) \subset V$ and the map $\phi \colon U \to V$ is locally elementary.
\end{lemma}

\begin{proof}
By Lemma~\ref{lem:decomp-1}, we can find an irreducible open locally elementary $\GL$-subvariety $Y'$ of $Y$. By Lemma~\ref{lem:decomp-1} again, we can find an irreducible open locally elementary $\GL$-subvariety $X'$ of $\phi^{-1}(Y')$. By Lemma~\ref{lem:decomp-2}, we can find non-empty open $\GL$-subvarieties $U \subset X'$ and $V \subset Y'$ such that $\phi(U) \subset V$ and the map $\phi \colon U \to V$ is locally elementary.
\end{proof}

\begin{proof}[Proof of Theorem~\ref{thm:decomp}]
Let $\phi \colon X \to Y$ be a given morphism of $\GL$-varieties. By noetherian induction on $Y$, we can assume that for any proper closed $\GL$-subvariety $Z$ of $Y$, the map $\phi^{-1}(Z) \to Z$ admits an LED.

First suppose that $Y$ is reducible. Write $Y=\ol{Y}_1 \cup \ol{Y}_2$ where $\ol{Y}_1$ and $\ol{Y}_2$ are proper closed $\GL$-subvarieties. Let $Y_3=\ol{Y}_1 \cap \ol{Y}_2$ and $Y_1=\ol{Y}_1 \setminus \ol{Y}_2$ and $Y_2=\ol{Y}_2 \setminus \ol{Y}_1$; thus $Y=Y_1 \sqcup Y_2 \sqcup Y_3$ is a decomposition of $Y$ into locally closed $\GL$-subvarieties. For $i \in \{1,2\}$, the morphism $\phi^{-1}(\ol{Y}_i) \to \ol{Y}_i$ admits an LED by the inductive hypothesis, and so $\phi^{-1}(Y_i) \to Y_i$ admits one by Proposition~\ref{prop:decomp-pullback}. The morphism $\phi^{-1}(Y_3) \to Y_3$ also admits an LED by the inductive hypothesis. Putting together the LED's for $\phi^{-1}(Y_i) \to Y_i$ for $i \in \{1,2,3\}$, we obtain one for $X \to Y$.

Now suppose that $Y$ is irreducible. By noetherian induction on $X$, we can assume that for any proper closed $\GL$-subvariety $Z$ of $X$, the map $\phi \colon Z \to Y$ admits an LED. If $\phi$ is not dominant then its image is contained in a proper closed $\GL$-subvariety $Z$ of $Y$. We obtain an LED for $\phi$ by taking one for $\phi^{-1}(Z) \to Z$ (which exists by the first inductive hypothesis) and simply throwing in the open set $Y \setminus Z$ into the decomposition of $Y$. We can thus assume in what follows that $\phi$ is dominant.

By Lemma~\ref{lem:decomp-3}, we can find non-empty open $\GL$-subvarieties $W \subset X$ and $V_1 \subset Y$ such that $\phi(W) \subset V_1$ and the map $\phi \colon W \to V_1$ is locally elementary. Let $Z=X \setminus W$. By the second inductive hypothesis, the map $\phi \colon Z \to Y$ admits an LED, say $Z=\bigsqcup_{i \in I} Z_i$ and $Y=\bigsqcup_{j \in J} Y_j$. Let $j_0$ be the unique element of $J$ such that $V_2=Y_{j_0}$ is open, and let $I_0 \subset I$ be the set of indices $i$ such that $Z_i$ maps into $V_2$. Let $V=V_1 \cap V_2$ and let $U=\phi^{-1}(V)$. We have $U=(W \cap U) \sqcup \bigsqcup_{i \in I_0} (Z_i \cap U)$. The maps $W \cap U \to V$ and $Z_i \cap U \to V$ are locally elementary by Proposition~\ref{prop:loc-elem-base-change}. We thus see that this decomposition of $U$, together with the trivial decomposition $V=V$ of $V$, is a LED for $\phi \colon U \to V$. Combining this LED with the one for $\phi^{-1}(Y \setminus V) \to Y \setminus V$ (which exists by the first inductive hypothesis), we obtain one for~$\phi$.
\end{proof}

\subsection{Chevalley's theorem}

%Recall that a subset of a variety $A$ is called \emph{constructible} if it is a finite union of sets of the form $Z \cap U$ where $Z$ is a closed subset of $A$ and $U$ is an open subset of $A$. Chevalley's theorem states that any morphism of varieties maps constructible sets to constructible sets.

%We now establish an analogue of Chevalley's theorem for $\GL$-varieties.
Let $X$ be a $\GL$-variety. We say that a subset of $X$ is \defn{$\GL$-constructible} if it is a finite union of sets of locally closed $\GL$-stable subsets. The following is an analog of Chevalley's theorem for $\GL$-varieties:

\begin{theorem} \label{thm:chev}
Let $\phi \colon X \to Y$ be a morphism of quasi-affine $\GL$-varieties, and let $C$ be a $\GL$-constructible subset of $X$. Then $\phi(C)$ is a $\GL$-constructible subset of $Y$.
\end{theorem}

\begin{proof}
First suppose that $C=X$. Let $X=\bigsqcup_{i \in I} X_i$ and $Y=\bigsqcup_{j \in J} Y_j$ be a LED for $\phi$. Thus for each $i \in I$ there is $j \in J$ such that $\phi$ induces a locally elementary map $X_i \to Y_j$. Since a locally elementary map is surjective, it follows that the image of $\phi$ is a union of some subset of $\{Y_j\}_{j \in J}$, and thus $\GL$-constructible.

Now if $C$ is a general $\GL$-constructible set in $X$, then it is a finite union of locally closed $\GL$-stable subsets of $X$, and we apply the previous paragraph to the restrictions of $\phi$ to these subsets.
\end{proof}

\begin{remark}
In EGA and the Stacks project (see \stacks{04ZC}), constructible sets are defined for arbitrary schemes (in fact, arbitrary topological spaces). However, only the so-called retrocompact open subsets are considered constructible. The basic example of a non-retrocompact open subset is $\bA^{\infty} \setminus \{0\}$. In particular, the open $\GL$-subset $\bA^{\ulambda} \setminus \{0\}$ of $\bA^{\ulambda}$ is not retrocompact, and therefore is not constructible in the usual sense. This is why we use the term $\GL$-constructible. We note, however, that $\GL$-constructible subsets of $X$ are in bijective correspondence with constructible subsets (in the usual sense) of $X^{\orb}$.
\end{remark}

\subsection{Lifting points}

Suppose that $\phi \colon X \to Y$ is a morphism of $\GL$-varieties and $y$ is a $K$-point in the image of $\phi$. Finding a point of $X$ lifting $y$ typically involves solving infinitely many equations, so it is not immediately clear how to control the extension of $K$ required to find a solution. Using the decomposition theorem, we find that the nicest reasonable statement is true:

\begin{proposition} \label{prop:points}
Let $\phi \colon X \to Y$ be a morphism of quasi-affine $\GL$-varieties. Let $\Omega$ be an extension of $K$ and let $y$ be an $\Omega$-point of the image of $\phi$. Then there exists a finite extension $\Omega'/\Omega$ and an $\Omega'$-point $x$ of $X$ mapping to $y$. In fact, there exists an integer $d$ (depending only on $\phi$) such that one can always take $[\Omega':\Omega] \le d$.
\end{proposition}

\begin{proof}
The statement is clear for maps of ordinary varieties, thus for elementary maps of $\GL$-varieties, thus for locally elementary maps (use Proposition~\ref{prop:loc-elem-unif} to get uniformity of $[\Omega':\Omega]$), and thus for all maps by Theorem~\ref{thm:decomp}.
\end{proof}

\subsection{A variant} \label{ss:decomp-var}

Suppose that $\cP$ is some property of varieties and morphisms. We define a $\GL$-variety to be \defn{$\cP$-elementary} if it has the form $B \times \bA^{\ulambda}$ where $B$ is a variety satisfying $\cP$. We define a morphism of $\GL$-varieties to be \defn{$\cP$-elementary} if there is a diagram like \eqref{eq:Diagram} where $\psi$ satisfies $\cP$. Our definitions of elementary correspond to taking $\cP$ to be the property ``irreducible affine'' on varieties and the property ``surjective'' on morphisms. One can prove an analog of the decomposition theorem for many different choices of $\cP$; essentially, one just needs to be able to decompose a variety or morphism of varieties into pieces that satisfy $\cP$. For instance, there is a version of the decomposition theorem where the strata are smooth and the morphisms on strata are flat. The particular $\cP$ we worked with is somewhat arbitrary, but was chosen to make the proof of Chevalley's theorem as simple as possible.

\section{Theory of types} \label{s:type}

\subsection{Types}

We fix an irreducible $\GL$-variety $X$ for this section.

\begin{definition} \label{defn:typical}
A \defn{typical morphism} for $X$ is a dominant morphism $\phi \colon B \times \bA^{\ulambda} \to X$, with $B$ an affine variety and $\ulambda$ a pure tuple, with the following property: if $Z$ is a proper closed $\GL$-subvariety of $B \times \bA^{\ulambda}$ then $\phi \vert_Z$ is not dominant.
\end{definition}

We note that in a typical morphism, the variety $B$ is necessarily irreducible.

\begin{proposition} \label{prop:typical-exists}
A typical morphism exists.
\end{proposition}

\begin{proof}
Let $\Lambda$ be the set of all pure tuples $\ulambda$ for which a dominant morphism $B \times \bA^{\ulambda} \to X$ exists. This set is non-empty by Theorem~\ref{thm:uni}. Let $\ulambda \in \Lambda$ be any element of minimal magnitude; such a minimal element exists since magnitudes are well-ordered. Now let $d \in \bN$ be minimal such that a dominant morphism $B \times \bA^{\ulambda} \to X$ exists with $\dim(B)=d$.

Let $\phi \colon B \times \bA^{\ulambda} \to X$ be any dominant morphism with $B$ irreducible of dimension $d$. We claim that $\phi$ is typical. Indeed, suppose that $Z$ is a proper closed subset of $B \times \bA^{\ulambda}$ such that $\phi \vert_Z$ is dominant. We will obtain a contradiction. We may as well suppose that $Z$ is irreducible. We now apply Theorem~\ref{thm:uni}(d). There are two cases. In the first case, $Z=C \times \bA^{\ulambda}$ for some (necessarily proper) subvariety $C$ of $B$ (to apply Theorem~\ref{thm:uni} as written, embed $B$ into some $\bA^r$). This gives a contradiction since $\dim(C)<d$. In the second case, there is a dominant morphism $B \times \bA^{\umu} \to Z$ for some $\umu$ with $\magn(\umu)<\magn(\ulambda)$; thus $\umu \in \Lambda$. This contradicts the minimality of $\ulambda$.
\end{proof}

\begin{proposition} \label{prop:typical-dom}
Let $\phi \colon B \times \bA^{\ulambda} \to X$ be a typical morphism and let $\psi \colon Y \to B \times \bA^{\ulambda}$ be another morphism of $\GL$-varieties. If $\phi \circ \psi$ is dominant then $\psi$ is dominant.
\end{proposition}

\begin{proof}
Let $Z=\ol{\im(\psi)}$. Since $\phi \circ \psi$ is dominant, it follows that $\phi \vert_Z$ is dominant. Since $\phi$ is typical, we have $Z=B \times \bA^{\ulambda}$, and thus $\psi$ is dominant.
\end{proof}

\begin{proposition} \label{prop:typical-factor}
Let $\phi \colon B \times \bA^{\ulambda} \to X$ be a typical morphism, and let $\psi \colon C \times \bA^{\umu} \to X$ be a dominant morphism with $C$ irreducible and $\umu$ pure. Then there exists a commutative diagram
\begin{displaymath}
\xymatrix@C=4em{
D \times \bA^{\umu} \ar[r]^{\alpha \times \id} \ar[d]_{\theta} & C \times \bA^{\umu} \ar[d]^{\psi} \\
B \times \bA^{\ulambda} \ar[r]^{\phi} & X }
\end{displaymath}
where $D$ is an irreducible variety, $\alpha \colon D \to C$ is dominant and quasi-finite, and $\theta$ is dominant. In particular, $\ulambda \subset \umu$ and $\dim(B) \le \dim(C)$.
\end{proposition}

\begin{proof}
Let $x$ be be the generic point of $C \times \bA^{\umu}$ and let $y$ be the generic point of $B \times \bA^{\ulambda}$. Then $\psi(x)=\phi(y)$ is the generic point of $X$ since both $\phi$ and $\psi$ are dominant; in particular, $\psi(x) \in \im(\phi)$. Applying Proposition~\ref{prop:lift} (with $Y=B \times \bA^{\ulambda}$), we obtain a commutative square having the stated properties except perhaps for the dominance of $\theta$. However, since $\phi \circ \theta$ is dominant, Proposition~\ref{prop:typical-dom} shows that $\theta$ is dominant. Proposition~\ref{prop:Adom} now shows that $\ulambda \subset \umu$. Since $\theta$ induces a dominant morphism $D \to B$, we have $\dim(B) \le \dim(D)=\dim(C)$.
\end{proof}

\begin{corollary}
Let $\phi \colon B \times \bA^{\ulambda} \to X$ and $\psi \colon C \times \bA^{\umu} \to X$ be typical morphisms. Then $\ulambda=\umu$ and $\dim(B)=\dim(C)$.
\end{corollary}

\begin{definition}
Let $\phi \colon B \times \bA^{\ulambda} \to X$ be a typical morphism. We define the \defn{type} of $X$, denoted $\type(X)$, to be the pure tuple $\ulambda$. We define the \defn{typical dimension} of $X$, denoted $\tdim(X)$, to be $\dim(B)$. These are well-defined by the above corollary.
\end{definition}

\begin{proposition} \label{prop:typical-bounds}
Let $\phi \colon B \times \bA^{\ulambda} \to X$ be a dominant morphism with $B$ irreducible and $\ulambda$ pure. Then $\type(X) \subset \ulambda$ and $\tdim(X) \le \dim(B)$, with equalities if and only if $\phi$ is typical.
\end{proposition}

\begin{proof}
The inequalities $\type(X) \subset \umu$ and $\tdim(X) \le \dim(C)$ follow from Proposition~\ref{prop:typical-factor}. If $\phi$ is typical, then both are equalities by definition. Conversely, suppose that both are equalities. We follow the proof of Proposition~\ref{prop:typical-exists}. By the inequalities proven here, the set $\Lambda$ defined there consists of all tuples $\umu$ such that $\type(X) \subset \umu$ and the numer $d$ defined there is just $\tdim(X)$. Thus the second paragraph of the proof shows that $\phi$ is typical.
\end{proof}

\subsection{Characterization of typical dimension} \label{ss:tdim}

Let $X$ be an irreducible $\GL$-variety. The following is the main theorem of this section:

\begin{theorem} \label{thm:typdim}
Let $\phi \colon B \times \bA^{\ulambda} \to X$ be a typical morphism. Then $K(B)$ is a finite extension of $K(X)^{\GL}$.
\end{theorem}

We note that $\phi^*$ induces a field extension $K(X) \subset K(B \times \bA^{\ulambda})$, and thus an extension $K(X)^{\GL} \subset K(B \times \bA^{\ulambda})^{\GL}=K(B)$, where we have used Proposition~\ref{prop:Ainv}. The content of the theorem is that this extension has finite degree.

\begin{corollary}
The typical dimension of $X$ is equal to the transcendence degree over $K$ of the invariant function field $K(X)^{\GL}$.
\end{corollary}

\begin{proof}
Let $\phi$ be as in Theorem~\ref{thm:typdim}. Then we have
\begin{displaymath}
\operatorname{tr. deg}(K(X)^{\GL}/K)
=\operatorname{tr. deg}(K(B)/K)
=\dim(B)
=\tdim(X).
\end{displaymath}
where the first equality follows from Theorem~\ref{thm:typdim}, the second is standard, and the third is the definition of typical dimension.
\end{proof}

The proof of the theorem occupies the remainder of \S \ref{ss:tdim}. It is inspired by Rosenlicht's theorem on separating general orbits by rational invariants \cite[Theorem 2]{rosenlicht}. It needs the following lemma, as well as its proof.

\begin{lemma} \label{lem:Invariant}
Let $V$ be a vector space over a field $K$, $F \supset K$ a field extension,
$G$ a group of $K$-automorphisms of $F$, and $W$ an $F$-subspace of $F
\otimes_K V$ that is stable under the action of $G$ on the first tensor
factor. Then $W$ is spanned over $F$ by elements in $W^G \subset F^G \otimes_K V$.
\end{lemma}

\begin{proof}
Choose a well-ordered $K$-basis $(e_i)_{i \in I}$ of $V$, so that each
element $w$ of $W$ has a unique expression $w=\sum_{i \in I} c_i \otimes
e_i$ where only finitely many of the $c_i \in F$ are nonzero; the largest
$e_i$ with $c_i \neq 0$ is called the leading term of $w$, and $c_i$
its leading coefficient.

Call $w$ minimal if the support $\{i: c_i \neq 0\}$ is inclusion-wise
minimal among all supports of nonzero vectors of $W$ and if moreover the leading
coefficient equals $1$. Then for each $g \in G$, $gw-w \in W$ has support
strictly contained in that of $w$, and must therefore be zero. Hence
any minimal $w$ has coefficients in $F^G$.

Furthermore, every $u \in W$ is an $F$-linear combination
of minimal elements: find a minimal $w$ whose leading term equals that
of $u$, subtract $w$ times the leading coefficient of $u$
from $u$ so that the leading term becomes smaller, and use well-orderedness to conclude.
\end{proof}

\begin{proof}[Proof of Theorem~\ref{thm:typdim}]
Let $\phi \colon B \times \bA^{\ulambda} \to X$ be a typical morphism.
We may assume that $K$ is algebraically
closed; we then may and will work exclusively with
$K$-points in this proof. After passing to an
affine open dense subset of $B$, we may further assume that for
all $K$-points $b \in B$, the map $\phi(b,.): \bA^{\ulambda} \to
\overline{\im(\phi(b,.))}$ is typical, i.e., does not factor through
$\bA^{\umu}$ for any $\umu \subset \ulambda$.

Let $\Gamma \subseteq
X \times B \times \bA^{\ulambda}$ be the graph of $\phi$, and let
$\pi:\Gamma \to X \times B$ be the restriction of the projection morphism;
this is a morphism of $\GL$-varieties. Note that the points of
$\im(\pi)$ are the pairs $(x,b)$ with $x \in \im(\phi(b,.))$.

Let $I \subseteq K[X \times B]$ be an ideal whose vanishing set is
$\overline{\im(\pi)}$ and which is generated by the $\GL$-orbits of
finitely many elements $h_1,\ldots,h_m$; $I$ exists by Noetherianity (Theorem~\ref{thm:Noetherianity}). Let
$J$ be the ideal generated by $I$ in $K(X) \otimes K[B]$.  Then
$J$ is stable under the action of $\GL$ on $K(X)$ and hence by
Lemma~\ref{lem:Invariant}, $J$ is spanned over $K(X)$ by elements in
$K(X)^\GL \otimes K[B]$. In particular, we can find $\GL$-invariant
elements $f_1,\ldots,f_k \in J$
that are minimal relative to an arbitrary well-ordered $K$-basis of $K[B]$
and such that each $h_i$ is a $K(X)$-linear combination of the $f_j$.
Write $f_j=\sum_l g_{jl} \otimes f_{jl}$ where the $f_{jl}$ are elements
of the chosen basis of $K[B]$ and $g_{jl} \in K(X)^{\GL}$. The proof
of Lemma~\ref{lem:Invariant} shows that we can arrange
things such that each $h_i$
is in fact a linear combination of the $f_j$ with coefficients
in the subalgebra $K[X][\{g_{jl}\}]$ of $K(X)$ generated by $K[X]$
and the $g_{jl}$. Note that then, by $\GL$-invariance of the $g_{jl}$, each
element in $\GL\cdot h_i$ is also a $K[X][\{g_{jl}\}]$-linear
combination of the $f_j$.

By Chevalley's theorem (Theorem~\ref{thm:chev}), $\im(\pi)$ contains a dense, $\GL$-stable,
open subset $U_1$ of $\overline{\im(\pi)}$. We then have
$U_1=\overline{\im(\pi)} \cap U_2$ for some $\GL$-stable, dense open
subset $U_2$ of $X \times B$. Now intersect $U_2$ with the preimage in
$X \times B$ of the domains of definition in $X$ of the finitely
many rational functions in $K(X)$ needed to express each $f_j$ as a
linear combination of elements in $I$. This yields a (not necessarily
$\GL$-stable) open, dense subset $U_3$ of $X \times B$,
which, since $\overline{\im(\pi)} \to X$ is dominant, still
intersects $\overline{\im(\pi)}$ in an open dense subset contained
in $\im(\pi)$ and satisfies $f_j(x,b)=0$ for all $(x,b) \in \im(\pi)
\cap U_3$.  Next intersect $U_3$ with the preimage in $X \times B$ of the
domains of definition in $X$ of the finitely many $g_{jl} \in K(X)^{\GL}$. This
yields an open, dense subset $U_4$ of $X \times B$, which still intersects
$\overline{\im(\pi)}$ in an open dense set and such that, for $(x,b)
\in U_4$, we have $(x,b) \in \im(\pi)$ if and only if $f_j(x,b)=0$ for
all $j$---this follows since, over the open subset $U_4$, every element of each
$\GL \cdot h_i$ is a linear combination of the $f_j$.
Finally, let $U_5 \subseteq B \times \bA^{\ulambda} \times B$ be the
preimage of $U_4$ under the morphism $\psi:B \times \bA^{\ulambda}
\times B \to X \times B,\ (b_1,a,b_2) \mapsto (\phi(b_1,a),b_2)$. Note
that $\psi(b,a,b)=(\phi(b,a),b)  \in \im(\pi)$ and that this precisely
parameterizes $\im(\pi)$, so $U_6:=\{(b,a) \in B \times \bA^{\ulambda}
\mid (b,a,b) \in U_5\}$ is open dense in $B \times \bA^{\ulambda}$
and for $(a,b) \in U_6$ we have $f_j(\psi(b,a,b))=0$ for all $j$.

By the first paragraph of this proof, the $g_{jl} \in K(X)^\GL$ may be
regarded as elements of $K(B)$. After shrinking $B$ to a dense open
affine subset (and adapting the open subsets above accordingly), we
may assume that they are in $K[B]$. For $(b_1,a,b_2) \in U_5 \subseteq
B \times \bA^{\ulambda} \times B$ we have $(\phi(b_1,a),b_2) \in \im(\pi)$ if and
only if, for all $j$,
\[ 0=f_j(\phi(b_1,a),b_2)=\sum_l g_{jl}(b_1) f_{jl}(b_2). \]
Note that this condition is, in fact, independent of $a$.

Now pick a $\GL$-generic $K$-point $a \in \bA^{\ulambda}$ such that
$(b,a) \in U_6$ for some, and hence most, $b \in B$; such a point exists
by Proposition~\ref{prop:generic-exists}.  Then $V:=U_5 \cap (B \times
\{a\} \times B)$ is nonempty and therefore open dense in $B \times \{a\}
\times B$, and contains $(b,a,b)$ for $b$ in an open dense subset of $B$.
After replacing $B$ by an open dense affine subset (and updating all
open subsets accordingly), we may and will assume that
$(b,a,b) \in V$ for all $b \in B$.

Then, for $(b_1,a,b_2) \in V$, we have
\[ \phi(b_1,a) \in \im(\phi(b_2,.)) \Leftrightarrow
\sum_l g_{jl}(b_1) f_{jl}(b_2)=0 \text{ for all }
j=1,\ldots,k. \]
Since, by assumption, $\phi(b_1,.)$ and $\phi(b_2,.)$ are typical
morphisms and since $a \in \bA^{\ulambda}$ is $\GL$-generic, we have
$\phi(b_1,a) \in \im(\phi(b_2,.))$ if and only if the more symmetric
property $\im(\phi(b_1,.))=\im(\phi(b_2,.))$ holds.

Now the functions $g_{jl} \in K(X)^\GL \cap K[B]$ define a dominant
morphism $\gamma$ from $B$ to a variety $C$ whose coordinate ring is
generated by the $g_{jl}$. Let $D$ be an irreducible closed subvariety of
$B$ such that $\gamma|_D: D \to C$ is dominant and quasi-finite. After
further shrinking $B,C,D$ we may assume that $\gamma|_D$ and $\gamma$
are surjective.  Since $(b,a,b) \in V$ for any $b \in D$, $V \cap (D
\times \{a\} \times B)$ is nonempty and hence dense in $D \times \{a\}
\times B$.

For every $b_1 \in D$, the set of $b_2 \in B$ with $(b_1,a,b_2) \in V$
and $\gamma(b_2)=b_1$ is an open neighborhood of $b_1$ in the fiber
$\gamma^{-1}(\gamma(b_1))$. The union of these open neighborhoods in
fibers is a constructible set in $B$ which is easily seen to be dense.
Hence it contains an open dense subset of $B$, which, after shrinking
$B,C,D$ appropriately, we may assume to be all of $B$. Then, for any $b_2
\in B$, there exists a $b_1 \in D$ such that $\gamma(b_1)=\gamma(b_2)$
and $(b_1,a,b_2) \in V$.

We claim that $D \times \bA^{\ulambda} \to X$ is already dominant. Indeed,
consider a point $x \in \im(\phi(b_2,.))$ for some $b_2 \in B$ and
let $b_1 \in D$ be such that $(b_1,a,b_2) \in V$ and
$\gamma(b_1)=\gamma(b_2)$.  Then $g_{jl}(b_1)=g_{jl}(b_2)$
holds for all $j,l$ and we find that, for each $j$,
\[ \sum_l g_{jl}(b_1) f_{jl}(b_2)
= \sum_l g_{jl}(b_2) f_{jl}(b_2)
= 0 \]
so that $\im(\phi(b_1,.))=\im(\phi(b_2,.)) \ni x$. Since $B$
was assumed to have minimal dimension, we have $D=B$. We thus see that $K(B)/K(C)$ is a finite extension. Since $K(C) \subset K(X)^{\GL}$, the result follows.
\end{proof}

\begin{remark}
Rosenlicht's theorem that inspired the proof above says that when
an algebraic group $G$ acts on an ordinary algebraic variety $Y$,
there is a $G$-stable open subset of $Y$ in which the orbits are
separated by rational invariants in $K(Y)^G$. The proof above shows
that a similar statement holds for orbit closures of certain points
in $X$. Indeed, after shrinking $B$ as we did, suppose that $x,y$
are $K$-points in $\im(\phi)$ of type $\ulambda$ (see the next
section for types of points), and that $g_{jl}(x)=g_{jl}(x')$
for all $j,l$. Then we claim that $\ol{O}_x=\ol{O}_{x'}$.  Indeed,
write $x=\phi(b_1,a_1)$ and $y=\phi(b_2,a_2)$, where $a_1,a_2 \in
\bA^{\ulambda}$ are $\GL$-generic. Then the orbit closures of $x$ and $y$
equal $\overline{\im(\phi(b_1,.))}$ and $\overline{\im(\phi(b_2,.))}$
and these are equal since, in the notation of the proof above,
$\gamma(b_1)=\gamma(b_2)$.
\end{remark}

\subsection{Types of points}

We now extend the definitions we made for irreducible $\GL$-varieties to points by taking orbit closures:

\begin{definition}
Let $X$ be an arbitrary $\GL$-variety and let $x \in X$.
\begin{itemize}
\item The \defn{type} of $x$, denoted $\type(x)$, is the type of $\ol{O}_x$.
\item The \defn{typical dimension} of $x$, denoted $\tdim(x)$, is the typical dimension of $\ol{O}_x$.
\item A \defn{typical morphism} for $x$ is one for $\ol{O}_x$. \qedhere
\end{itemize}
\end{definition}

We note that if $X$ is an irreducible $\GL$-variety, then type, typical dimension, and typical morphisms for $X$ coincide with those for the generic point of $X$. We also note that the type of a point depends only on its generalized orbit, and thus makes sense for points of $X^{\orb}$.

\begin{proposition} \label{prop:type-img}
Let $\phi \colon X \to Y$ be a morphism of $\GL$-varieties and let $x \in X$. Then $\type(\phi(x)) \subset \type(x)$ and $\tdim(\phi(x)) \le \tdim(x)$.
\end{proposition}

\begin{proof}
Let $\psi \colon B \times \bA^{\ulambda} \to \ol{O}_x$ be a typical morphism. The morphism $\phi$ induces a dominant morphism $\phi' \colon \ol{O}_x \to \ol{O}_{\phi(x)}$. The composition $\phi' \circ \psi \colon B \times \bA^{\ulambda} \to \ol{O}_{\phi(x)}$ is dominant, and so the result follows from Proposition~\ref{prop:typical-bounds}.
\end{proof}

\begin{proposition} \label{prop:type-mag}
Let $\ulambda$ be a pure tuple and let $x \in \bA^r \times \bA^{\ulambda}$. Then either $\type(x)=\ulambda$ or $\magn(\type(x))<\magn(\ulambda)$, with the former occurring if and only if $\ol{O}_x=C \times \bA^{\ulambda}$ for some closed subvariety $C \subset \bA^r$.
\end{proposition}

\begin{proof}
This follows from Theorem~\ref{thm:uni}(d).
\end{proof}

\subsection{A stratification of the orbit space} \label{ss:orb-lambda}

Let $X$ be a $\GL$-variety. For a pure tuple $\ulambda$, let $X^{\orb}_{\ulambda}$ be the subset of $X^{\orb}$ consisting of all points $x$ with $\type(x) \subset \ulambda$. Endow this subset with the subspace topology. We note that if $\ulambda \subset \umu$ then $X^{\orb}_{\ulambda} \subset X^{\orb}_{\umu}$; also $X^{\orb}=\bigcup_{\ulambda} X^{\orb}_{\ulambda}$. We can thus regard $\{ X^{\orb}_{\ulambda} \}_{\ulambda}$ as a stratification of the space $X^{\orb}$.

\begin{proposition} \label{prop:dense-type}
Suppose that the points of type $\ulambda$ are dense in $X$. Then each generic point of $X$ has type $\ulambda$.
\end{proposition}

\begin{proof}
Without loss of generality, we assume $X$ is irreducible. Let $\phi \colon B \times \bA^{\umu} \to X$ be a typical morphism. The natural map $\cM_{\ulambda}(X) \times \bA^{\ulambda} \to X$ contains all points of type $\ulambda$ in its image, and it is therefore dominant. It follows that there is some irreducible component $C$ of $\cM_{\ulambda}(X)$ such that the map $\psi \colon C \times \bA^{\ulambda} \to X$ is dominant. We thus see that $\umu \subset \ulambda$. By Chevalley's theorem, there is an open dense $\GL$-subset $U$ of $X$ contained in both $\im(\phi)$. Since the type $\ulambda$ points are dense in $X$, there is a point $x \in U$ of type $\ulambda$. Let $y \in \phi^{-1}(x)$. Then $\ulambda=\type(x) \subset \type(y)$ (Proposition~\ref{prop:type-img}) and $\magn(\type(y)) \le \magn(\umu)$ (Proposition~\ref{prop:type-mag}). Thus, putting $\unu=\type(y)$, we have
\begin{displaymath}
\umu \subset \ulambda \subset \unu, \qquad \magn(\unu) \le \magn(\umu).
\end{displaymath}
It follows that $\ulambda=\unu=\umu$. Thus $X$ has type $\ulambda$, as required.
\end{proof}

\begin{proposition}
The space $X^{\orb}_{\ulambda}$ is a noetherian spectral space.
\end{proposition}

\begin{proof}
Let $Z$ be an irreducible closed subset of $X^{\orb}_{\ulambda}$. For $\umu \subset \ulambda$, let $Z_{\umu}$ be the set of points of $Z$ of type $\umu$. Since $Z$ is the union of the $Z_{\umu}$'s and there are only finitely many choices for $\umu$, it follows that there is some $Z_{\umu}$ that is dense in $Z$. Let $\ol{Z}$ be the Zariski closure of $Z$ in $X^{\orb}$, which is irreducible. Then $Z_{\umu}$ is dense in $\ol{Z}$, and thus the generic point $z$ of $\ol{Z}$ has type $\umu$ by Proposition~\ref{prop:dense-type}. It follows that $z \in \ol{Z} \cap X_{\ulambda}^{\orb}=Z$. Clearly, $z$ is a generic point of $Z$. It is also unique: if $z'$ is a generic point of $Z$ then it is also one for $\ol{Z}$, and thus $z=z'$ since $X^{\orb}$ is sober. We have thus shown that $X^{\orb}_{\ulambda}$ is sober. It is noetherian since it is a subspace of the noetherian space $X^{\orb}$. It is therefore also spectral.
\end{proof}

\subsection{Self-maps of affine spaces}

Let $\ulambda$ be a pure tuple. We let $\Gamma^+_{\ulambda}$ be the endomorphism monoid of $\bA^{\ulambda}$ and we let $\Gamma_{\ulambda}$ be the automorphism group of $\bA^{\ulambda}$. We regard both as algebraic varieties. More precisely, $\Gamma^+_{\ulambda}$ is the mapping space $\uMap_K(\bA^{\ulambda}, \bA^{\ulambda})$ and $\Gamma_{\ulambda}$ is the closed subvariety of $\Gamma^+_{\ulambda} \times \Gamma^+_{\ulambda}$ consisting of pairs $(x,y)$ such that $xy=1$.

\begin{proposition} \label{prop:Gamma}
Both $\Gamma_{\ulambda}$ and $\Gamma^+_{\ulambda}$ are irreducible varieties. Moreover, $\Gamma_{\ulambda}$ is an open dense subvariety of $\Gamma^+_{\ulambda}$.
\end{proposition}

\begin{proof}
Let $\mu_1, \ldots, \mu_n$ be the distinct partitions appearing in $\ulambda$, arranged so that $\# \mu_i \le \# \mu_{i+1}$, and let $m_i$ be the multiplicity of $\mu_i$ in $\ulambda$. Let $Y_i=\bA^{\mu_i} \otimes K^{m_i}$, which is simply a product of $m_i$ copies of $\mu_i$, and let $X_i=Y_1 \times \cdots \times Y_i$. Thus $X_n=\bA^{\ulambda}$. We have $\Gamma^+_{\ulambda} = \prod_{i=1}^n \uMap_K(X, Y_i)$. Now, any map $X \to Y_i$ must factor through the projection $X \to X_i$. Any map $X_i \to Y_i$ decomposes as $f+g$, where $f$ is induced by a linear endomorphism of $K^{m_i}$ and $g$ factors through the projection $X_i \to X_{i-1}$. We thus have
\begin{displaymath}
\Gamma_{\ulambda}^+ = \prod_{i=1}^n (\bM_{m_i} \times \uMap(X_{i-1}, Y_i)),
\end{displaymath}
where $\bM_d$ denotes the space of $d \times d$ matrices. By Proposition~\ref{prop:Map}(c), the mapping spaces appearing above are affine spaces. We thus see that $\Gamma_{\ulambda}^+$ is an affine space, and, in particular, irreducible.

It is not difficult to see that, under the above identification, we have
\begin{displaymath}
\Gamma_{\ulambda} = \prod_{i=1}^n (\GL_{m_i} \times \uMap(X_{i-1}, X_i)),
\end{displaymath}
from which it easily follows that $\Gamma_{\ulambda}$ is irreducible, and open and dense in $\Gamma_{\ulambda}^+$.
\end{proof}

\begin{corollary}
Any irreducible component of $\cM_{\ulambda}(X)$ is stable by $\Gamma_{\ulambda}$.
\end{corollary}

\subsection{The principal component of the mapping space} \label{ss:prin}

We now isolate an important irreducible component of the mapping space $\cM_{\ulambda}(X)$, in certain cases.

\begin{lemma} \label{lem:prin}
Let $X$ be an irreducible variety of type $\umu$. Let $\ulambda$ be a pure tuple containing $\umu$ and let $\pi \colon \bA^{\ulambda} \to \bA^{\umu}$ be the projection map. Let $\phi \colon B \times \bA^{\umu} \to X$ be a typical morphism and let $\psi \colon C \times \bA^{\ulambda} \to X$ be a dominant morphism, with $C$ irreducible. Then there exist non-empty open subsets $U \subset B$ and $V \subset C$ such that, for any algebraically closed field $\Omega$ containing $K$, we have:
\begin{enumerate}
\item Given $x \in U(\Omega)$ there exists $z \in C(\Omega)$ and $\sigma \in \Gamma_{\ulambda}(\Omega)$ such that $\pi^*(\phi_x)=\psi_z \circ \sigma$.
\item Given $z \in V(\Omega)$ there exists $x \in B(\Omega)$ and $\sigma \in \Gamma_{\ulambda}(\Omega)$ such that $\psi_z=\pi^*(\phi_x) \circ \sigma$.
\end{enumerate}
\end{lemma}

\begin{proof}
Applying Proposition~\ref{prop:typical-factor}, we obtain a commutative diagram
\begin{displaymath}
\xymatrix@C=4em{
D \times \bA^{\umu} \ar[r]^{\alpha \times \id} \ar[d]_{\theta} & C \times \bA^{\umu} \ar[d]^{\psi} \\
B \times \bA^{\ulambda} \ar[r]^{\phi} & X }
\end{displaymath}
with $\alpha$ and $\theta$ dominant. Applying Proposition~\ref{prop:Adom}, after possibly replacing $D$ with a dense open, we see that $\theta=(\beta \times \pi) \circ \sigma$ where $\beta \colon D \to B$ is a dominant morphism and $\sigma$ is a $D$-automorphism of $D \times \bA^{\umu}$. Let $U$ be an open subset of $B$ contained in $\im(\beta)$ and $V$ be an open subset of $C$ contained in $\im(\alpha)$; these exist by Chevalley's theorem (for ordinary varieties). Suppose that $x \in U(\Omega)$. Let $y \in D(\Omega)$ satisfy $\beta(y)=x$, and let $z=\alpha(y) \in C(\Omega)$. We then have $\pi^*(\phi_x) \circ \sigma_y=\psi_z$. This proves (a). The proof of (b) is similar.
\end{proof}

\begin{proposition} \label{prop:prin}
Let $X$ be an irreducible $\GL$-variety of type $\umu$ and let $\ulambda$ be a pure tuple with $\umu \subset \ulambda$.
\begin{enumerate}
\item There exists a unique irreducible component $M$ of $\cM_{\ulambda}(X)$ for which the natural map $M \times \bA^{\ulambda} \to X$ is dominant.
\item Suppose that $C \times \bA^{\ulambda} \to X$ is dominant with $C$ irreducible. Then the map $C \to \cM_{\ulambda}(X)$ factors through $M$, and the induced map $\Gamma_{\ulambda} \times C \to M$ is dominant.
\end{enumerate}
\end{proposition}

\begin{proof}
Fix a typical morphism $\phi \colon B \times \bA^{\umu} \to X$, and let $\pi \colon \bA^{\ulambda} \to \bA^{\umu}$ be the projection map. Composing $\phi$ with $\pi$ yields a dominant map $B \times \bA^{\ulambda} \to X$. It follows that the universal map $\cM_{\ulambda}(X) \times \bA^{\ulambda} \to X$ is dominant, and so there is some component $M$ of $\cM_{\ulambda}(X)$ such that $M \times \bA^{\ulambda} \to X$ is dominant. This proves the existence statement in (a).

Let $\alpha \colon B \to \cM_{\ulambda}(X)$ be the map $b \mapsto \pi^*(\phi_b)$, and let $M$ be any irreducible component of $\cM_{\ulambda}(X)$ such that the map $\psi \colon M \times \bA^{\ulambda} \to X$ is dominant. We claim that $M$ is the closure of $\Gamma_{\ulambda} \im(\alpha)$, which will establish the uniqueness in (a). Let $U \subset B$ and $V \subset M$ be as in Lemma~\ref{lem:prin}. Let $\Omega$ be an algebraically closed field containing $K$ and let $\Gamma=\Gamma_{\ulambda}(\Omega)$. The lemma shows that
\begin{displaymath}
\alpha(U(\Omega)) \subset \Gamma \cdot M(\Omega), \qquad V(\Omega) \subset \Gamma \cdot \alpha(U(\Omega))
\end{displaymath}
as subsets of $\cM_{\ulambda}(X)(\Omega)$. Since $M$ is $\Gamma_{\ulambda}$-stable, we see that $\alpha(U) \subset M$, and thus $\alpha(B) \subset M$. Since $V$ is dense in $M$, we see that $\Gamma_{\ulambda} \cdot \alpha(B)$ is dense in $M$. This verifies the claim, and completes the proof of (a).

We now prove (b). Let $\theta \colon C \times \bA^{\ulambda} \to X$ be dominant with $C$ irreducible, and let $\beta \colon C \to \cM_{\ulambda}(X)$ be the induced map. Let $U \subset B$ and $V \subset C$ be as in Lemma~\ref{lem:prin}. Then, as above, we find
\begin{displaymath}
\alpha(U(\Omega)) \subset \Gamma \cdot \beta(C(\Omega)), \qquad V(\Omega) \subset \Gamma \cdot \alpha(U(\Omega)).
\end{displaymath}
We thus see that $\beta(V) \subset M$, and thus $\beta(C) \subset M$. Since $\Gamma \cdot \alpha(U(\Omega))$ is dense in $M$, we see that $\Gamma_{\ulambda} \cdot \beta(V)$ is also dense in $M$, which completes the proof.
\end{proof}

\begin{definition}
Let $X$ be an irreducible $\GL$-variety and let $\ulambda$ be a pure tuple with $\type(X) \subset \ulambda$. We define the \defn{principal component} of $\cM_{\ulambda}(X)$, denoted $\cM^{\prin}_{\ulambda}(X)$, to be the irreducible component $M$ of $\cM_{\ulambda}(X)$ from Proposition~\ref{prop:prin}.
\end{definition}

In fact, in our subsequent paper, we show that $\cM_{\ulambda}(X)$ is irreducible in the situation above (see \S \ref{ss:further}).

\subsection{The space of types}

Let $\ulambda$ be a pure tuple. Recall that $\cM_{\ulambda}=\cM_{\ulambda}(X)$ is the mapping space $\uMap_K(\bA^{\ulambda}, X)$. The group $\Gamma_{\ulambda}$ acts on this space. We define $X^{\type}_{\ulambda}$ to be the set of $\Gamma_{\ulambda}$ fixed points of $\cM_{\ulambda}$, equipped with the subspace topology. We note that here we are thinking of $\cM_{\ulambda}$ as a scheme, and most points of $X^{\type}_{\ulambda}$ will not be closed points. One should think of a point of $X^{\type}_{\ulambda}$ as corresponding to an irreducible closed $\Gamma_{\ulambda}$-stable subvariety of $\cM_{\ulambda}$. We have a natural map $\cM_{\ulambda} \to X^{\type}_{\ulambda}$ taking a point to the generic point of the closure of its $\Gamma_{\ulambda}$-orbit, and this realizes $X^{\type}_{\ulambda}$ as a quotient space of $\cM_{\ulambda}$; the proof is similar to that of Proposition~\ref{prop:XGL-Xorb}.

\begin{proposition}
The space $X^{\type}_{\ulambda}$ is a noetherian spectral space.
\end{proposition}

\begin{proof}
Sobriety follows from the same argument as in the first paragraph of the proof of Proposition~\ref{prop:Xorb-spectral}. Since $X^{\type}_{\ulambda}$ is a subspace of the noetherian space $\cM_{\ulambda}$, it is also noetherian, and therefore spectral.
\end{proof}

Now suppose $\ulambda \subset \umu$. Choose a projection map $\pi \colon \bA^{\umu} \to \bA^{\ulambda}$, and let $\pi^* \colon \cM_{\ulambda} \to \cM_{\umu}$ be the induced map. Let $x \in X^{\type}_{\ulambda}$, and let $Z \subset \cM_{\ulambda}$ be the corresponding $\Gamma_{\ulambda}$-stable closed subset, i.e., the Zariski closure of $\{x\} \subset \cM_{\ulambda}$. Then the Zariski closure $W$ of $\Gamma_{\umu} \pi^*(Z)$ is a $\Gamma_{\umu}$-stable irreducible closed subset of $\cM_{\umu}$, and thus defines a point of $X^{\type}_{\umu}$. We have thus defined a function $X^{\type}_{\ulambda} \to X^{\type}_{\umu}$. This construction is independent of the choice of $\pi$, since any two projection maps differ by an element of $\Gamma_{\umu}$.

\begin{lemma} \label{lem:pistar}
Let $Z$ and $W$ be as above. Then $Z=(\pi^*)^{-1}(W)$.
\end{lemma}

\begin{proof}
Let $Z'=(\pi^*)^{-1}(W)$; we show $Z=Z'$. The inclusion $Z \subset Z'$ is clear. Let $i \colon \bA^{\ulambda} \to \bA^{\umu}$ be the inclusion, so that $\pi \circ i = \id$. We thus see that $Z'=i^*(\pi^*(Z')) \subset i^*(W)$. Now, $i^*(W)$ is the closure of the set $i^*(\Gamma_{\umu}\pi^*(Z))$. An element of this set has the form $\phi \circ \pi \circ \gamma \circ i$, where $\phi \colon \bA^{\ulambda} \to X$ belongs to $Z$ and $\gamma \in \Gamma_{\umu}$. Note that $\pi \circ \gamma \circ i$ is a self-map of $\Gamma_{\ulambda}$. We thus see that $i^*(W)$ is contained in the closure of the set $\Gamma_{\ulambda}^+ Z$. But $Z$ is $\Gamma_{\ulambda}^+$-stable since it is closed and $\Gamma_{\lambda}$-stable and $\Gamma_{\lambda}$ is dense in $\Gamma_{\lambda}^+$. We thus see that $i^*(W) \subset Z$, which shows that $Z' \subset Z$.
\end{proof}

\begin{proposition} \label{prop:type-homeo}
The map $X^{\type}_{\ulambda} \to X^{\type}_{\umu}$ is a homeomorphism onto its image.
\end{proposition}

\begin{proof}
Call the map in question $f$. Let $x,y \in X^{\type}_{\ulambda}$. Let $Z_x \subset \cM_{\ulambda}$ be the Zariski closure of $\{x\}$, and let $W_x \subset \cM_{\umu}$ be the Zariski closure of $\Gamma_{\umu} \pi^*(Z_x)$, so that $f(x)$ is the generic point of $W_x$. Analogously define $Z_y$ and $W_y$. Note that $y$ is a specialization of $x$ if and only if $Z_y \subset Z_x$, and $f(y)$ is a specialization of $f(x)$ if and only if $W_y \subset W_x$. Of course, $Z_y \subset Z_x$ implies $W_y \subset W_x$. The reverse implication holds by Lemma~\ref{lem:pistar}. We thus see that $y$ is a specialization of $x$ if and only if $f(y)$ is a specialization of $f(x)$. Since the spaces in question are noetherian spectral spaces, the result follows \stacks{09XU}.
\end{proof}

\subsection{The map \texorpdfstring{$\rho$}{rho}} \label{ss:rho}

Let $x$ be a point of $X^{\type}_{\ulambda}$ and let $B \subset \cM_{\ulambda}$ be its Zariski closure. We have a natural map $B \times \bA^{\ulambda} \to X$. We define $\rho_{\ulambda}(x) \in X^{\orb}$ to be the image of the generic point of $B \times \bA^{\ulambda}$. We note that $\type(\rho_{\ulambda}(x)) \subset \ulambda$ by Proposition~\ref{prop:typical-bounds}, and so $\rho_{\ulambda}(x) \in X^{\orb}_{\ulambda}$. The following is result is one of the main points of type theory:

\begin{theorem} \label{thm:rholambda}
The function $\rho_{\ulambda} \colon X^{\type}_{\ulambda} \to X^{\orb}_{\ulambda}$ is a continuous bijection.
\end{theorem}

\begin{proof}
Let $x$ be a $\GL$-generic $K$-point of $\bA^{\ulambda}$. Consider the maps
\begin{displaymath}
\xymatrix@C=3em{
X^{\type}_{\ulambda} \ar[r] &
\cM_{\ulambda} \ar[r]^-{\id \times x} \ar[r] &
\cM_{\ulambda} \times \bA^{\ulambda} \ar[r] &
X \ar[r] & X^{\orb}. }
\end{displaymath}
The first map is the inclusion, the second is as labeled, the third is the evaluation map, and the last is the quotient map. The composition of these maps is $\rho_{\ulambda}$. Since each map is continuous, so is $\rho_{\ulambda}$.

We now show that $\rho_{\ulambda}$ is injective. Let $x \in X^{\type}_{\ulambda}$, let $M \subset \cM_{\ulambda}$ be its Zariski closure, and let $Z$ be the Zariski closure of $\rho_{\ulambda}(x)$. We thus have a dominant map $\phi \colon M \times \bA^{\ulambda} \to Z$. We thus see that $M$ is contained in $\cM_{\ulambda}(Z) \subset \cM_{\ulambda}(X)$. Since $M$ is $\Gamma_{\ulambda}$ stable, it follows from Proposition~\ref{prop:prin} that $M$ must be the principal component of $\cM_{\ulambda}(Z)$. Thus $x$ is uniquely determined from $\rho_{\ulambda}(x)$.

Finally, we show $\rho_{\ulambda}$ is surjective. Thus let
$y \in X^{\orb}_{\ulambda}$ be given. Let $W \subset X$ be
the irreducible closed $\GL$-subvariety with generic point
$y$. Since $\type(y) \subset \ulambda$, we can find a
dominant morphism $\psi \colon B \times \bA^{\ulambda} \to
W$ with $B$ irreducible. (If $\type(y)$ is smaller,
inclusionwise, than $\ulambda$, then $\psi$ will factor through a projection map.) Let $C$ be the closure of the $\Gamma_{\ulambda}$-orbit of the image of the natural map $B \to \cM_{\ulambda}$. Then the natural map $C \times \bA^{\ulambda} \to X$ also maps dominantly to $W$. It follows that if $x \in X^{\type}_{\ulambda}$ is the generic point of $C$ then $\rho_{\ulambda}(x)=y$. This completes the proof.
\end{proof}

\begin{proposition}
Let $\ulambda \subset \umu$ be pure tuples. Then the restriction of $\rho_{\umu}$ to $X^{\type}_{\ulambda}$ agrees with $\rho_{\ulambda}$.
\end{proposition}

\begin{proof}
Let $x \in X^{\type}_{\ulambda}$ and let $Z \subset \cM_{\ulambda}$ be its Zariski closure. Let $\pi \colon \bA^{\umu} \to \bA^{\ulambda}$ be a projection map, let $W_0 \subset \cM_{\umu}$ be the set $\Gamma_{\umu} \cdot \pi^*(Z)$, and let $W$ be the Zariski closure of $W_0$. The maps $Z \times \bA^{\ulambda} \to X$ and $W_0 \times \bA^{\umu} \to X$ have the same image, since applying $\pi^*$ and automorphisms of $\bA^{\umu}$ does not change the image. It follows that the maps $Z \times \bA^{\ulambda} \to X$ and $W \times \bA^{\umu} \to X$ have the same image closures. Thus $\rho_{\ulambda}(x)=\rho_{\umu}(x)$.
\end{proof}

\subsection{Putting the pieces together} \label{ss:Xtype}

As we have seen, the spaces $X^{\type}_{\ulambda}$ form a directed system. We define $X^{\type}$ to be their direct limit. This spaces carries a direct limit topology, and there is a natural bijection
\begin{displaymath}
\rho \colon X^{\type} \to X^{\orb}
\end{displaymath}
that is continuous with respect to the direct limit topology on $X^{\type}$. However, $\rho$ is typically not a homeomorphism when $X^{\type}$ is given the direct limit topology; this is due to the fact that $X^{\orb}$ is typically not the direct limit of its subspaces $X^{\orb}_{\ulambda}$. There is a refined topology one can give $X^{\type}$ for which we expect $\rho$ to be a homeomorphism. We will return to this topic in the future.

\subsection{An example} \label{ss:type-example}

Let $\kappa=(2)$ and consider $X=\bA^{\kappa}$. One can think of $X$ as the space of symmetric bilinear forms on $\bV$. For $r \in \bN$, let $Z_r$ be the locus of forms of rank $\le r$. These sets, together with $\emptyset$ and $X$, account for all of the closed $\GL$-subvarieties. We thus see that $X^{\orb}$ is identified with $\bN \cup \{\infty\}$. The topology on $X^{\orb}$ is the order topology: the closed sets are the finite intervals $\{0, \ldots, r\}$, together with $\emptyset$ and $X^{\orb}$.

Let $\ulambda=[(1)^n,(2)]$. Then $X^{\orb}_{\ulambda}=\{0, \ldots, n, \infty\}$ with the subspace topology. It is not difficult to see in this case that $X^{\type}_{\ulambda} \to X^{\orb}_{\ulambda}$ is a homeomorphism. The space $X^{\type}$ can be identified with $\bN \cup \{\infty\}$ equipped with the direct limit topology with respect to the $X^{\type}_{\ulambda}$. Since $\bN \cap X^{\type}_{\ulambda}$ is closed for all $\ulambda$, it follows that $\bN$ is closed in $X^{\type}$. However, it is not closed in $X^{\orb}$. We thus see that $\rho$ is not a homeomorphism in this case.

%Identify $X^{\type}$ with $\bN \cup \{\infty\}$ as well.
%Since $\rho$ is continuous, every closed set in $X^{\orb}$
%is closed in $X^{\type}$. Furthermore, by
%Proposition~\ref{prop:Precomposition}, if an element of $\bN
%\cup \{\infty\}$ is in a closed set $Z$, then so are all smaller
%elements. In addition to the closed sets in $X^{\orb}$ there
%is one more set with this property, namely, $\bN$.
%
%It turns out that $\bN$ is, indeed, closed in $X^\type$. To see this, note that any map $\bA^{\ulambda} \to X$ is either surjective, or lands in $Z_r$ where $r$ is the number of 1's in $\ulambda$. It follows that the inverse image of $\bN$ under the map $\cM_{\ulambda} \to X^{\type}$ is the same as the inverse image of $\{0,\ldots,r\}$, and therefore closed. This proves that $\bN$ is closed in $X^{\type}$. Note that the set $\{\infty\}$ is open in $X^{\type}$, which is perhaps in line with the intuition that the locus of non-degenerate forms should be open.

\section{Examples and applications} \label{s:examples}

\subsection{Systems of variables}\label{ss:system}

Let $\ulambda$ be a pure tuple. We write $\bA^{\ulambda}(K)$ for the set of $K$-points of $\bA^{\ulambda}$. We say that $x \in \bA^{\ulambda}(K)$ is \defn{degenerate} if it is not $\GL$-generic. We have the following characterization of these points:

\begin{proposition} \label{prop:degen-char}
Let $x \in \bA^{\ulambda}(K)$. Then $x$ is degenerate if and only if $x$ has the form $\phi(y)$ where $\phi \colon \bA^{\umu} \to \bA^{\ulambda}$ is a morphism where $\umu$ is a pure tuple with $\magn(\umu)<\magn(\lambda)$ and $y \in \bA^{\umu}(K)$.
\end{proposition}

\begin{proof}
Suppose $x$ is degenerate. Then $\ol{O}_x$ is a proper closed $\GL$-subvariety of $\bA^{\ulambda}$. Let $\phi \colon B \times \bA^{\umu} \to \ol{O}_x$ be the morphism produced by Theorem~\ref{thm:uni}(d). Thus $\magn(\umu)<\magn(\ulambda)$, the image of $\phi$ contains a non-empty $\GL$-open subset $U$ of $\ol{O}_x$, and every $K$-point of $U$ lifts to a $K$-point of $\bA^{\umu}$. Since $U$ contains $O_x$ (Proposition~\ref{prop:O-as-int-opens}), it follows that $x=\phi(b,y)$ for some $b \in B(K)$ and $y \in \bA^{\umu}(K)$. Letting $\phi_b$ be the restriction of $\phi$ to $\{b\} \times \bA^{\umu}$, we have $x=\phi_b(y)$, and so $x$ has the stated form.

Now suppose that $x=\phi(y)$ for some $\phi \colon \bA^{\umu} \to \bA^{\ulambda}$ with $\magn(\umu)<\magn(\ulambda)$. Since $\phi$ is not dominant (by, e.g., Proposition~\ref{prop:Adom}) it follows that $\ol{O}_x$ is a proper subset of $\bA^{\ulambda}$, and so $x$ is degenerate.
\end{proof}

\begin{proposition}
Let $\lambda$ be a non-empty partition. Then the set of degenerate points in $\bA^{\lambda}(K)$ forms a $K$-vector subspace of $\bA^{\lambda}(K)$.
\end{proposition}

\begin{proof}
The subset of degenerate points clearly contains $0$ and is closed under scaling. We now show that it is closed under addition. Let $x,y \in \bA^{\lambda}$ be degenerate. Applying the previous proposition, write $x=\phi(x')$ where $\phi \colon \bA^{\umu} \to \bA^{\lambda}$ is a morphism with $\magn(\umu)<\magn(\lambda)$ and $x' \in \bA^{\umu}(K)$; similarly, write $y=\psi(y')$ where $\psi \colon \bA^{\unu} \to \bA^{\lambda}$ is a morphism with $\magn(\unu)<\magn(\lambda)$ and $y' \in \bA^{\unu}(K)$. Then $x+y$ is the image of $(x',y')$ under the map
\begin{displaymath}
\bA^{\umu \cup \unu} \to \bA^{\lambda}, \qquad (u,v) \mapsto \phi(u)+\psi(v).
\end{displaymath}
Since $\magn(\umu \cup \unu)<\magn(\lambda)$, the previous proposition shows that $x+y$ is degenerate.
\end{proof}

The above proposition is not true for tuples. Indeed, if $(x,y)$ is a non-degenerate point of $\bA^{[\lambda,\mu]}=\bA^{\lambda} \times \bA^{\mu}$ then $(x,y)=(x,0)+(0,y)$ and both $(x,0)$ and $(0,y)$ are degenerate. We say that a set of $K$-points of $\bA^{\lambda}$ is \defn{collectively non-degenerate} if they are linearly independent modulo the subspace of degenerate points, and \defn{collectively degenerate} otherwise.

\begin{proposition} \label{prop:nondegen}
Let $\lambda_1, \ldots, \lambda_n$ be distinct non-empty partitions, let $x_{i,1}, \ldots, x_{i,m(i)}$ be $K$-points of $\bA^{\lambda_i}$, and let $x=(x_{i,j})$, which we regard as a $K$-point $\prod_{i=1}^n(\bA^{\lambda_i})^{m(i)}$. Then $x$ is non-degenerate if and only if for each $i$ the points $x_{i,1},\ldots,x_{i,m(i)}$ are collectively non-degenerate.
\end{proposition}

\begin{proof}
Suppose that $x$ is degenerate, and write $x=\phi(y)$ where $\phi \colon \bA^{\umu} \to \bA^{\ulambda}$ is a morphism with $\magn(\umu)<\magn(\ulambda)$ and $\ulambda=[\lambda_1^{m(1)}, \ldots, \lambda_n^{m(n)}]$. Now, if $\kappa$ is a partition then any morphism $\psi \colon \bA^{\umu} \to \bA^{\kappa}$ has the form $\psi=\psi_1+\psi_2$, where $\psi_1$ is a linear map defined on factors of $\bA^{\umu}$ with $\mu_i=\kappa$ and $\psi_2$ depends only on the factors of $\bA^{\umu}$ with $\# \mu_i<\# \kappa$. In particular, we see that $\psi(y)$ is equal to a linear combination of the $y_i$'s with $\mu_i=\kappa$ modulo degenerate elements. Now, since $\magn(\umu)<\magn(\ulambda)$ there must be some $\lambda_i$ that appears $r<m(i)$ times in $\umu$. We thus see that $x_{i,1},\ldots,x_{i,m(i)}$ is a linear combination of $r$ components of $y$ modulo degenerate elements, and so collectively degenerate.

Now suppose that $x_{i,1}, \ldots, x_{i,m(i)}$ is collectively degenerate for some $i$. Applying an element of $\GL_{m(i)}$, we can assume that $x_{i,m(i)}$ is itself degenerate. Since the projection $x_{i,m(i)}$ of $x$ to $\bA^{\lambda_i}$ is not $\GL$-generic, it follows that $x$ is not $\GL$-generic.
\end{proof}

We now come to the central concept of \S \ref{ss:system}.

\begin{definition}
For a non-empty partition $\lambda$, a \defn{system of $\lambda$-variables} is a subset $\Xi_{\lambda}$ of $\bA^{\lambda}(K)$ that forms a basis modulo the subspace of degenerate elements. A \defn{system of variables} is a collection $\Xi=\{\Xi_{\lambda}\}_{\lambda}$ where $\Xi_{\lambda}$ is a system of $\lambda$-variables for each non-empty partition $\lambda$.
\end{definition}

We fix a system of variables $\Xi$ in what follows. Given a $K$-point $x$ of $\bA^{\ul{\kappa}}$, an expression of the form $x=\phi(\xi_1, \ldots, \xi_n)$ means that $\phi$ is a morphism $\bA^{\ulambda} \to \bA^{\ul{\kappa}}$ with $\xi_i \in \Xi_{\lambda_i}$. The following theorem, which is our main result on systems of variables, essentially says that every element of $\bA^{\ul{\kappa}}$ can be expressed uniquely in terms of $\Xi$.

\begin{theorem} \label{thm:vars}
Let $x$ be a $K$-point of $\bA^{\ul{\kappa}}$. Then we have $x=\phi(\xi_1, \ldots, \xi_n)$ for distinct $\xi_1, \ldots, \xi_n$. Fix such an expression with $n$ minimal, and consider a second such expression $x=\psi(\xi'_1, \ldots, \xi'_m)$ with $\xi'_1, \ldots, \xi'_m$ distinct. Then, after applying a permutation, we have $\xi'_i=\xi_i$ for $1 \le i \le n$ and $\psi(y_1, \ldots, y_m)=\phi(y_1, \ldots, y_n)$ identically.
\end{theorem}

\begin{proof}
For the first statement, it suffices to treat the case where $\ul{\kappa}$ is a single partition $\kappa$. Write $x=\sum_{i=1}^r c_i \xi_i + y$ where the $c_i$'s are scalars, the $\xi_i$ belong to $\Xi_{\kappa}$, and $y$ is degenerate; we can do this since $\Xi_{\kappa}$ is a basis modulo degenerate elements. Since $y$ is degenerate, we can write $y=\psi(z)$ for some morphism $\psi \colon \bA^{\umu} \to \bA^{\kappa}$ with $\magn(\umu)<\magn(\kappa)$. By induction, we have an expression for $z$ in terms of $\xi$'s. Combining all of this, we obtain the required expression for $x$.

Now, fix an expression $x=\phi(\xi_1, \ldots, \xi_n)$ with $\xi_1, \ldots, \xi_n$ with $n$ minimal; this implies that the $\xi_i$ are distinct. Consider a second expression $x=\psi(\xi'_1, \ldots, \xi'_m)$ with $\xi'_1, \ldots, \xi'_m$ distinct. Apply a permutation so that $\xi_i=\xi'_i$ for $1 \le i \le r$ with $r$ maximal; thus the elements $\xi_1, \ldots, \xi_n, \xi'_{r+1}, \ldots, \xi'_m$ is distinct. Let $\xi_i \in \Xi_{\lambda_i}$ and $\xi'_i \in \Xi_{\lambda'_i}$, and put $\ulambda=[\lambda_1, \ldots, \lambda_n, \lambda'_{r+1}, \ldots, \lambda'_m]$. Let $\tilde{\phi} \colon \bA^{\ulambda} \to \bA^{\ul{\kappa}}$ be the composition of $\phi$ with the appropriate projection, and similarly define $\tilde{\psi}$. Then both $\tilde{\phi}$ and $\tilde{\psi}$ map $(\xi_1, \ldots, \xi_n, \xi'_{r+1}, \ldots, \xi'_m)$ to $x$. Since this point is $\GL$-generic by Proposition~\ref{prop:nondegen}, it follows that $\tilde{\phi}=\tilde{\psi}$. Evaluating both sides on the point $(\xi_1, \ldots, \xi_r, 0, \ldots, 0, \xi'_{r+1}, \dots, \xi'_m)$, we find $\phi(\xi_1, \ldots, \xi_r, 0, \ldots, 0)=x$. By the minimality of $n$, we conclude that $r=n$. It now follows that $\phi(y_1, \ldots, y_n)=\psi(y_1, \ldots, y_m)$ for all $y_1, \ldots, y_m$, which completes the proof.
\end{proof}

\subsection{Big polynomial rings} \label{ss:bigpoly}

Let $\cR$ be the inverse limit of the standard-graded polynomial rings $K[x_1, \ldots, x_n]$ in the category of graded rings. Thus $\cR$ is a graded ring, and a degree $d$ element of $\cR$ is a formal $K$-linear combination of degree $d$ monomials in the variables $x_1, x_2, \ldots$; in particular, $\cR_d$ is identified with $\bA^{(d)}(K)$. The following theorem was established in \cite{ess} (and previously in \cite{am}), and used to give two new proofs of Stillman's conjecture:

\begin{theorem} \label{thm:poly}
The ring $\cR$ is (isomorphic to) a polynomial ring.
\end{theorem}

In fact, this result is essentially a special case of the material in the previous section, as we now show:

\begin{proof}[Proof of Theorem~\ref{thm:poly}]
It follows from the Littlewood--Richardson rule that if a symmetric power $\bV_{(d)}$ appears in a tensor product $\bV_{\lambda} \otimes \bV_{\mu}$ then $\lambda$ and $\mu$ each have a single row, i.e., $\bV_{\lambda}$ and $\bV_{\mu}$ are themselves symmetric powers; moreover, in this case, $\bV_{(d)}$ occurs with multiplicity one, and so, up to scaling, the only map $\bV_{\lambda} \otimes \bV_{\mu} \to \bV_{(d)}$ is the multiplication map. It follows that any morphism $\bA^{\ulambda} \to \bA^{(d)}$ only depends on the coordinates of $\bA^{\ulambda}$ that are symmetric powers, and is simply given by some polynomial in those coordinates.

Now, fix a system of variables $\Xi$. By Theorem~\ref{thm:vars} and the above discussion, we see that every element of $\cR_d$ can be expressed as a polynomial in the elements $S=\bigcup_{n \ge 1} \Xi_{(n)}$, and that this expression is unique up to permutations. Thus $\cR$ is a polynomial ring on the generators $S$.
\end{proof}

Now, suppose that $K$ is the field $\bC\lpp t \rpp$ of Laurent series. Let $\cR^{\flat}$ be the subring of $\cR$ consisting of elements $f$ with bounded denominators, i.e., $f$ has coefficients in $t^{-n} \bC\lbb t \rbb$ for some $n$. In \cite{ess}, it was also shown that $\cR^{\flat}$ is also a polynomial ring over $K$. The following complementary (but much deeper) result was established in \cite{relbig}:

\begin{theorem} \label{thm:relpoly}
The ring $\cR$ is a polynomial algebra over $\cR^{\flat}$.
\end{theorem}

We now show how this more difficult theorem can be deduced from our theory in a fairly simple manner. We say that a $K$-point of $\bA^{\ulambda}$ is \defn{bounded} if there is some $n$ such that its coordinates all belong to $t^{-n} \bC\lbb t \rbb$. Thus $\cR^{\flat}_d$ can be identified with the set of bounded $K$-points of $\bA^{(d)}$. More generally, suppose that $x$ is a $K$-point of a quasi-affine $\GL$-variety $X$. Let $R=\Gamma(X, \cO_X)$ and let $\phi \colon R \to K$ be the homomorphism corresponding to $x$. We say that $x$ is \defn{bounded} if for every finite length subrepresentation $V$ of $R$ there is some $n$ such that $\phi(V) \subset t^{-n} \bC\lbb t \rbb$; in fact, it suffices to verify this for any subrepresentation $V$ that generates $R$ as an algebra. The key result we need is the following:

\begin{proposition} \label{prop:bounded-lift}
Let $\phi \colon X \to Y$ be a morphism of $\GL$-varieties and let $x$ be a bounded $K$-point of $X$ belonging to $\im(\phi)$. Then there is a bounded $K(t^{1/n})$-point $y$ of $Y$, for some $n$, such that $x=\phi(y)$.
\end{proposition}

\begin{proof}
This is clear for elementary morphisms, and thus the result follows from the decomposition theorem. (We note that there is a notion of boundedness for $K$-points in a $G(n)$-variety, and a point is $G(n)$-bounded if and only if it is $G(m)$-bounded for $m \ge n$. This theory will be better developed in our subsequent paper.)
\end{proof}

Recall from Example~\ref{ex:str-var} the notion of strength. As explained in the introduction to \cite{relbig}, Theorem~\ref{thm:relpoly} follows easily from the following:

\begin{proposition} \label{prop:str}
Suppose that $f$ is an element of $\cR^{\flat}$ that has finite strength in $\cR$. Then $f$ has finite strength in $\cR^{\flat}$.
\end{proposition}

\begin{proof}
Say $f$ has degree $d$, and we have $f=\sum_{i=1}^r g_i h_i$ where $g_i,h_i \in \cR$ are homogeneous of degrees $<d$. Put $\deg(g_i)=e_i$ $\deg(h_i)=e'_i$ and let $\ulambda=[(e_1),(e'_1),\ldots,(e_r),(e'_r)]$. Consider the map $\pi \colon \bA^{\ulambda} \to \bA^{(d)}$ given by $\pi(a_1,b_1,\ldots,a_r,b_r)=\sum_{i=1}^r a_i b_i$. We have $f=\pi(g_1,h_1,\ldots,g_r,h_r)$, and so $f$ is in the image of $\pi$. By Proposition~\ref{prop:bounded-lift}, $f$ is the image of a bounded $K(t^{1/n})$-point for some $n \ge 1$. This means we can write $f=\sum_{i=1}^r g'_i h'_i$ where $g'_i,h'_i \in \cR^{\flat}_{K(t^{1/n})}$ are homogeneous of degrees $<d$. Write $g'_i=\sum_{k=0}^{n-1} g'_{i,k} t^{k/n}$ with $g'_{i,k} \in \cR^{\flat}$, and similarly express $h'_i$. Multiplying $g'_i h'_i$ out and taking the piece that belongs to $\cR^{\flat}$, we see that $f$ has strength $\le rn$ in $\cR^{\flat}$.
\end{proof}

\subsection{An extended example} \label{ss:ex-type}

Suppose that $K$ is algebraically closed field and let $X \subset \bA^{[(1),(1)]}$ be the rank $\le 1$ locus. Precisely, we identify $\bA^{[(1),(1)]}$ with the spectrum of $S=K[x_i,y_i]_{i \ge 1}$, with each set of variables generating a copy of the standard representation of $\GL$, and $X$ is the spectrum of $R=S/(x_i y_j-x_j y_i)$. We look at several of our definitions and results in this case.
\begin{itemize}
\item For each $[a:b] \in \bP^1(K)$ we have the line $L_{a,b} \subset X$ where $bx_i=ay_i$. This is a closed $\GL$-subset. Besides these, there are only two other closed $\GL$-subsets, namely $\{0\}$ and $X$. We thus see that $X^{\orb}$ can be identified with $\bP^1(K)$ plus two additional points. In fact, the generalized orbit of the generic point of $X$ corresponds to the generic point of $\bP^1_K$. Thus $X^{\orb}$ is the topological space underlying the scheme $\bP^1_K$ plus one extra point $\ast$ corresponding to $\{0\}$. The point $\ast$ is closed, and the closure of any point contains $\ast$; this completely describes the topology on $X^{\orb}$.
\item Let $t=x_i/y_i \in \Frac(R)$; note that this is independent of $i$. Then $K(X)^{\GL}=K(t)$. Thus $X$ gives an example of a pure $\GL$-variety with a non-trivial invariant function field.
\item We now consider the shift theorem for $X$. Identify $\Sh_1(S)$ with $S[\xi,\eta]$, where $\xi$ the ``shifted off'' $x$ variable, and $\eta$ is the ``shifted off'' $y$ variable. Then $\Sh_1(R)$ is the quotient of $S[\xi,\eta]$ by the relations
\begin{displaymath}
x_iy_j-x_jy_i, \qquad
\xi y_i - \eta x_i
\end{displaymath}
for all $i$ and $j$. We thus see that in $\Sh_1(R)[1/\eta]$ we have $x_i=(\xi/\eta) y_i$ for all $i$. We thus find that $\Sh_1(R)[1/\eta]=K[\xi,\eta^{\pm 1},y_1,y_2,\ldots]$. Hence $\Sh_1(X) \cong B \times \bA^{(1)}$ where $B=\bA^1 \times (\bA^1 \setminus \{0\})$. From this, we see that the invariant function field of any $\Sh_n(X)$ is a rational function field over $K$.
\item Besides 0, every point in $X$ has type $[(1)]$.
\item A typical morphism for $X$ is given by the map $\bA^1 \times \bA^{(1)} \to X$ corresponding to the ring homomorphism $R \to K[t,y_1, y_2, \ldots]$ mapping $y_i$ to $y_i$ and $x_i$ to $t y_i$.
\item Consider the mapping space $M=\Map_K(\bA^{(1)}_K, X)$. Let $A$ be a $K$-algebra. An $A$-point of $M$ corresponds to a $\GL$-algebra homomorphism $f \colon R \to A[z_1, z_2, \ldots]$. Under such a homomorphism $x_1$ must map to a $G(1)$-invariant element, and thus a scalar multiple of $z_1$; similarly, $y_1$ must also map to a scalar multiple of $z_1$. Say $f(x_1)=az_1$ and $f(y_1)=bz_1$. Then $f(x_i)=az_i$ and $f(y_i)=bz_i$, and so $f$ is completely determined. Conversely, for any $a$ and $b$, we see that these formulas define a $\GL$-algebra homomorphism, i.e., the relations in $R$ hold. We thus see that $M=\Spec(K[a,b])=\bA^2$.
%\item The only way for a map from $\bA^{(1)}_E$ to be non-typical is if it is constant, i.e., factors through $\Spec(E)$. We thus see that the typical locus $M^0$ in $M$ is $\bA_K^2 \setminus \{0\}$. The closed points of $M^0$, as well as the generic points of lines through the origin of $\bA^2_K$, correspond to typical morphisms for the points $\bP^1(K) \subset X^{\orb}$, while the other points (i.e., those of $M^0$ that map to the generic point of $\bP^1_K$) correspond to typical morphisms for the generic point of $X$.
\end{itemize}

\subsection{Examples of invariant function fields} \label{ss:ex-field}

Let $\ulambda$ be the tuple consisting of $n+1$ copies of $(1)$, and let $X$ be the rank $\le 1$ locus in $\bA^{\ulambda}$. We can think of $X \setminus \{0\}$ as the space of rank~1 maps $\bV \to K^{n+1}$. We thus have a natural map $\pi \colon X \setminus \{0\} \to \bP^n$ by associating to a rank~1 map its image. The map $\pi$ is clearly $\GL$-equivariant, using the trivial action on the target. Let $W$ be an irreducible subvariety of $\bP^n$, and let $Z \subset X$ be the closure of $\pi^{-1}(W)$, which is a $\GL$-variety. The map $\pi$ induces a field homomorphism $\pi^* \colon K(W) \to K(Z)^{\GL}$, which one can show is an isomorphism. We thus see that any finitely generated extension of $K$ can be realized as the invariant function field of a $\GL$-variety that is pure over $K$.

\subsection{Morphisms of affine spaces need not have closed image} \label{ss:ex-notclosed}

Suppose that $\phi \colon \bA^{\ulambda} \to \bA^{\umu}$ is a morphism of $\GL$-varieties with $\ulambda$ and $\umu$ pure. From Examples~\ref{ex:rank-var} and~\ref{ex:str-var}, we know that $\phi$ need not have closed image. We now give a simple example that directly exhibits this phenomenon.

Consider the $\GL$-morphism
\begin{eqnarray*}
\varphi\colon\bA^{[(2),(2),(2)]}&\to&\bA^{(4)}\\
(f,g,h)&\mapsto&fg-h^2
\end{eqnarray*}
and the closure $X$ of its image. We show that the image of $\varphi$ is not closed by finding a $\GL$-morphism $\psi$ to $X$ such that $\psi\neq \varphi\circ\gamma$ for every $\GL$-morphism $\gamma$.

Note that for all closed points $(x,y,f,g,h)\in\bA^{[(1),(1),(2),(2),(2)]}$ and $t\in K\setminus\{0\}$, we have
\begin{eqnarray*}
t^{-1}\varphi\left(y^2+tf,x^2+tg,xy-\mbox{$\frac{1}{2}$}th\right)&=&t^{-1}\left((y^2+tf)(x^2+tg)-(xy-\mbox{$\frac{1}{2}$}th)^2\right)\\
&=&x^2f+y^2g+xyh+t(\dots)\in X
\end{eqnarray*}
Hence the $\GL$-morphism
\begin{eqnarray*}
\psi\colon\bA^{[(1),(1),(2),(2),(2)]}&\to&\bA^{(4)}\\
(x,y,f,g,h)&\mapsto&x^2f+y^2g+xyh
\end{eqnarray*}
maps into $X$. If $\im(\psi)\subseteq\im(\varphi)$, then $\psi=\varphi\circ\gamma$ for some $\GL$-morphism
\[
\gamma\colon\bA^{[(1),(1),(2),(2),(2)]}\to\bA^{[(2),(2),(2)]}
\]
by Propositions~\ref{prop:generic-exists} and~\ref{prop:lift}. Such a $\GL$-morphism $\gamma$ must be of the form
$$
\gamma(x,y,f,g,h)=\left(\begin{array}{c}c_{11}x^2+c_{12}xy+c_{13}y^2+c_{14}f+c_{15}g+c_{16}h\\c_{21}x^2+c_{22}xy+c_{23}y^2+c_{24}f+c_{25}g+c_{26}h\\c_{31}x^2+c_{32}xy+c_{33}y^2+c_{34}f+c_{35}g+c_{36}h\end{array}\right)
$$
for some constants $c_{ij}\in K$. This turns the condition $\psi=\varphi\circ\gamma$ into polynomial equations in the $c_{ij}$. Now, one can check using a Gr\"obner basis calculation that this system has no solutions. Hence $\im(\psi)\not\subseteq\im(\varphi)$. So, in particular, we see that $\im(\varphi)\neq X$ is not closed.

\begin{question}
Is $\im(\varphi)\cup\im(\psi)$ closed?
\end{question}

\begin{question}
Is $\im(\psi)$ closed?
\end{question}

\subsection{The typical locus need not be open} \label{ss:notopen}

We consider the set of all $\GL$-morphisms $\bA^{[(1),(1),(2),(2),(2)]}\to\bA^{(4)}$. Such a $\GL$-morphism is typical (as a map to its image closure) precisely when it does not factor via $\bA^{\ulambda}$ for any $\ulambda\subsetneq[(1),(1),(2),(2),(2)]$. Note that for all $t\in K\setminus\{0\}$, the $\GL$-morphism
\begin{eqnarray*}
\varphi_t\colon \bA^{[(1),(1),(2),(2),(2)]}&\to&\bA^{(4)}\\
(x,y,f,g,h)&\mapsto& t^{-1}\left((x^2+tf)(y^2+tg)-(xy+th)^2\right)
\end{eqnarray*}
factors via $\bA^{[(2),(2),(2)]}$ and is hence not typical. If the set of typical $\GL$-morphisms were open, then the limit $\varphi_0=\lim_{t\to 0}\varphi_t$ would also have to be not typical. However the $\GL$-morphism
\begin{eqnarray*}
\varphi_0\colon \bA^{[(1),(1),(2),(2),(2)]}&\to&\bA^{(4)}\\
(x,y,f,g,h)&\mapsto& x^2g+y^2f-2xyh
\end{eqnarray*}
cannot factor through $\bA^{[(1),(1),(2),(2)]}$ for dimension reasons and it cannot factor through $\bA^{[(1),(2),(2),(2)]}$ since the coefficients $x^2,y^2$ of $g,f$ in $\varphi_0$ are linearly independent. So $\varphi_0$ is typical. Hence the set of typical $\GL$-morphisms $\bA^{[(1),(1),(2),(2),(2)]}\to\bA^{(4)}$ is not open.


\begin{thebibliography}{BDDE}

\bibitem[AH]{ah}
Tigran Ananyan, Melvin Hochster. Small subalgebras of polynomial rings and Stillman's conjecture. {\it J.~Amer.~Math.~Soc.} {\bf 33} (2020), 291--309.  \doi{10.1090/jams/932}

\bibitem[AM]{am}
A.~L.~Allen, S.~Moran. The inverse limit of some free
algebras. \textit{J.~Aust.~Math.~Soc.} \textbf{16} (1973), no.~2, 129--145. \doi{10.1017/S1446788700014154}

\bibitem[BBOV]{bbov}
Edoardo Ballico, Arthur Bik, Alessandro Oneto, Emanuele Ventura. The
set of forms with bounded strength is not closed. \textit{C.~R., Math.,
Acad. Sci. Paris} \textbf{360} (2022), 371--380. \doi{10.5802/crmath.302}

\bibitem[Bi]{bik}
Arthur Bik. Strength and Noetherianity for infinite tensors, Ph.D. Thesis,
University of Bern, 2020.
\url{https://biblio.unibe.ch/download/eldiss/20bik_ma.pdf}

\bibitem[BDD]{bdd}
Arthur Bik, Alessandro Danelon, Jan Draisma. Topological
Noetherianity of polynomial functors II: base rings with
Noetherian spectrum. \textit{Math.~Ann.} (2022),
\doi{10.1007/s00208-022-02386-9}

\bibitem[BDE]{bde}
Arthur Bik, Jan Draisma, Rob H.~Eggermont. Polynomials and tensors of bounded strength. \textit{Commun.~Contemp.~Math.} \textbf{21}(7) (2019), paper number 1850062. \doi{10.1142/S0219199718500621}

\bibitem[BDDE]{bdde}
Arthur Bik, Alessandro Danelon, Jan Draisma, Rob
H.~Eggermont. Universality of high-strength tensors.
\textit{Vietnam J.~Math.} \textbf{50} (2022), no.~2
(Sturmfels special issue), 557--580.
\doi{10.1007/s10013-021-00522-7}

\bibitem[BB]{bb}
Weronika {Buczy\'nska}, Jaros{\l}aw {Buczy\'nski}.
Secant varieties to high degree Veronese
reembeddings, catalecticant matrices and smoothable
Gorenstein schemes.
\textit{J.~Algebr.~Geom.},
\textbf{23} (2014), no.~1, 63--90. \arxiv{1012.3563}

\bibitem[BL]{bl}
Jaros{\l}aw Buczy\'nski, Joseph M.~Landsberg.  On the third secant
variety. \textit{J.~Alg.~Comb.} \textbf{40} (2014), no.~2,
475--502. \doi{10.1090/S1056-3911-2013-00595-0}

\bibitem[DES]{des} Harm Derksen, Rob H.~Eggermont, Andrew Snowden. Topological noetherianity for cubic polynomials. \textit{Algebra Number Theory} \textbf{11}(9) (2017), pp.\ 2197--2212. \doi{10.2140/ant.2017.11.2197}

\bibitem[Dr]{draisma}
Jan Draisma. Topological Noetherianity of polynomial functors. \textit{J.\ Amer.\ Math.\ Soc.}\ \textbf{32}(3) (2019), pp.\ 691--707. \doi{10.1090/jams/923}

\bibitem[DLL]{dll}
Jan Draisma, Micha\l{} Laso\'n, Anton Leykin. Stillman's
conjecture via generic initial ideals. \textit{Commun.
Algebra} \textbf{47} (2019), no.~6 (Lyubeznik special
issue), 2384--2395. \doi{10.1080/00927872.2019.1574806}

\bibitem[ES]{es}
Rob H. Eggermont, Andrew Snowden. Topological noetherianity
for algebraic representations of infinite rank classical
groups. \textit{Transform.\ Groups} (2021),
\doi{10.1007/s00031-021-09656-x}

\bibitem[ESS1]{genstillman}
Daniel Erman, Steven V. Sam, Andrew Snowden. Generalizations of Stillman's conjecture via twisted commutative algebras. {\it Int. Math. Res. Not. IMRN}
\textbf{2021} (2021), no.~16, Pages 12281--12304. \doi{10.1093/imrn/rnz123}

\bibitem[ESS2]{ess}
Daniel Erman, Steven V. Sam, Andrew Snowden. Big polynomial rings and Stillman's conjecture. \textit{Invent.\ Math.} \textbf{218} (2019), no.~2, 413--439. \doi{10.1007/s00222-019-00889-y}

\bibitem[GSS]{gss}
Luis D.~Garcia, Michaela Stillman, Bernd Sturmfels. Algebraic Geometry of
Bayesian Networks. \textit{J.~Symb.~Comput.} \textbf{39}
(2005), no.~3--4, 331--355. \doi{10.1016/j.jsc.2004.11.007}

\bibitem[KaZ1]{kaz1}
David Kazhdan, Tamar Ziegler. On ranks of polynomials.
\textit{Algebr. Represent. Theory} \textbf{21} (2018),
no.~5, 1017--1021. \doi{10.1007/s10468-018-9783-7}

\bibitem[KaZ2]{kaz2}
David Kazhdan, Tamar Ziegler. Properties of high rank
subvarieties of affine spaces. \textit{Geom. Funct.
Analysis} \textbf{30} (2020), 1063--1096. \doi{10.1007/s00039-020-00542-4}

\bibitem[La]{landsberg}
Joseph M.~Landsberg. \textit{Tensors: geometry and applications.} Graduate Studies in Mathematics 128, American Mathematical Society, Providence, RI, 2012. 

\bibitem[LO]{lo}
Joseph M.~Landsberg, Giorgio Ottaviani. Equations for secant varieties
of Veronese and other varieties. \textit{Ann. Mat. Pura Appl.}
\textbf{192} (2013), no.~4, 569--606. \doi{10.1007/s10231-011-0238-6}

\bibitem[LW]{lw}
Joseph M. Landsberg, Jerzy Weyman.
On the ideals and singularities of secant varieties of {S}egre
varieties. \textit{Bull.~Lond.~Math.~Soc.} \textbf{39}
(2007), no.~4, 685--697. \doi{10.1112/blms/bdm049}

\bibitem[La3]{lang}
Serge Lang. Hilbert's Nullstellensatz in infinite-dimensional space. \textit{Proc.\ Amer.\ Math.\ Soc.} \textbf{3} (1952), 407--410.

\bibitem[OR]{or}
Luke Oeding, Claudiu Raicu. Tangential varieties of
Segre--Veronese varieties. \textit{Collect.~Math.}
\textbf{65} (2014), no.~3, 303--330. \doi{10.1007/s13348-014-0111-1}

\bibitem[Na]{nagpal}
Rohit Nagpal. FI-modules and the cohomology of modular $S_n$-representations. Ph.D.\ thesis, University of Wisconsin, 2015. \arxiv{1505.04294v1}

\bibitem[Ra]{raicu}
Claudiu Raicu. Secant Varieties of Segre--Veronese Varieties.
\textit{Algebra \& Number Theory} \textbf{6} (2012), no. 8,
1817-1868. \doi{10.2140/ant.2012.6.1817}

\bibitem[Ra2]{raicu2}
Claudiu Raicu. 3x3 Minors of Catalecticants.
\textit{Math.~Res.~Lett.} \textbf{20} (2013), no. 4, 745-756.
\doi{10.4310/MRL.2013.v20.n4.a10}

\bibitem[Ro]{rosenlicht}
Maxwell Rosenlicht. Some basic theorems on algebraic groups. \textit{Amer.\ J.\ Math.} \textbf{78} (1956), no.~2, 401--443.

\bibitem[Sn]{relbig}
Andrew Snowden. Relative big polynomial rings. \textit{J.\ Commut.\ Algebra}, to appear. \arxiv{2002.09665}

%\bibitem[Sn2]{tcares}
%Andrew Snowden. Stable representation theory: beyond the classical groups. \arxiv{2109.11702}

\bibitem[Stacks]{stacks}
Stacks Project. \url{http://stacks.math.columbia.edu} (accessed June, 2020).

\end{thebibliography}
\end{document}